\documentclass[11pt,reqno]{amsart}

\pdfoutput=1

\usepackage{LPPSS1_Macros2}
\usepackage{array}
\usepackage[section]{placeins} 
\usepackage{refcount} 

\providebibmacro*{bbx:dashcheck}{}

\allowdisplaybreaks

\title[The degree-one vortex in the abelian YMH model: Spectral Theory and Numerics]{Asymptotic stability of the degree-one vortex \\ in the abelian Yang-Mills-Higgs model: \\ Spectral Theory and Numerics}

\author[J. L\"uhrmann]{Jonas L\"uhrmann}
\address{Department of Mathematics and Computer Science, University of Cologne, Cologne, Germany}
\email{jonas.luehrmann@uni-koeln.de}

\author[J. M. Palacios]{Jos\'e M. Palacios}
\address{Institute of Mathematics, \'Ecole Polytechnique F\'ed\'erale de Lausanne (EPFL), Lausanne, Switzerland}
\email{jose.palaciosarmesto@epfl.ch}

\author[F. Pusateri]{Fabio Pusateri}
\address{Department of Mathematics, University of Toronto, Toronto, Ontario, Canada}
\email{fabiop@mail.math.toronto.edu}

\author[W. Schlag]{Wilhelm Schlag}
\address{Department of Mathematics, Yale University, New Haven, Connecticut, USA}
\email{wilhelm.schlag@yale.edu}

\author[S. Shahshahani]{Sohrab Shahshahani}
\address{Department of Mathematics and Statistics, University of Massachusetts Amherst, Amherst, Massachusetts, USA}
\email{sshahshahani@umass.edu}

\thanks{
J.L.\  was partially supported by NSF CAREER grant DMS-2235233.
F.P.\ is supported in part by a start-up grant from the University of Toronto and NSERC grants RGPIN-2018-0648 and RGPIN-2025-06419. 
J.M.P.\ was partially supported by NSERC grants RGPIN-2018-0648, and is also supported by the Swiss National Science Foundation grant 225701. 
W.S.\ is partially supported by the United States NSF DMS-2350356, and he thanks the Isaac Newton Institute in Cambridge for its hospitality during the preparation of this paper. S.S. was supported by the Simons Foundation grant 639284. 
The authors relied on OpenAI's Codex system for assistance with writing and running the C++ code,
reviewing numerical transcripts, and preparation of reproducibility
materials.  All mathematical arguments, numerical claims, and final verification
decisions are the responsibility of the authors.
}

\begin{document}

\begin{abstract}
This is the first of three papers proving asymptotic stability of the degree-one vortex under equivariant perturbations in the $(1+2)$-dimensional abelian Yang--Mills--Higgs model at self-dual coupling.  In the orthogonal gauge, the linearized dynamics are  governed by a selfadjoint matrix Schr\"odinger operator $\bfM$. The super-symmetric partner operator is a diagonal matrix whose diagonal entries are strongly singular radial Schr\"odinger operators on $\bbR^2$. After a conjugation, this reduces the spectral problem to the analysis of two strongly singular scalar half-line operators. Combining analysis with rigorous interval arithmetic, we prove absence of threshold resonances and show that the discrete spectrum of $\bfM$ consists of exactly one positive gap eigenvalue (internal mode) with a two-dimensional eigenspace.  We also certify that the relevant nonlinear Fermi Golden Rule coefficients form a definite quadratic form, yielding effective nonlinear damping of the internal mode. These spectral inputs  form the basis of the  stability analysis in the subsequent papers.
\end{abstract}

\maketitle 
  
\setcounter{tocdepth}{2}

\tableofcontents

\section{Introduction}

\subsection{The model}
We consider the abelian Yang-Mills-Higgs (Maxwell-Higgs) model on $(1+2)$--dimensional Minkowski space $\bbR^{1+2}$ with coupling parameter $\lambda>0$, described by the Lagrangian action functional  
\begin{equation}
    \calL_\lambda[\phi,A]
    := \int_{\bbR^{1+2}}
    \biggl(
        \frac14 F_{\mu\nu}F^{\mu\nu}
        + \frac12 \bfD_\mu\phi\,\overline{\bfD^\mu\phi}
        + \frac{\lambda}{8}\bigl(1-|\phi|^2\bigr)^2
    \biggr)\,\ud x\,\ud t,
\end{equation}
for a complex-valued field $\phi:\bbR^{1+2}\to\bbC$, usually called the Higgs field, and a real-valued connection $1$-form (or electromagnetic potential)
\begin{equation}
    A = A_0\,\ud t + A_1\,\ud x^1 + A_2\,\ud x^2,
    \qquad A_\mu:\bbR^{1+2}\to\bbR,\quad 0\le \mu\le 2.
\end{equation}
The associated covariant derivative and curvature (or electromagnetic tensor) are
\begin{equation}
    \bfD_\mu := \nabla_\mu - iA_\mu,
    \qquad
    F= \ud A = F_{\mu\nu} \, \ud x^\mu \wedge\ud x^\nu,\quad F_{\mu\nu} := \partial_\mu A_\nu - \partial_\nu A_\mu.
\end{equation}
The parameter $\lambda>0$ is the Higgs self-coupling constant.
We use Einstein's summation convention\footnote{We lower and raise indices with respect to the Minkowski metric $\mathrm{diag}[-1,1,1]$ and its inverse. Greek letters run over $0,1,2$ and Roman indices run over $1,2$.}.

The Euler--Lagrange equations associated with $\calL_\lambda$ are
\begin{equation}\label{equ:EL_general_lambda_smooth}
    \left\{
    \begin{aligned}
        \bfD^\mu \bfD_\mu\phi + \frac{\lambda}{2}\bigl(1-|\phi|^2\bigr)\phi &= 0,\\
        \nabla^\mu F_{\mu\nu} + \Im\bigl(\overline{\phi}\,\bfD_\nu\phi\bigr) &= 0.
    \end{aligned}
    \right.
\end{equation}
A key feature is their invariance under $U(1)$ gauge transformations 
\[
    (\phi,A) \,\mapsto\, \bigl( e^{i\chi} \phi, A + \ud \chi \bigr) \quad \hbox{ for any } \chi \in C^2(\R^{1+2}).
\]
Defining the magnetic field $B$ and the electric field components $E_1,E_2$ by
\begin{equation}
    B := F_{12},\qquad E_1 := F_{01},\qquad E_2 := F_{02},
\end{equation}
we may write \eqref{equ:EL_general_lambda_smooth} more concretely in rectangular coordinates as 
\begin{equation}\label{KG_smooth}
    \left\{
    \begin{aligned}
        -\bfD_0^2\phi + \bfD^j\bfD_j\phi + \frac{\lambda}{2}\bigl(1-|\phi|^2\bigr)\phi &= 0,\\
        \partial_0E_1 + \partial_2B - \Im\bigl(\overline{\phi}\,\bfD_1\phi\bigr) &= 0,\\
        \partial_0E_2 - \partial_1B - \Im\bigl(\overline{\phi}\,\bfD_2\phi\bigr) &= 0,\\
        \partial_1E_1 + \partial_2E_2 - \Im\bigl(\overline{\phi}\,\bfD_0\phi\bigr) &= 0.
    \end{aligned}
    \right.
\end{equation}
The (formally conserved) Ginzburg–Landau (GL) energy is given by
\begin{equation}\label{equ:GL_energy_general_lambda_smooth}
    \mathcal{E}_\lambda[\phi,A]
    := \int_{\bbR^2}
    \biggl(
        \frac12 (B^2 + E_1^2+E_2^2)
        + 
        \frac12|\bfD_0\phi|^2+\frac12|\bfD_1\phi|^2
        + \frac12|\bfD_2\phi|^2
        + \frac{\lambda}{8}\bigl(1-|\phi|^2\bigr)^2
    \biggr)\,\ud x .
\end{equation}
The time-independent critical points of $\mathcal{E}_\lambda[\phi,A]$ are precisely the static solutions of \eqref{equ:EL_general_lambda_smooth} (in a gauge with $A_0=0$), and in particular include the vortex solutions described below. The static Euler-Lagrange equations are as follows: 
\begin{equation} \label{EL}
    \left\{ \begin{aligned}
        \bfD^j \bfD_j\phi+\frac{\lambda}{2}\big(1-\vert\phi\vert^2\big)\phi &= 0, \\
        \partial_2 B - \Im\big(\overline{\phi}\bfD_1\phi\big) &= 0, \\
        \partial_1 B + \Im\big(\overline{\phi}\bfD_2\phi\big) &= 0.
    \end{aligned} \right.
\end{equation}

Finite-energy configurations are forced to approach a vacuum state at spatial infinity. In particular, the potential term requires that the following conditions are met: \[
\vert \phi (x) \vert \to 1 \quad \hbox{and} \quad \bfD\phi(x) \to 0 \quad \hbox{as} \quad \vert x\vert\to +\infty,
\]
which defines $\phi$ on the circle at infinity. Its winding number or degree $\mathrm{deg}(\phi)\in\Z$ as a map $\bbS^1 \to \bbS^1$ is well-defined and gauge-invariant.
On the other hand, for regular solutions, these same asymptotics imply that
\begin{equation} \label{equ:magnetic_flux}
    \int_{\R^2} B \, \ud x = 2\pi \, \mathrm{deg}(\phi),
\end{equation}
i.e., $\mathrm{deg}(\phi)$ is also the flux quantization of the induced magnetic field $B$. 
In the Yang-Mills-Higgs theory, the degree $\mathrm{deg}(\phi)$ is also known as the \textit{first Chern number} (also called the topological charge or magnetic flux) of the associated $U(1)$ principal bundle with connection $A$. 
Moreover, the mapping
\[
    \phi \mapsto \overline{\phi} \quad \hbox{ and } \quad A \mapsto -A
\]
leaves \eqref{EL} invariant and swaps negative degrees for positive degrees.
Thus, in what follows, it is enough to consider solutions with nonnegative degrees.

\subsection{Static solutions}

For all values of the coupling constant $\lambda$ except $\lambda = 1$, all known finite-energy solutions of the stationary  equations \eqref{EL} have equivariant symmetry about some point. 
By translational symmetry, we can choose that point as the origin.
They are thus of the form
\begin{equation} \label{phinAn_smooth}
    \begin{aligned}
        \underline{\Phi}_n(x,y) & = U_n(r) e^{in\theta}, \quad  & & \quad U_n(0)=0, & U_n(+\infty)=1, \\
        \underline{A}_n(x,y) &= a_n(r) \, \ud \theta, \quad & & \quad  a_n(0)=0, & a_n(+\infty)=n,
    \end{aligned}
\end{equation}
where $(r,\theta)$ are the polar coordinates in the plane and where $(U_n, a_n)$ are real-valued, smooth, increasing, and positive solutions of the system of ODEs
\begin{equation} \label{eq_vortex_smooth}
    \left\{ \begin{aligned}
        U_n'' + \dfrac{1}{r} U_n' - \dfrac{(n-a_n)^2}{r^2} U_n + \frac{\lambda}{2}(1-U_n^2) U_n &= 0, & & U_n(0)=0, & & U_n(+\infty)=1, & & U'_n >0, \\
        a_n'' - \tfrac{1}{r} a_n' + (n-a_n) U_n^2 &= 0, & & a_n(0)=0, & & a_n(+\infty)=n, & & a_n' >0.
    \end{aligned} \right.
\end{equation}
These solutions satisfy the asymptotics $U_n(r) \to 1$, $a_n(r) \to n$ exponentially as $r\to\infty$
and near the origin (for some constants $\mathfrak{c}_n, \mathfrak{d}_n > 0$)
\begin{equation}\label{asympt_U_0_smooth}
U_n(r)\approx \mathfrak{c}_n r^n
\quad\text{and}\quad
a_n(r)\approx \mathfrak{d}_n r^2
\quad\text{as}\quad r\to 0.
\end{equation}
We refer to \cite{Plo1,BeCh} for the existence theory of solutions to \eqref{eq_vortex_smooth}.
The basic solution $(\underline{\Phi}_1, \underline{A}_1)$ with winding number $n=1$ is called the degree-one vortex, while we refer to solutions with higher winding number $n > 1$ as $n$-vortices or multi-vortices. They are well-known examples of topological solitons in two space dimensions.
With a slight abuse of notation, we will sometimes refer to $(U_1, a_1)$ as a vortex as well.

For all values $\lambda > 0$ of the coupling constant, the degree-one vortex is (spectrally and variationally) stable, while for $n\geq2$ the $n$-vortices are stable if $\lambda<1$ and unstable if $\lambda>1$ \cite{GusSigVort}. 
The case $\lambda=1$ is called the self-dual case.
In that case, as observed by Bogomolny~\cite{Bogo,BeCh}{bookJT}, the energy functional can be rewritten as a constant plus a sum of squares. Setting the latter equal to zero reduces the Ginzburg-Landau equations to a system of first-order self-dual equations (see \cite{bookJT}).
Indeed, at the critical coupling $\lambda=1$, the static energy functional exhibits the Bogomolny factorization \cite{Bogo}
\begin{equation} \label{EBogo}
    \begin{aligned}
        E_{\mathrm{static}}\bigl[\phi, A \bigr] &= \int_{\bbR^2} \biggl( \frac12 B^2 + \frac12 |\bfD_1\phi|^2 + \frac12 |\bfD_2\phi|^2 + \frac18 \bigl( 1 - |\phi|^2 \bigr)^2 \biggr) \, \ud x \\ 
        &= \frac12 \int_{\bbR^2} \biggl( \Bigl( B \mp \frac12 \bigl( 1 - |\phi|^2 \bigr) \Bigr)^2 + |(\bfD_1 \pm i \bfD_2) \phi|^2 \biggr) \, \ud x \pm \frac12 \int_{\bbR^2} B \, \ud x. 
    \end{aligned}
\end{equation}
The choice of sign corresponds to the sign of the topological degree.
In view of the magnetic flux quantization \eqref{equ:magnetic_flux}, energy minimizers in a fixed degree class must therefore solve the first-order Bogomolny equations. For positive degree these are
\begin{equation} \label{equ:1st_order_Bogomolny_smooth}
    B = \frac12 \bigl( 1 - |\phi|^2 \bigr), \quad (\bfD_1 + i \bfD_2) \phi = 0.
\end{equation}
We look for static equivariant $1$-vortex solutions to \eqref{equ:1st_order_Bogomolny_smooth} with finite energy of the form
\begin{equation}\label{equ:onevortex_sol}
    \underline{\Phi}(r,\theta) = e^{i\theta} U(r), \qquad \underline{A}(r,\theta) = a_\theta(r) \, \ud \theta.
\end{equation}
Passing to polar coordinates, the Bogomolny equations \eqref{equ:1st_order_Bogomolny_smooth} read as
\begin{equation}\label{equ:onevortex_ODEs_smooth}
   \partial_r U  = \frac{1-a_\theta}{r} U, \qquad \partial_r a_\theta = \frac{r}{2} \bigl( 1 - U^2 \bigr).
\end{equation}
The finite energy requirement implies
\begin{equation}
    U(r) \to 1, \qquad a_\theta(r) \to 1 \qquad \text{as } r \to \infty.
\end{equation}
The degree-one ($n=1$) vortex is the unique solution (see \cite{bookJT}) to the system \eqref{equ:onevortex_ODEs_smooth} satisfying the boundary conditions 
\begin{equation}
    U(0) = 0, \quad \lim_{r \to \infty} U(r) = 1, \quad U'(r) > 0, \qquad
    a_\theta(0) = 0, \quad \lim_{r\to\infty} a_\theta(r) = 1, \quad a_\theta'(r) > 0.
\end{equation}

\begin{figure}
\centering
\includegraphics[scale=0.95]{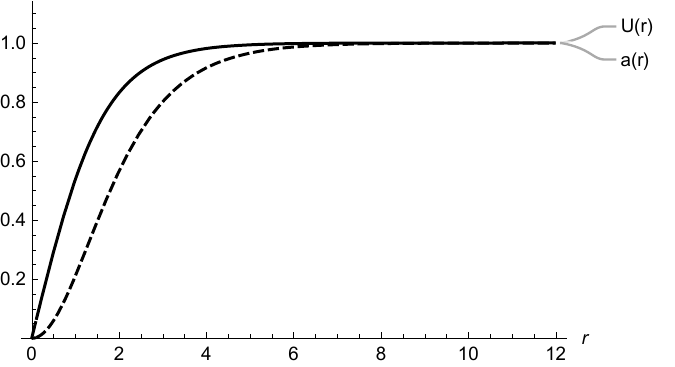}
\caption{Magnetic vortex $(U,a)$ in the case $n=1$.}
\label{figvortex}
\end{figure}


\subsection{Equivariant perturbations of the degree-one vortex solution}

We consider equivariant perturbations of the degree-one vortex \eqref{equ:onevortex_sol} and decompose the Higgs field $\phi$ and the connection form $A$ correspondingly as
\begin{equation}
    \begin{aligned}
        \phi(t,r,\theta) &= \underline{\Phi}(r,\theta) + \varphi(t,r,\theta) =  e^{i\theta} \bigl( U(r) + \alpha(t,r) + i \beta(t,r) \bigr), 
        \\ 
        A(t, r,\theta) &= \underline{A}(r,\theta) + \eta(t,r,\theta) = a_\theta(r) \, \ud \theta + \eta_\theta(t,r) \, \ud \theta + \eta_r(t,r) \, \ud r + \eta_0(t, r) \, \ud t,
    \end{aligned}
\end{equation}
evolving according to the relativistic flow \eqref{KG_smooth}. 
Moreover, we impose Stuart's gauge (see \cite{Stu})
\begin{equation}\label{eq:Stuart}
    \frac{1}{r} \partial_r \bigl( r \eta_r \bigr) = U \beta.
\end{equation}
We introduce the following short-hand notation
\begin{equation}
   \zeta := \frac{\eta_\theta}{r}, \quad \mu := - \eta_r.
\end{equation}
Then by direct computation, see \cite[Section 1.4]{LPPSS3}, we infer from the field equations~\eqref{KG_smooth} 
that the dynamical variables $\alpha$, $\zeta$, $\beta$, $\mu$ satisfy the following evolution equations relative to Stuart's gauge,
\begin{equation} \label{equ:sys3}
    \begin{aligned}
        \pt^2 \begin{pmatrix} \alpha \\ \zeta \\ \beta \\ \mu \end{pmatrix} 
        + 
        \begin{pmatrix}
            L_1 & -2U' & 0 & 0 \\
            -2U' & L_2 & 0 & 0 \\ 
            0 & 0 & L_1 & -2U' \\
            0 & 0 & -2U' & L_2 
        \end{pmatrix}
        \begin{pmatrix} \alpha \\ \zeta \\ \beta \\ \mu \end{pmatrix} 
        = 
        \begin{pmatrix}
            \calN_\alpha \\ \calN_\zeta \\ \calN_\beta \\ \calN_\mu 
        \end{pmatrix},
    \end{aligned}
\end{equation}
coupled to an elliptic equation for the temporal component $\eta_0$,
\begin{equation} \label{sys2}
    \big(-\Delta + U^2\big) \eta_0 = \calN_0.
\end{equation}
Here we use the following notation for the scalar linearized operators
\begin{equation}
    L_1 := -\Delta + \frac{(1-a)^2}{r^2} - \frac{\pr a}{r} + U^2, \quad L_2 := -\Delta + \frac{1}{r^2} + U^2. 
\end{equation}
Note that for the remainder of the paper $a(r)=a_\theta(r)$. 
Defining
\begin{align} \label{eq:bL1L2}
    b := \frac{1-a}{r}, 
\end{align}
the nonlinearities are given by 
\begin{equation} \label{sysN1}
    \begin{aligned}
        \calN_\alpha &:= - \frac32 U \alpha^2 + \frac12 U \beta^2 - \frac12 (\alpha^2+\beta^2)\alpha - 2 \mu \pr\beta + 2 b \zeta \alpha - U (\zeta^2 + \mu^2) - (\zeta^2 + \mu^2) \alpha \\
        &\quad \quad - 2\eta_0 \pt \beta - (\pt \eta_0) \beta + U \eta_0^2 + \eta_0^2 \alpha, \\
        \calN_\zeta & := b (\alpha^2 + \beta^2) - 2 U \zeta \alpha - \zeta (\alpha^2 +\beta^2), \\   
        \calN_\beta & := - 2U \alpha\beta - \frac12 (\alpha^2+\beta^2)\beta + 2\mu \pr \alpha + 2 b \zeta \beta - (\zeta^2+\mu^2) \beta \\
        &\quad \quad + 2\eta_0 \pt \alpha + U (\pt \eta_0) + (\pt \eta_0) \alpha + \eta_0^2 \beta, \\ 
        \calN_\mu & := -\pr \pt \eta_0 - \alpha \pr \beta + \beta \pr \alpha - 2U \alpha \mu - (\alpha^2+\beta^2) \mu,
    \end{aligned}
\end{equation}
and 
\begin{equation} 
    \calN_0 := -\beta \pt \alpha + \alpha \pt \beta - 2 U \eta_0 \alpha - \eta_0 (\alpha^2 + \beta^2).   
\end{equation}

Differentiating \eqref{sys2} in time and inserting the evolution equations \eqref{equ:sys3} gives 
\begin{equation} \label{equ:pt_eta0_equation}
    \begin{aligned}
        \big(-\Delta + U^2\big) \pt \eta_0 &= 
        -\beta \pt^2 \alpha + \alpha \pt^2 \beta + \partial_t \big( - 2 U \eta_0 \alpha - \eta_0 (\alpha^2 + \beta^2) \big)
        \\
        &= \beta L_1\alpha - 2U' \beta \zeta - \alpha L_1 \beta + 2U' \alpha \mu + \calC_0
    \end{aligned}
\end{equation}
with 
\begin{equation}
    \begin{aligned}
        \calC_0 &:= - \beta \calN_\alpha + \alpha \calN_\beta + \pt \bigl( - 2U\eta_0\alpha - \eta_0 (\alpha^2 + \beta^2) \bigr).
    \end{aligned}
\end{equation}
In particular, both $\eta_0$ and $\partial_t \eta_0$ can be 
elliptically recovered from the other four unknowns.

\subsection{Results: Spectral theory and numerics for the dynamics near the degree one vortex}

From now on we write the evolution equations
\eqref{equ:sys3} for the perturbation of the 1-vortex as
\begin{equation}
    (\pt^2 + \bfM) \bmr = \bmN
\end{equation}
with
\begin{equation}\label{M}
    \bmr := \begin{pmatrix} \alpha \\ \zeta \\ \beta \\ \mu \end{pmatrix}, 
    \quad 
    \bmN := \begin{pmatrix} \calN_\alpha \\ \calN_\zeta \\ \calN_\beta \\ \calN_\mu \end{pmatrix},
    \quad
    \bfM := \begin{pmatrix} \bfL & 0 \\ 0 & \bfL \end{pmatrix}, 
    \quad 
    \bfL := \begin{pmatrix} L_1 & -2U' \\ -2U' & L_2 \end{pmatrix}.
\end{equation}
Define $\bfM$ as a symmetric operator on $(L_{\mathrm{rad}}^2(\R^2))^4$ 
with domain $(\calD_0)^4$ with
\[\calD_0:=\{f\in C_{\mathrm{rad}}^2(\R^2)\::\: f=0 \text{\ \ near the origin and near infinity}\}. \] 
Then $\bfM$ is essentially selfadjoint with domain 
$(\calD)^4$ where 
\begin{equation}\label{defcalD} 
\calD:=\big\{ f\in L^2_{\textrm{rad}}(\R^2)\::\: -g''(r)+\frac{3}{4r^2}g(r)\in L^2((0,\infty),dr),\; g(r)=r^{\frac12}f(r)\big\},
\end{equation}
the derivative being in the weak sense.
In this paper we study the spectrum of $\bfM$ 
and, using a combination of analytical and rigorous numerical methods, 
we prove the following result.

From this point on the coupling constant is fixed at one, and $\lambda>0$ denotes the internal frequency.

\begin{theorem}[Spectrum of the linearized operator]\label{maintheo1}
Let $\bfM$ be the self-adjoint operator defined above. Then:

\setlength{\leftmargini}{2em}
\begin{itemize}

\smallskip
\item[(i)] The restriction $\bfM|_{\operatorname{Ran}\bfP_c}$ exhibits purely absolutely continuous spectrum $[1,\infty)$, where $\bfP_c$ denotes the projection onto the continuous spectrum. In particular, the bottom of the continuous spectrum\footnote{The threshold is also not a resonance but characterizing what this means takes some care. This property will manifest itself in the construction of the distorted Fourier transform for $\bfM$ in Part~II, see \cite{LPPSS2}.} is not an eigenvalue.

\smallskip
\item[(ii)] $\bfM$ has a unique gap eigenvalue of multiplicity two (also known as an ``internal" mode, or oscillating mode), which we denote by $\lambda^2$. It satisfies $\lambda^2\in [0.777471875,\,0.77747375]$. The associated eigenspace is spanned by
\begin{align*}
\bmY_1 := \lambda^{-1}(U\psi, -\psi',0,0)^T, \qquad \quad  \bmY_2 := \lambda^{-1}(0, 0, U\psi , -\psi')^T, \qquad \bfM \bmY_j = \lambda^2\bmY_j,
\end{align*}
where $\psi$ is the $L^2_{rdr}$-normalized\footnote{This means $\int_0^\infty \psi(r)^2r\, dr=1$ without a $2\pi$ factor. That $\|\bmY_j\|_2=1$ follows from the eigenvalue equation and integration by parts. } radial ground state  of $-\Delta + U^2$ with energy $\lambda^2$.
\end{itemize}

\end{theorem}

\begin{rem}
This work contains two other main  results that will be crucial for the analysis in \cites{LPPSS2}{LPPSS3}. Stating these  precisely requires more preparation, so they are postponed to later sections. Here we include a brief description:
\begin{enumerate}
    \item We will see that the spectral analysis of $\bfM$ reduces to that of a pair of scalar Schr\"odinger operators $\calL_1$ and $\calL_2$. See \eqref{eq:lemSd1}. In Proposition~\ref{prop:lem_p2_t2} we use rigorous numerics to prove that $\calL_2$ does not have a threshold resonance. See Remark~\ref{rem:c1*notzero} for the corresponding characterization in terms of the threshold asymptotics. Proposition~\ref{prop:lem_p2_t2} also shows the absence of a threshold resonance for $\calL_1$, although this is done purely analytically. That the restriction of $\bfM$ to its continuous spectral subspace exhibits purely ac spectrum follows from the shape of the distorted Fourier basis in Proposition~\ref{propdFT} by standard arguments such as \cite[Theorem XIII.19]{RS4}. To apply that theorem, one expresses the resolvent kernel via the distorted Fourier basis as in~\cite[Section 2.4.4]{LPPSS2}. The remaining statements in Theorem~\ref{maintheo1} are then deduced later from the reduction \eqref{eq:lemSd1} (together with \eqref{eq:LtoH} and \eqref{eq:LBstarB}), Proposition~\ref{prop:lem_p2_t2}, the kernel computation in Lemma~\ref{lemB}, and the eigenvalue enclosure\footnote{Throughout this paper, an
``enclosure'' of a number means a certified interval containing it.  An
enclosure of a function means certified pointwise intervals containing
its values, componentwise, on the stated domain. ``Certified" means that the computations use outward-rounded multiprecision interval arithmetic, and that the finite truncations are justified by analytic truncation and tail estimates so that the resulting intervals provably enclose the exact quantities. }
in Lemma~\ref{lem:fgr_eigenvalue_box}.
    \item In the nonlinear analysis the contribution of the internal modes of $\bfM$ decays in time through a nonlinear mechanism known as radiation damping. For this it is important that certain constants that encode the nonlinear coupling between the discrete and continuous spectra of $\bfM$ are non-vanishing. The verification of this non-vanishing in Proposition~\ref{prop:FGR1} is the other main result of this work. See Section~\ref{sec:FGR} for a more detailed discussion.
    \item[] The bulk of the paper is devoted to the asymptotic analysis and certified computation (via interval arithmetic) of the vortex profile $(U,a)$ and of the spectral and scattering properties of the associated linearized operators built from $(U,a)$, including: absence of threshold resonances, enclosure of the internal mode $\lambda^2$, simplicity of the discrete spectrum, computation of the distorted Fourier basis at energy $4\lambda^2$, and verification of a negative sign for the FGR coefficients. This is technically delicate because the underlying ODEs are singular at $r=0$ and require quantitative tail bounds as $r\to\infty$, and because all steps must be carried out with fully rigorous numerical error control. The computation of $(U,a)$ is delicate due to the fact that there is a unique value of $U'(0)>0$ for which the ODE remains stable, while the linearization about the vortex is exponentially unstable. 
\end{enumerate} 
\end{rem}
As we will explain in Section~\ref{sec:spectral_fourier}, Theorem~\ref{maintheo1} is obtained by first using a super-symmetric or Bogomolny factorization and then studying a pair of relatively simpler ``strongly singular'' (see~\cite{GNZ} for this terminology and a comprehensive study of Sturm-Liouville operators) scalar operators $\mathcal{L}_1$ and $\mathcal{L}_2$. See Proposition~\ref{prop:lem_p2_t2} together with the definitions
\eqref{eq:LtoH}, \eqref{eq:lemSd1} and \eqref{eq:LBstarB}.
By studying the spectral properties of these operators,
and in particular proving the existence of a unique ground state with energy $\lambda^2$ for the operator $\mathcal{L}_2 :=-\Delta + U^2$ we then deduce Theorem~\ref{maintheo1}. Details are provided after the statement of Proposition~\ref{prop:lem_p2_t2}.

For the dynamical decomposition, we choose an orthonormal basis
$\bmY_1,\bmY_2$ of the gap eigenspace in Theorem~\ref{maintheo1},
with each vector supported in its corresponding two-component block.
In view of Theorem~\ref{maintheo1}, we naturally further decompose 
$\bmr$ as follows:
\begin{equation}\label{eq:ruz}
    \bmr(t) = \bmu(t) + z_1(t) \bmY_1 + z_2(t) \bmY_2, \qquad \bmu(t) := \bfP_c \bmr(t), \qquad z_j(t) := \langle \bmY_j, \bmr(t) \rangle.
\end{equation}
Here, $\bfP_c$ denotes the projection to the continuous spectral subspace of $L^2(\bbR^2;\bbR^4)$ with respect to~$\bfM$.

The presence of a bound state for $\bfM$ is a clear obstruction and can, in principle, generate a non-decaying state for the full nonlinear flow, thus making the dynamics too complicated to control for long times. However, it is well-known that, under a suitable generic condition, the solutions $e^{\pm i\lambda t}\bmY_j$ of the linear equation 
 are unstable for the nonlinear flow of \eqref{equ:sys3}. 
 This condition is the so-called nonlinear {\it Fermi Golden Rule} (FGR),
 and was put forth and analyzed in the groundbreaking works  \cite{Sig,BP95}. 
 See also the fundamental work of Soffer-Weinstein \cite{SW} and the discussion at the end of Subsection~\ref{Ssecprevious}. 
 The effect of this FGR is to turn the linear harmonic oscillator into a nonlinearly damped oscillator that exhibits dissipation.

With the decomposition \eqref{eq:ruz} we then obtain the following system of evolution equations for the radiation term and for the internal mode components
\begin{equation}\label{eq:sysuzj}
    \left\{ \begin{aligned}
        (\pt^2 + \bfM) \bmu &= \bfP_c \bmN, \\ 
        (\pt^2 + \lambda^2) z_1 &= \langle \bmY_1, \bmN \rangle, \\ 
        (\pt^2 + \lambda^2) z_2 &= \langle \bmY_2, \bmN \rangle,
    \end{aligned} \right.
\end{equation}
coupled to the elliptic equation \eqref{sys2} for the variable $\eta_0$.

Following \cite{SW} (see also \cite{LP22}), in Section~\ref{sec:FGR} we present the leading order ODEs satisfied by $(z_1,z_2)$. In Section~\ref{sec_NumVer_FGR} we verify via interval arithmetic that the Fermi Golden Rule coefficients yield a definite multilinear form, allowing for the dissipative behavior of the system for $(z_1,z_2)$. In order to calculate these coefficients, we first need to introduce the distorted Fourier transform associated with $\bfM$, which we do in Section~\ref{sec:spectral_fourier}.


\subsection{Previous results}\label{Ssecprevious}
Following common practice, cf.~the classic texts by Jaffe, Taubes~\cite{bookJT} and Manton, Sutcliffe~\cite{bookManSut}, one separates
{\it non--gauged} (or ``non--magnetic'') vortex models, where the scalar field evolves without
an electromagnetic force, from the {\it gauged} (magnetic) abelian Yang-Mills-Higgs or Ginzburg--Landau models,
in which the Higgs field is coupled to a $U(1)$ connection as in Maxwell's theory.  In both settings, one further distinguishes the
static theory focusing on existence and uniqueness, variational structure, and spectral properties,  from the dynamics. The latter typically includes the parabolic gradient flow, as well as conservative Schr\"odinger and relativistic flows. 
From the point of view of dynamics, which is the one we are mostly concerned with here, long-time stability questions for vortices 
and the interaction laws governing multi--vortex motions
constitute an exciting and challenging area
that sits within the broader context of the {\it soliton resolution conjecture} for nonlinear dispersive PDEs.

\subsubsection*{Gauged theories}
 The monograph of Jaffe and Taubes~\cite{bookJT} is a foundational reference for the planar abelian
Yang-Mills-Higgs, or Maxwell--Higgs,  model at the critical coupling $\lambda=1$, where the static energy
admits a Bogomolny factorization (see \eqref{EBogo}) and the Euler--Lagrange equations reduce to the first--order
Bogomolny system (see \eqref{equ:1st_order_Bogomolny_smooth}). 
This type of factorization
is also crucial to the classification of harmonic maps in two dimensions where the second order harmonic maps PDE reduces to a first order Cauchy-Riemann system.

Jaffe and Taubes established the existence and uniqueness of
finite--energy Bogomolny vortices with prescribed vortex set, i.e., the zeros of~$\phi$. These vortices minimize the energy in their  degree class and the proof in~\cite{bookJT} proceeds by  solving the Bogomolny equations via variational methods. For the sake of completeness, as well as for computational reasons, Proposition~\ref{prop:exuniqUa} below contains an independent ODE proof of this fact for the degree~$1$ vortex. The self--dual structure is
also the starting point for the moduli--space viewpoint developed in the physics literature and
systematically explained in~\cite[Ch.~7]{bookManSut}, including the interpretation of slow vortex
motion in terms of geodesic flow on the moduli space due to Manton.

Away from the critical coupling $\lambda\neq1$, vortices are no longer energy minimizers in their degree class and the approach in \cite{bookJT} does not apply. They are typically {\it critical points} and may be stable or unstable depending on their degree~$n$ and whether one is in the type~I ($\lambda<1$, vortices of degree~$1$ attract) or type~II ($\lambda>1$, vortices of degree~$1$ repel) regimes. A detailed spectral and
dynamical picture for magnetic vortices, including transitions between stability and instability, was developed
by Gustafson and Sigal~\cite{GusSigVort,GusSig2} and related works. See also Gustafson~\cite{GusVort0,GusVort}
for earlier contributions. These results connect the variational structure of the static energy
to the spectral properties of the linearized operator and to effective descriptions of vortex
motion.

A key reference for {\it non--selfdual} vortex dynamics in the gauged abelian Higgs model is
Stuart~\cite{Stu}, who analyzed dynamics in the near--Bogomolny regime. Stuart confirmed Nick Manton's  
 adiabatic approximation for the Lorentzian abelian Higgs equations when  $\lambda$ is close to~$1$. Solutions starting near the Bogomolny vortex manifold remain close  to it (modulo the gauge) for long times. Moreover,  the vortex parameters evolve according to an effective finite-dimensional system: geodesic motion on the Bogomolny moduli space perturbed by a small $O(\lambda-1)$ potential that induces attraction or repulsion of vortices of equal sign. 
A crucial ingredient is a gauge choice that eliminates the infinite-dimensional gauge kernel and ensures the coercivity of the linearized operator transverse to the moduli directions. This gauge is now commonly called Stuart's gauge, and it is the gauge we choose to work with in this paper, see \eqref{eq:Stuart}.

For parabolic, i.e., gradient flow dynamics in the magnetic model, Demoulini and
Stuart as well as Zhao~\cite{StuDem, Zhao26} studied the planar gradient flow of the superconducting Ginzburg--Landau
functional, contributing to the rigorous understanding of dissipative vortex evolution in the
gauged setting.

We refer the reader to the monograph of Sandier and Serfaty~\cite{bookSS} on the magnetic Ginzburg--Landau model for vortex
asymptotics, $\Gamma$--convergence, and mean--field limits with many vortices. 

The present paper focuses on the self--dual gauged model and on equivariant perturbations of the degree--one
magnetic vortex, with an emphasis on spectral properties of the linearized operator and their
consequences for the dynamics. This involves careful analysis of the role of the internal modes and the associated nonlinear Fermi Golden
Rule mechanism.  
From the above perspective, this lies at the intersection of (i) the
Jaffe--Taubes self-dual existence and structure theory~\cite{bookJT}, (ii) the spectral stability
program for vortices in the gauged setting~\cite{GusSigVort,GusSig2}, and (iii) the long-time nonlinear stability of topological solitons for dispersive PDEs.

Finally, we mention the papers~\cite{AIMCQ, AIBPMCNOQu} from the applied mathematics and physics communities that are closely related to this work. There, too, spectral properties are studied via analytic and numerical methods to describe the dynamics near abelian Higgs vortices.

\subsubsection*{Non-gauged theories}
On the non--gauged side, one typically studies evolutions for a complex scalar field governed by a 
Ginzburg--Landau Hamiltonian without a gauge potential. In this model, vortices exhibit infinite energy and the effective
interaction between well--separated vortices is governed by a renormalized energy, as proposed by Ovchinnikov-Sigal \cite{OvSig1,OvSig2}.
It is interesting to note that, at least formally, two well-separated vortices of degrees~$+1$, respectively~$-1$ do lead to a finite energy state of degree~$0$ as the logarithmic divergences of the energy in each vortex cancel out. 
In a seminal contribution, Neu~\cite{Neu} derived vortex motion laws in complex scalar field models and clarified how the vortex cores move according to classical mechanical equations. The formal analysis is based on matched asymptotics and energy considerations. 
Ovchinnikov and Sigal~\cite{OvSig1,OvSigNon1,OvSigNon2,OvSig2} made this analysis more rigorous studying static vortex configurations and the
renormalized interaction energy that controls the leading order forces between well--separated
vortices. These authors developed PDE tools in the non-gauged setting such as the renormalized energy, vortex interaction energy, interaction expansions, and reducing the problem to a finite set of vortex parameters, that have since become central tools in the mathematical analysis  of vortex dynamics. 
More recently, del Pino, Juneman, and Musso~\cite{delPJM} applied Lyapunov–Schmidt reduction and gluing techniques to construct entire families of multi-vortex solutions. Their work recovers the leading-order interactions first described by John Neu~\cite{Neu}. In addition, they discover a wide range of nonradial vortex configurations that extend well beyond the familiar symmetric $n$-vortices.
The vortex cores move according to the classical Helmholtz--Kirchhoff system, at least up to some large finite time. They obtain an asymptotic expansion of the vortex positions in terms of the vortex core size and the first correction to the leading order dynamics is determined by the solution of a linear wave equation.

In the context of the conservative Schr\"odinger or Gross-Pitaevskii evolution, Gravejat--Pacherie--Smets~\cite{GPS} proved an orbital stability result for
the degree--one Ginzburg--Landau vortex in the non--gauged setting, combining variational and
spectral information with a Lyapunov functional approach.  
For the non-magnetic Gross-Pitaevskii equations, Collot, Germain, and Pacherie~\cite{CGP} construct a distorted Fourier basis for the linearized operator around the degree--one vortex at all energies. Additionally, the delicate question of embedded eigenvalues in the essential spectrum, which is highly nontrivial for energies bounded away from $0$ and $\infty$, is resolved by an analytical argument. Using oscillatory integral bounds, \cite{CGP} then also establishes dispersive estimates. This is a necessary step in any attempt to answer the asymptotic stability question in the fully nonlinear setting. 
In a contemporaneous work, the first and the last two authors~\cite{LSS} construct a distorted Fourier basis at small energies, which is the most delicate regime. The main focus of~\cite{LSS} is to derive a Stone-type formula for the not self-adjoint linearized operator $\calL$ for small~$\lambda$. This requires particular attention to $\lambda=0$ which turns out to lead to linear growth of $\| e^{t\calL}\|$ due to a nilpotent contribution. This method was applied to the Ginzburg-Landau model on the hyperbolic plane by Landoulsi and the fifth author~\cite{LaSh} for all energies. 
In contrast, Collot, Germain, and Pacherie derive a putative distorted Fourier basis as in~\cite{KriSchNLS} for real energies and justify after the fact that this basis is complete.

The second and third authors~\cite{PaPu} study the linearized operator for equivariant perturbations around the vortex in the relativistic flow,
which is the non-gauged version of the equations considered here. 
They characterize the spectrum, construct the distorted Fourier transform and establish various linear decay estimates.
Compared to the Schr\"odinger case, in the relativistic context the analysis is simplified by the observation 
that 
$\calL$ turns out to be a diagonal operator and thus self-adjoint,
so that the spectral and scattering analysis reduces to the scalar analysis in the spirit of~\cite{KST}.

Finally, we mention the important work of Weinstein-Xin \cite{WX}
who analyzed corotational perturbations of the $n$-vortex for the GL equations and, among other things, established asymptotic stability
under the parabolic flow.

\subsubsection*{Nonlinear dynamics and Fermi Golden Rule (FGR)}
Spectral and linear scattering theory constitute the first fundamental step towards understanding nonlinear dynamics.
In the presence of internal modes for the linearized operator,
which is the scenario we encounter in this paper,
the next important step is the verification of the FGR.
The latter is a non-degeneracy condition that couples the discrete and continuum spectrum, allowing energy to leak from the bound states to 
the scattering states.
This mechanism was identified in groundbreaking work \cite{Sig,BP95},
where it was shown that the persistence of time-periodic or quasi-periodic solutions
is linked to the failure of this type of coupling, which leads to the decay of the amplitude of oscillations and asymptotic stability in the full nonlinear problem.
Following this, in the case of the cubic Klein-Gordon equation in $3d$, Soffer and Weinstein \cite{SW} showed that 
if the FGR holds, then an effective damping mechanism turns linear bound states into metastable (slowly decaying) states for the full nonlinear flow. See also the successive related works \cite{TsaiYau,BamCuc,LP22,LP22b,LeiLiuYang1} and references therein.
As an application of our spectral and linear scattering theory, we use rigorous interval arithmetic to enclose the coefficients that govern the leading-order dynamics of the internal mode amplitudes $z_j$, and to verify that the FGR holds for our system~\eqref{eq:sysuzj}.

\subsubsection*{Topological solitons in one dimension}
As mentioned previously,
interaction laws governing multi--vortex motion and, in general, topological solitons are an exciting and challenging area, related to the {\it soliton resolution conjecture} for nonlinear dispersive and wave equations. 
While very few results on dynamical stability are available in dimensions larger than one, the mathematical literature in $1$d is richer, though many fundamental problems remain open.

In one-dimensional classical field theories such as the sine-Gordon and $\phi^4$-models, topological solitons are known as kinks. They play exactly analogous roles to vortices, and many open questions remain concerning kink stability and kink--antikink solutions in the $\phi^4$ model. The sine-Gordon equation is completely integrable and can be treated by complex analytic methods~\cite{CLL,AMP23,KochYu}, but see also~\cite{LSch,ChenL} for a perturbative approach. Due to the presence of an internal mode in the linearized $\phi^4$~model, there are parallels to the asymptotic stability problem for the vortex, which we address here. The global asymptotic stability problem for the $\phi^4$-kink is open. See, however, the works~\cite{KoMaMu, KoMaMu2, KoMa} which prove local asymptotic stability under odd perturbations.


\subsection{Numerical algorithms}

As already mentioned, the numerical parts of this work are performed on a computer (a modern laptop suffices). We carry out the certification run with C++ code using the CAPD library.  The CAPD package is open-source, well-documented,  and accepted in the computer-assisted proof community. See for example Dahne and Figueras~\cite{DahFig}.  Much of the analytical work in this paper builds the rigorous foundations required for the interval computations. The final section records the precise role of the code in the proof.

Throughout the paper, we use the term {\em certificate} in the usual computer-assisted
proof sense. It is a finite package of interval enclosures, residual\footnote{The residual of an approximate solution is the difference between the two sides of the differential equation after that approximation is substituted. It vanishes for an exact solution, and bounds on it are used to control the error of the approximation. In other words, it is the error in an ODE generated by an approximate solution.} bounds, tail
estimates, and arithmetic inequalities that can be checked independently of the
heuristic computations that motivated them. 

Such error bounds are particularly important in spectral computations.
For infinite-dimensional operators, straightforward discretizations can
produce spurious eigenvalues or fail to detect parts of the spectrum. For
a comprehensive account of these issues, see Colbrook~\cite{Colbrook}.

\subsubsection{Code and certificate repository.}
The computer-assisted certificate files used in this paper are archived in the
GitHub repository~\cite{AYMHPartICertificateRepo}.  The repository contains the
multiprecision C++ code based on CAPD, the passing transcripts, hash manifests, reproducibility instructions, and
documentation. The CAPD transcripts
certify that the computations are carried out in multiprecision interval arithmetic
(using CAPD's internal \texttt{MpInterval}/\texttt{MpFloat} number types) at $220$-bit precision. They also record an overall pass for the corresponding finite certificate blocks.

\section{Overview}\label{sec:overview}

We now provide a road map of the main technical parts of the paper. To start the entire process, which is a hybrid of analysis and numerical certification, we need to approximate the vortex profiles $(U,a)$. All of this hybrid work is done in 
Sections~\ref{sec_app1}--\ref{sec_NumVer_FGR}. To start off, the decomposition~\eqref{eq:ruz} separates the discrete component, which comes from the internal modes, from the radiation. To control these two discrete components, we first need to establish the required spectral properties: we locate the internal eigenvalue\footnote{Concretely, this means enclosing it in an interval of size~$<2\cdot 10^{-6}$.}, approximate its eigenfunction, and exclude threshold resonances. We then study the coupling of the harmonic oscillator with the continuous spectrum, which is precisely the FGR mechanism. At the resonant frequency
\[
  k_\lambda:=\sqrt{4\lambda^2-1},
\]
this involves the distorted Fourier transform associated with~$\bfM$. Computing these dFT basis elements requires several steps. First, we need to approximate both the Weyl solutions~$\Phi_j$ at zero, and the oscillatory Weyl solutions  $\Psi_j$ at infinity associated with the linearized operators~$\mathcal H_j$. Second, we have to solve the connection problem to glue them together via the connection coefficients~$a_j$.

Using the distorted Fourier representation from Proposition~\ref{propdFT}, we rewrite the FGR coefficients in~\eqref{FGRODE14} as radial integrals at frequency~$k_\lambda$. Showing that these integrals are nonzero proves that the internal modes genuinely interact with the radiation. With the sign convention in~\eqref{FGRODE14}, this interaction leads to dissipative cubic terms, i.e., nonlinear damping of the internal modes. Our spectral results are summarized in Theorem~\ref{maintheo1}, and the needed FGR nondegeneracy is proved in Proposition~\ref{prop:FGR1}. Together, these provide the inputs for the linear and nonlinear analyses in our companion papers \cites{LPPSS2}{LPPSS3}.

The order in which we prove the FGR property is
\[
 (U,a,\uprz)
 \ \Longrightarrow\
 (\mathcal L_1,\mathcal L_2)
 \ \Longrightarrow\
 (\lambda^2,\psi_0,\Phi_1,\Phi_2,\Psi_1,\Psi_2)
 \ \Longrightarrow\
 (D_1,D_2,D_{12}).
\]
Here $\uprz=U'(0)$ is the unique slope of the scalar vortex profile $U$, $\psi_0=\psi/\psi(0)$ is the internal eigenfunction normalized to have height $\psi_0(0)=1$ associated with the eigenvalue~$\lambda^2$, and finally $D_1, D_2, D_{12}$ are the coefficients governing the FGR. The ultimate goal is to prove their strict negativity.   

The vortex profiles determine the coefficients of the Schr\"odinger operators~$\mathcal L_1$ and~$\mathcal L_2$.  Every step relies on the earlier ones, so we need accurate approximations with rigorous error bounds throughout. The ODEs we have to solve are singular at $0$, and depend on the parameters~$\uprz$ and~$\lambda^2$. We deal with the singular behavior near $0$ and the decay as $r\to\infty$ analytically, while the C++ code is used on a compact positive interval. To be more specific, we split $(0,\infty)$ as follows:
\[
\begin{array}{rcl}
 (0,0.1] &:&
 \text{Frobenius expansions and analytic remainder bounds},\\
 \left[0.1,R\right] &:&
 \text{ODE propagation via CAPD and interval quadrature},\\
 \left[R,\infty\right) &:&
 \text{analytic comparison and Volterra tail bounds}.
\end{array}
\]
The radius~$R$ depends on the context, with our main choices being $4$ and~$16$. In this way, each problem on the half-line is reduced to finitely many computations on compact intervals, combined with analytic bounds for the singular endpoint and the infinite tail.

On the middle interval $[0.1,R]$, our C++ code computes tubuluar neighborhoods of the true ODE solutions using CAPD's approximations via Taylor polynomials. Initial data and parameters (including~$\upr$ and the spectral parameter) are represented by intervals and all computations are done with outward rounding. The integrator rigorously encloses both the truncation error and the propagated solution, so the output bounds cover the exact solution for 
\emph{every} admissible choice of parameters. We treat radial integrals the same way, using interval quadrature on each interval of the Riemann sums. Thus, we reduce the computational part of the argument to a finite list of interval enclosures and sign checks. The singularity  $r=0$ and the tail $r\to\infty$ are handled by  analytic techniques. Implementation details and the output of the C++ code are described in Section~\ref{sec:capd_certificate}.

\subsection{The vortex profile and its quantitative tails}

We start by approximating the vortex profiles~$U$ and~$a$ with rigorous error bounds, since they feed into the scalar linearized operators and into essentially every ODE solved later on. In particular, it is not enough to compute accurate numerical approximations: we also need (i) a certified enclosure of the shooting parameter, (ii) validated bounds for $U$ and~$a$ on the compact interval where we compute, and (iii) explicit tail estimates as $r\to\infty$. Section~\ref{sec_app1} provides these three pieces.

\paragraph{Step 1: enclose the shooting parameter.}
The vortex system~\eqref{numerics2} has a one-parameter family of regular local solutions, indexed by the initial slope
$\upr=U'(0)$. Imposing the boundary condition at infinity singles out one value, denoted $\uprz$, for which the solution converges to $(1,1)$. Proposition~\ref{prop:exuniqUa} gives an analytic shooting theory before any numerics: increasing~$\upr$ increases~$U$ and decreases~$a$, and $\uprz$ separates two distinct ``crossing'' behaviors. Remark~\ref{rem:J0} turns this qualitative dichotomy into a practical bracketing criterion: if
\[
 a(r_-,\upr_-)>2,
 \qquad
 U(r_+,\upr_+)>2,
\]
then $\uprz\in[\upr_-,\upr_+]$. The code checks these two inequalities for the endpoint slopes, and the analytic criterion then gives our first rigorous enclosure of~$\uprz$.

\paragraph{Step 2: start the ODE away from the singularity.}
To prove the crossing inequalities, we propagate the endpoint solutions outward. Because the vortex equations are singular at $r=0$, we do not start the numerical integration there; instead we start at the fixed radius $\rorigin=0.1$. Lemma~\ref{lem:Wberror} derives  Frobenius expansions for~$U$ and~$a$ near $0$, together with explicit remainder bounds. Evaluating the truncated series with outward-rounded interval arithmetic and adding the analytic remainder produces the initial box~\eqref{eq:vortex_start_box}, which contains the exact values of~$U$ and~$a$ at $\rorigin=0.1$ uniformly for $\upr$ in the slope interval~$J_0$. CAPD then propagates these intervals for the two endpoint slopes, verifies the inequalities in Remark~\ref{rem:J0}, and Proposition~\ref{prop:J0slope} establishes the first rigorous approximation of~$\uprz$.

\paragraph{Step 3: sharpen $\uprz$ for later computations.}
This first enclosure is enough to control the vortex globally, but we found that the later Wronskian and FGR steps require a smaller interval. Lemma~\ref{lem:Newton} achieves this by applying Newton's method to
\[
 F_{20}(\upr):=U(20,\upr)-1.
\]
For this to work, the C++ code propagates the vortex equations together with their $\upr$-derivatives; Lemma~\ref{lem:Wbderror} supplies the initial bounds for these derivatives. Since $\uprz$ is defined by the condition at infinity, the root of $F_{20}$ serves as a proxy for~$\uprz$. Thanks to the explicit tail estimates on $1-U(r)$ we can bound $|F_{20}(\uprz)|$. Combining this with a certified lower bound on $F_{20}'(\upr)=\partial_\upr U(20,\upr)$ produces the sharper enclosure~\eqref{eq:Icert}. From this point on, we keep $\upr$ as an interval parameter in~\eqref{eq:Icert}, so every integration encloses the true vortex.

\paragraph{Step 4: control the tail without identifying the leading constant.}
Finally, we need explicit bounds as $r\to\infty$, since the spectral and FGR arguments involve integrals over $(0,\infty)$. Lemma~\ref{lem:Ua} identifies the modified-Bessel asymptotics of $(U,a)$ and hence the correct decay rate, but it does not give a certified enclosure of the leading coefficient. For our purposes, a purely qualitative estimate (e.g.
$O(r^{-1/2}e^{-r})$) is not enough, because the later Volterra and tail estimates need explicit constants.

Lemma~\ref{lem_1-U2} avoids this issue by building explicit upper and lower comparison profiles. The code checks the required inequalities at  $\rfour=4$ in~\eqref{app_num_condition_a6}; the differential inequalities for the comparison profiles then prevent any subsequent crossings and extend the bounds analytically to all $r>4$. In this way, four inequalities at $r=4$ prove explicit exponential bounds on the entire tail $1-U(r)$ for all~$r\ge4$. In   Corollary~\ref{cor:vortex_tail_bridge} this then leads to pointwise and integral estimates for $1-U^2$, which are used repeatedly later in the spectral comparison, see the Volterra constructions, eigenvalue counting, and the FGR tail estimates.

\subsection{The spectral properties of \texorpdfstring{$\mathcal L_1$ and
$\mathcal L_2$}{L1 and L2}}

With the vortex profile under quantitative control, we turn to the two scalar Schr\"odinger operators $\mathcal L_1$ and~$\mathcal L_2$ whose coefficients derive from $U$ and~$a$. The supersymmetric switch~\eqref{eq:lemSd1} reduces us to the following two spectral questions about these operators:
\begin{itemize}
\item is the threshold of the continuous spectrum resonant?
\item how many eigenvalues lie in the gap $(0,1)$?
\end{itemize}
The first controls the low-frequency behavior of the distorted Fourier basis, while the second precedes any discussion of a FGR.

For the Schr\"odinger operator $\mathcal L_1$, the key input is a uniform positivity bound for its potential above the threshold~$1$. We prove this on the whole half-line by combining three regions, in the same spirit as before: near the origin, a compact middle interval, and the tail. The argument is analytic on $(0,\sqrt6)$  and on $(4,\infty)$, where the explicit tail bounds from Lemma~\ref{lem_1-U2} apply. The only place where we use numerics is the remaining compact interval $[\sqrt6,4]$, see Lemma~\ref{lem_appV1g1}. Putting the three pieces together, positivity of the quadratic form rules out eigenvalues below~$1$. A routine integration-by-parts argument shows that there is no solution that is subordinate at both endpoints at  energy~$1$, and hence there is no threshold resonance.

The Schr\"odinger operator $\mathcal L_2$ is more delicate.
At the threshold~$1$, the resonance question becomes a matching problem. One distinguished solution is selected by regularity at the origin, i.e., by the Friedrichs condition\footnote{No artificial boundary condition is imposed at the origin; the domain of the Friedrichs extension determines the regular solution.}. Another distinguished solution is picked out by subordination at infinity. A threshold resonance would occur exactly when these two solutions are linearly dependent.

We initialize the regular solution~$\Phi_{2,0}^{(0)}$ at $\rorigin=0.1$ using the  Frobenius approximation from Lemma~\ref{lem:Acheck}. Coming from~$\infty$, we use the vortex tail bounds to rewrite the threshold equation as a Volterra equation. The resulting Volterra operator is a contraction on $[16,\infty)$, which yields intervals containing the subordinate solution~$\Phi_{2,\infty}^{(0)}$ and its derivative at $R=16$, see \eqref{eq:threshold_wprime_volterra} and~\eqref{eq:threshold_infinity_Cauchy_box}.

CAPD then propagates $\Phi_{2,0}^{(0)}$ forward from $\rorigin=0.1$ and propagates $\Phi_{2,\infty}^{(0)}$ backward from $R=16$ to the matching point $r_{\rm m}=10$. The Wronskian at $r_{\rm m}$ vanishes if and only if the two solutions are linearly dependent. Since the interval in~\eqref{noresonance_wronskian} stays strictly away from zero, Lemma~\ref{Lemspec0} is valid. 

Turning to the discrete spectrum, 
Proposition~\ref{theoLT2} shows that $\mathcal L_2$ has at most one eigenvalue in $(0,1)$.  The proof applies Set\^o's two-dimensional bound~\cite{seto} to
\[
 \mathcal L_2-1=-\Delta_r-(1-U^2),
\]
where the attractive potential is $1-U^2$. Its total mass is known exactly from~\eqref{eq:vortex_mass}, so the remaining task is to bound the logarithmic double integral~\eqref{eq:seto_Lambda_definition}.

To estimate this integral, we split the domain into three parts: a compact region, a neighborhood of the origin, and the tail. On the compact region, validated enclosures of the vortex give cellwise bounds for the integrand\footnote{A ``cell'' is one subinterval in the 1D partition; each pair of cells $I_i\times I_j$ forms a rectangle in the $(r,s)$-plane for the double Riemann sum.}. We integrate the logarithmic kernel over each rectangle, using an exact formula on diagonal cells to handle the $r=s$ singularity. Near the origin we use the certified Frobenius bounds for $U$, and on the tail we rely on the explicit exponential bounds from Corollary~\ref{cor:vortex_tail_bridge}. The resulting upper bound falls strictly below the threshold in Set\^o's inequality~\eqref{eq:seto_eigenvalue_bound}. Since the eigenvalue count is an integer, this proves that $\mathcal L_2$ has at most one eigenvalue in $(0,1)$.

\subsection{The internal mode and the resonant distorted Fourier basis functions}

Having shown that $\mathcal L_2$ has at most one eigenvalue in $(0,1)$, we now prove that such an eigenvalue exists and enclose it sharply. We also need more than the eigenpair itself: the FGR coefficients are half-line integrals that involve the internal mode and the distorted Fourier basis at the resonant energy.

If the internal mode oscillates at frequency~$\lambda$, then its quadratic self-interaction has a component at frequency~$2\lambda$. This matches the continuous spectral mode with
\[
 \sqrt{1+k_\lambda^2}=2\lambda,
 \qquad
 k_\lambda:=\sqrt{4\lambda^2-1},
\]
so the relevant spectral energy is $1+k_\lambda^2=4\lambda^2$. Accordingly, Section~\ref{sec:internal_mode_dft_data} (i) locates $\lambda^2$, (ii) constructs a certified enclosure of the associated eigenfunction, and (iii) computes certified enclosures of the regular distorted Fourier basis elements $\Phi_1,\Phi_2$ from~\eqref{eq:Phi12} at this energy.

For a trial value $\mu$, let $\psi_0(\cdot,\mu)$ denote the regular solution of~\eqref{eq:evalmu} normalized at the origin. This solution exists for every $\mu$, but it belongs to $L^2_{rdr}$ (equivalently, it decays at infinity) only when $\mu$ is an eigenvalue. Once we have enclosed $\mu=\lambda^2$, the normalized eigenfunction~$\psi$ is just a scalar multiple of $\psi_0(\cdot,\lambda^2)$.

The  equation for the internal eigenfunction and the equations defining $\Phi_1,\Phi_2$ are all regular-singular at $r=0$. After factoring out the leading indicial behavior, the remaining Frobenius factors are analytic in $x=r^2$. Unlike the vortex ODE, the coefficients of these equations depend on $U$ and~$a$. Substituting the vortex series produces the triangular recurrences in~\eqref{eq:certified_frob_recurrences}, with convolution terms coming from the vortex coefficients. As a result, computing a long truncated series is not enough by itself: we also need a certified bound for the infinite coefficient tail in order to rigorously approximate  the functions and their  derivatives at $\rorigin=0.1$.

The purely analytical Lemma~\ref{lem:capd_origin_frobenius}, which is then combined with the C++ code in 
 Lemma~\ref{lem:Acheck},  provides exactly this. If we truncate a Frobenius series at degree~$M$, the code computes coefficients up to degree $M+K$ and verifies the geometric decay test
\[
 A_{M+m}\le A_M\rho^m,
 \qquad 1\le m\le K.
\]
The extra $K$ coefficients are used only to bootstrap an analytic induction: after taking absolute values, the recurrence yields a monotone majorant. The code verifies the finite checks in~\eqref{eq:certified_frob_finite_interval_check} and the bootstrap constants $\beta_\alpha<1$ in~\eqref{eq:certified_frob_bootstrap_constants}. The induction then extends the tail bound
\[
 a_j^\alpha\le T\rho^{j-M}
\]
to all $j>M$; see \eqref{eq:frob_tail_bootstrap}--\eqref{eq:frob_geometric_tail_bound}. Summing the geometric tail gives the explicit remainder estimates~\eqref{eq:value} and~\eqref{eq:deriv}. Adding these remainder bounds to the truncated polynomials yields rigorous intervals for the exact values and the derivatives at $\rorigin=0.1$.
We apply the same  mechanism uniformly to $\psi_0(\cdot,\mu)$, to $\Phi_1,\Phi_2$ (for the energies $E\in4\Lambda_{\rm FGR}$), and to the regular threshold solution used in the previous subsection.

We locate the eigenvalue~$\lambda^2$ by matching a solution that is regular at the origin to a solution that decays at infinity. The regular solution starts from the Frobenius enclosure at $\rorigin=0.1$. For each $\mu$ in an initial interval near the true value of~$\lambda^2$, the decaying solution $\psi_{\infty,\mu}$ is normalized by
\[
 \psi_{\infty,\mu}(r)=K_0\bigl(\sqrt{1-\mu}\,r\bigr)\bigl(1+o(1)\bigr),
 \qquad r\to\infty.
\]
Using the explicit vortex tail, the Volterra equation~\eqref{eq:internal_decaying_volterra} becomes contractive and produces rigorous intervals for $\psi_{\infty,\mu}(16)$ and $\psi_{\infty,\mu}'(16)$. These intervals depend on $K_0$ and $K_1$, so the code encloses the Bessel functions themselves via the integral representation~\eqref{eq:internal_Bessel_integral_representation}, finite upper and lower Riemann sums, and the analytic remainder~\eqref{eq:internal_Bessel_quadrature_tail}. This is to say that no floating-point Bessel evaluation enters.

CAPD then propagates the regular solution forward from $\rorigin=0.1$ and propagates the decaying solution backward from $R=16$ to  $r_{\rm m}=10$. Throughout, both the shooting slope $\upr$ and the trial energy $\mu$ are treated as interval parameters. The Wronskian at $r_{\rm m}$ vanishes exactly when the regular solution at $r=0$ also decays, i.e., when $\mu$ is an eigenvalue. Via iterated bisection the code shrinks the initial trial interval to $\Lambda_{\rm FGR}$. The bisection terminates when it becomes inconclusive.  The opposite  signs of the Wronskian at the endpoints of $\Lambda_{\rm FGR}$ then imply that
\[
 \lambda^2\in\Lambda_{\rm FGR}.
\]
Uniqueness comes from Proposition~\ref{theoLT2}.

For the FGR tail integral, it is not enough to know that the eigenfunction $\psi_0$ decays exponentially: we need an explicit integrable majorant for both $\psi_0$ and $\psi_0'$. Lemma~\ref{lem:psi_K0} provides such a bound for $\psi_0$ via a certified  comparison at one point and a maximum principle argument, and Lemma~\ref{lem:capd_psi_prime_tail} converts it into a bound for $\psi_0'$ using the eigenvalue ODE. With
\[
 \kappa_\lambda:=\sqrt{1-\delta_0-\lambda^2},
 \qquad \delta_0=0.0005,
\]
we obtain the explicit estimates
\[
 0<\psi_0(r)\le1.4K_0(\kappa_\lambda r),
 \qquad
 |\psi_0'(r)|\le1.4K_0(\kappa_\lambda r),
 \qquad r\ge16.
\]
These are the bounds used to control the contribution of $(16,\infty)$ in the FGR integrals.

\subsection{Certification of the Fermi Golden Rule}

Section~\ref{sec_NumVer_FGR} combines the certified vortex profile, the internal eigenpair, and the distorted Fourier basis functions at the resonant energy to finally verify the FGR condition. In our formulation, we must show that the distorted Fourier transform of the quadratic interaction $\bmQ_s(\bmY_i,\bmY_j)$ in~\eqref{eq:bmQsdef} is nonzero at the frequency~$k_\lambda$. We certify this nonvanishing for the three interactions that produce the coefficients $D_1,D_2,D_{12}$ in~\eqref{FGRODE14}, which proves Proposition~\ref{prop:FGR1}.

To carry out this check, we need rigorous  enclosures of all pieces of the distorted Fourier kernel. The Weyl solutions $\Phi_j$ normalized at $0$ determine, via the supersymmetric switch, the matrix columns $\upi_j$ in~\eqref{FGR_num_app_upi}. The kernel also depends on the connection coefficients $a_j$, which are computed from the outgoing Weyl solutions $\Psi_j$ at infinity. These are characterized by
\[
 k^{1/2}e^{-ikr}\Psi_j(r,k)\longrightarrow1,
 \qquad r\to\infty.
\]
The Wronskian identity~\eqref{def:a} writes $a_j$ in terms of $\Phi_j$ and $\Psi_j$. Applying the supersymmetric map to $\Psi_j$ gives the columns $\upsi_j$ in~\eqref{FGR_num_app_upsi_2}, which are also used to control the tail of the FGR integrals. Together, $\upi_j$ and $a_j$ determine the kernel in Proposition~\ref{propdFT}.

Because $\Psi_j$ are defined by their behavior at infinity, we first turn the outgoing condition into numerically verifiable  Cauchy data at a finite radius. Lemma~\ref{lem:capd_outgoing_start} does this at $R=16$: starting from a Poincar\'e series ansatz of order $N=10$, we compute the ODE residual and correct it using a Volterra equation whose kernel is controlled by the vortex tail bounds. A contraction argument then produces rigorous intervals for $\Psi_j(16)$ and $\Psi_j'(16)$, uniformly for $\mu\in\Lambda_{\rm FGR}$. This also gives uniform bounds for the columns $\upsi_j$ on $r\ge16$. CAPD then propagates these intervals of initial data back to the matching radius, where we compute the Wronskians that determine the coefficients~$a_j$.

With all required functions enclosed, we evaluate the radial integrals defining the FGR coefficients. As before, we divide the half-line into three regions. The main contribution comes from $[0.1,16]$, where the code uses interval quadrature to obtain the certified enclosures in~\eqref{numerics_fgr1}. The remaining pieces, $(0,0.1)$ and $(16,\infty)$, are treated as explicit error terms.

Near the origin, Lemma~\ref{lem:capd_fgr_origin} combines certified small-$r$ bounds for the internal mode and for the regular distorted Fourier columns to produce the error bound $B_{\rm origin}$. On the tail, the bounds for $\psi_0,\psi_0'$ and for $\upsi_j$ give the pointwise estimate~\eqref{eq:fgr_tail_pointwise_bound}; the remaining Bessel integral is controlled analytically. Outward-rounded evaluation, uniform for $\mu\in\Lambda_{\rm FGR}$, yields the tail bound $B_{\rm tail}$ in~\eqref{eq:Btail_definition}--\eqref{FGR_rg15_numerics}. Only after all bounds have been verified uniformly on $\Lambda_{\rm FGR}$ do we specialize to $\mu=\lambda^2$.

To obtain a lower bound for the squared norm of the vector integral, we combine the main contribution and the two error terms before taking norms. The triangle inequality in~\eqref{eq:Dhat_triangle_ineq} yields strictly positive lower bounds for the auxiliary quantities $\widehat D_\alpha^{(0)}$ in~\eqref{eq:Dhat_lower_bounds}.

Finally, Remark~\ref{rem:fgr_code_normalization} relates $\widehat D_\alpha^{(0)}$ to the coefficients $D_\alpha$ in~\eqref{FGRODE14}; they differ by an explicit negative factor. The positivity of~\eqref{eq:Dhat_lower_bounds} therefore implies
\[
 D_1<0,\qquad D_2<0,\qquad D_{12}<0.
\]
Via~\eqref{FGRODErho}, this yields $\Gamma_0,\Gamma_1>0$ in the differential inequality for $|Z_1|^2+|Z_2|^2$, i.e., the FGR nondegeneracy used to obtain nonlinear damping of the internal modes.

%
%
%
%
%
%
%
%
%
%
%
%
%
%
%
%
%
%
%
%
%
%
%
%
%
%
%
%
%
%
%
%
%
%
%
%

\section{Spectral Theory and Distorted Fourier Transform} \label{sec:spectral_fourier}

In this section, we describe the spectral and distorted Fourier theory for the linearized operator obtained by considering corotational perturbations around the degree-one vortex solution in the self-dual case:
\begin{equation}\label{eq:ML}
    \bfM := \begin{pmatrix} \bfL & 0 \\ 0 & \bfL \end{pmatrix}, 
    \quad  
    \bfL := \begin{pmatrix} L_1 & -2U' \\ -2U' & L_2 \end{pmatrix},
\end{equation}
where $L_1, L_2$ are given by \begin{align*}
    L_1 := -\Delta + \frac{(1-a_\theta)^2}{r^2} - \frac{\pr a_\theta}{r} + U^2, \qquad L_2 := -\Delta + \frac{1}{r^2} + U^2.
\end{align*}

\subsection{Domains, super-symmetry and spectra}

\begin{lem}\label{lem:2.1}
    Define $\bfL$ as a symmetric operator on $L_{\mathrm{rad}}^2(\R^2)\times L_{\mathrm{rad}}^2(\R^2)$ with domain $\calD_0\times\calD_0$, where 
    \[\calD_0:=\{f\in C_{\mathrm{rad}}^2(\R^2)\:|\: f=0 \text{\ \ near the origin and near infinity }\}\] 
    Then $\bfL$ is essentially selfadjoint, and its closure is selfadjoint with domain $\calD\times\calD$, where $\calD$ is defined in~\eqref{defcalD}. 
\end{lem}
\begin{proof}
We can write $\bfL=\bfL_0 + \mathbf{U}$ where
$\bfL_0 := \mathrm{diag}(-\Delta+\frac{1}{r^2}, -\Delta+\frac{1}{r^2})$
and $\mathbf{U} \in L^\infty(\R^2)$ is a symmetric matrix. 
Define the unitary transformation $S:L^2((0,\infty),r\,dr)\to L^2((0,\infty),dr)$ as the map $f(r)\mapsto r^{\frac12} f(r)$.
On the domain $\calD_0$,
\[S(-\Delta+r^{-2} )S^{-1}=-\frac{d^2}{dr^2}+\frac{3}{4r^2}=:L_0\]
Since $L_0 (r^{\frac12\pm1})=0$, $L_0$ is limit point at both endpoints as an operator\footnote{For a comprehensive introduction to Sturm-Liouville theory, see the treatise~\cite{GNZ}.} 
on $L^2((0,\infty),dr)$. Therefore, it is essentially selfadjoint by \cite[Theorem X.7]{RS} and \cite[Theorem 9.6]{teschl}. The domain of $L_0$ is the range of its resolvent at energy $-1$, say, which is 
\[S\calD=\Big\{ g\in L^2((0,\infty),dr)\:|\: -g''(r)+\frac{3}{4r^2} g(r)\in L^2((0,\infty),dr)\Big\} 
\]
By the Kato-Rellich theorem~\cite[Theorem 6.4]{teschl},  $\bfL$ is selfadjoint with the same domain as $\bfL_0$ in the radial subspace of $L^2\times L^2(\R^2)$.  
\end{proof}

A significant feature of working in the critical case $\lambda=1$ is the following super-symmetric factorization
\begin{equation}\label{eq:LBstarB}
 \bfL = \calB^\ast \calB ,
\end{equation}
as an identity on $\calD_0\times\calD_0$ where
\begin{equation}
 \calB = \begin{pmatrix} \pr - \frac{1-a_\theta}{r} & U \\ U & \pr + \frac{1}{r} \end{pmatrix}, \quad \calB^\ast = \begin{pmatrix} -\pr - \frac{2-a_\theta}{r} & U \\ U & -\pr \end{pmatrix}.
\end{equation}
The super-symmetric partner operator is given by
\begin{equation}\label{eq:lemSd1}
 \calL := \calB \calB^\ast = \begin{pmatrix} \calL_1 & 0 \\ 0 & \calL_2 \end{pmatrix}
\end{equation}
with $\calL_j = - \Delta + V_j$, $1 \leq j \leq 2$, and 
\begin{equation}\label{eq:lemSd2}
 \begin{aligned}
  V_1 &:= \frac{1}{r^2} + \frac{\partial_r a_\theta}{r} + \frac{(1-a_\theta)(3-a_\theta)}{r^2} + U^2 = \frac{(2-a_\theta)^2}{r^2} + \frac12 (1+U^2), \\
  V_2 &:= U^2.
 \end{aligned}
\end{equation}
We stress that $\mathcal{L}$ is diagonal, which is particularly convenient for constructing the distorted Fourier transform associated with $\bfL$. In fact, this reduces the original (coupled) matrix construction to two independent scalar transforms, one for each diagonal component of $\mathcal{L}$.

To rewrite these two-dimensional operators as self-adjoint one-dimensional operators on the half-line, we introduce the conjugated operators
\begin{equation}\label{eq:LtoH}
    \begin{aligned}
        \bfH &:= r^{\frac12} \cdot \bfM \cdot r^{-\frac12}, \\ 
        \calH &:= r^{\frac12} \cdot \calL \cdot r^{-\frac12}.
    \end{aligned}
\end{equation}
We also use the notation
\begin{equation}\label{eq:Hj}
    \bfH_j := r^{\frac12} \cdot L_j \cdot r^{-\frac12}, \qquad \calH_j := r^{\frac12} \cdot \calL_j \cdot r^{-\frac12}, \quad 1 \leq j \leq 2.
\end{equation}
Explicitly, we have the half-line operators in the ambient Hilbert space $L^2((0,\infty),dr)$
\begin{align}\label{eq:calHjdef1}
\begin{split}
    \calH_1 &= -\pr^2 - \frac{1}{4 r^2} + \frac{(2-a_\theta)^2}{r^2} + \frac12 (1+U^2), \\
    \calH_2 &= -\pr^2 - \frac{1}{4 r^2} + U^2,
\end{split}
\end{align}
initially defined on $\calD_0$. 
In analogy to Lemma~\ref{lem:2.1}, we will work with the following selfadjoint extensions of these operators.

\begin{lem}\label{lem:domains}
    The operator $\calH_1$ is a selfadjoint operator on the domain 
    \[
  \calD(\calH_1):=  \Big\{ f\in L^2((0,\infty),dr)\:|\: -f''(r)+\frac{15}{4r^2} f(r)\in L^2((0,\infty),dr)\Big\}
    \]
    and $\calH_2$ is selfadjoint on 
    $\calD(\calH_2)=r^{\frac12}H^2_{\rm{rad}}(\R^2)$. 
\end{lem}
\begin{proof}
Since
\[
  \calH_2=S(-\Delta+U^2)S^{-1}
\]
and $U^2\in L^\infty$, the operator $-\Delta+U^2$ is self-adjoint on
$H^2_{\rm rad}(\R^2)$. Hence $\calH_2$ is self-adjoint on
$r^{1/2}H^2_{\rm rad}(\R^2)$. Similarly, $\calH_1$ is a bounded symmetric
perturbation of $-\partial_r^2+15/(4r^2)$, which is limit-point at both
endpoints and has the domain stated above. Finally, the
$r^{1/2}\log r$ branch corresponds under $S^{-1}$ to $\log r$, which does not
belong to the form domain of $-\Delta+U^2$ because it has infinite local
Dirichlet energy near the origin in $\R^2$. Thus the Friedrichs extension
excludes this branch.
\end{proof}

The next proposition summarizes the basic spectral properties of $\calH_1$ and~$\calH_2$. Both operators have continuous spectrum $[1,\infty)$, and Proposition~\ref{propdFT} shows that this spectrum is purely absolutely continuous.

\begin{prop}\label{prop:lem_p2_t2}
     The operator $\mathcal{H}_1$ has no discrete spectrum and no threshold resonance. Moreover, the operator $\mathcal H_2$ has exactly one eigenvalue
in $(0,1)$ and no threshold resonance.
\end{prop}
\begin{proof}
The absence of discrete spectrum and a threshold resonance for $\mathcal{H}_1$ are proved in Lemma~\ref{lem_appV1g1}. Lemma~\ref{Lemspec0} shows that $\mathcal{H}_2$ has no threshold resonance, and Proposition~\ref{theoLT2} proves that it has at most one eigenvalue (with multiplicity), while Lemma~\ref{lem:fgr_eigenvalue_box} locates an eigenvalue in the interval~$\Lambda_{\rm FGR}$. 
\end{proof}

We first state the domains and the super-symmetric identities
\[
 \bfL=\calB^*\calB,\qquad \calL=\calB\calB^*.
\]

For $\bmf\in\calD_0\times\calD_0$, integration by parts gives
\[
 \langle\bfL\bmf,\bmf\rangle=\|\calB\bmf\|^2.
\]
It follows that the form domain of $\bfL$ equals the domain $\calD_1^2$ of
the closure of $\calB$, which is initially defined on
$\calD_0\times\calD_0$; see Section~2.3 of Teschl's book~\cite{teschl}, in
particular Theorem~2.13 on Friedrichs extensions. Indeed, the graph norm of
$\calB$ is the norm defined by the quadratic form associated with $\bfL+1$.
Since $\bfL$ is essentially selfadjoint on $\calD_0\times\calD_0$, its
closure is its Friedrichs extension. From now on, $\calB$ denotes the
closure of the initially defined operator. The first representation theorem
for closed quadratic forms (see, e.g.,~\cite[Theorem~2.1, p.~322]{kato})
then gives
\[
 \bfL=\calB^*\calB
\]
as an identity of selfadjoint operators. In particular,
\[
 \langle\bfL\bmf,\bmg\rangle
 =\langle\calB\bmf,\calB\bmg\rangle
\]
for all $\bmf\in\calD(\bfL)$ and $\bmg\in\calD_1^2$.

Via integration by
parts and Hardy's inequality in $L^2((0,\infty),dr)$, we compute, with
$Y:=L^2((0,\infty),r\,dr)$,
\[
 \calD_1^2
 :=
 \Bigl\{(f_1,f_2)\in Y\times Y:
 r^{1/2}(f_1,f_2)\in
 H_0^1((0,\infty),dr)\times H_0^1((0,\infty),dr)\Bigr\}.
\]
On the other hand, the adjoint $\calB^*$ has domain
\[
 \calD_2^2
 :=
 \Bigl\{(g_1,g_2)\in Y\times Y:
 g_1'+2g_1/r\in Y,\quad g_2'\in Y\Bigr\}.
\]
Analogously, $\calL$ is selfadjoint on
\[
 \calD(\calL):=\calD(\calL_1)\times\calD(\calL_2),
\]
and
\[
 \langle\calL\bmf,\bmg\rangle
 =\langle\calB^*\bmf,\calB^*\bmg\rangle
\]
for all $\bmf\in\calD(\calL)$ and $\bmg\in\calD_2^2$. Hence
\[
 \calL=\calB\calB^*
\]
as an identity of selfadjoint operators. Here
$\calD(\calL_j)=r^{-1/2}\calD(\calH_j)$, where the latter domains are given
by Lemma~\ref{lem:domains}. In particular,
\[
 \calD(\bfL)\subset\calD_1^2,
 \qquad
 \calD(\calL)\subset\calD_2^2.
\]

We next show that $\ker(\calB)=\ker(\calB^*)=\{0\}$ for radial functions.
This also follows from the fact that $\calB$
is Fredholm of index zero, as argued in \cite[Proposition 5]{GusSigVort}
as a consequence of the results in \cite{Stu},
and the fact that $\calB^\ast (\phi_1,\phi_2) = 0$ implies $(-\Delta+U^2)\phi_2=0$, and hence  $\phi_2=0=\phi_1$. 
We give here an independent, more direct proof that $\ker(\calB)=\ker(\calB^*)=\{0\} $
for radial functions.

\begin{lem}\label{lemB}
    The following holds: \begin{enumerate}
        \item Let $(y_1,y_2)\in \calD_1^2$ with $(y_1,y_2)\in \ker(\calB)$. Then, $(y_1,y_2)\equiv(0,0)$.
        \item Let $(y_1,y_2)\in \calD_2^2$ with $(y_1,y_2)\in \ker(\calB^*)$. Then, $(y_1,y_2)\equiv(0,0)$.
    \end{enumerate}
\end{lem}

\begin{proof}
Let us prove $(1)$ first. The equation $\calB(y_1,y_2)=0$ reads
\begin{equation}\begin{aligned}\label{eq:kerB_zero_proof}
    & y_1'(r) - b(r) y_1(r) + U(r)y_2(r)=0,\\
    & y_2'(r)+\dfrac{1}{r}y_2(r) + U(r)y_1(r)=0.
\end{aligned}\end{equation}
Since $U(r)\to 1$ and $b(r)\to 0$ exponentially as $r\to+\infty$, the associated limiting system at infinity is
\begin{align*}
    y_1'(r)+y_2(r)=0, \qquad y_2'(r)+\dfrac{1}{r}y_2(r)+y_1(r)=0.
\end{align*}
A fundamental system of solutions for the above system is explicitly given by
\begin{align*}
    (y_1,y_2)=\big(I_0(r),-I_1(r)\big), \qquad (y_1,y_2)=\big(K_0(r),K_1(r)\big).
\end{align*}
Moreover, as $r\to+\infty$ one has
\begin{align*}
    & I_0(r),\, I_1(r)=O(e^{r}),\\
    & K_0(r),\, K_1(r)=O(e^{-r}).
\end{align*}
Because $(y_1,y_2)\in \calD_1^2$, the growing mode is excluded. A standard perturbation argument then yields the exponential decay
\begin{align*}
    (y_1(r),y_2(r))=\big( O(e^{-r}),\, O(e^{-r})\big) \quad \text{as } r\to+\infty.
\end{align*}
We now multiply the first equation in \eqref{eq:kerB_zero_proof} by
$r^2\overline{y_1}$ and the second by $r^2\overline{y_2}$, subtract the
resulting identities, take real parts, and integrate over $(0,\infty)$.
After integration by parts we obtain
\[
 0=-\int_0^\infty |y_1(r)|^2r\,dr
   -\int_0^\infty b(r)|y_1(r)|^2r^2\,dr.
\]
When integrating by parts, the boundary term at $+\infty$ vanishes thanks to the exponential decay, while the boundary term at $0$ is zero since $r^{1/2}y_j\in H^1_0$ implies $r\,y_j(r)\to 0$ as $r\to 0^+$. Finally, since $b(r)\ge 0$ for all $r>0$, both terms on the right are nonpositive, hence they must vanish. In particular $y_1\equiv 0$, and \eqref{eq:kerB_zero_proof} gives $y_2\equiv 0$.

We now prove $(2)$. The only additional point needed to justify the
integration by parts is the behavior at the origin.
Let $(y_1,y_2)\in  \calD_2^2$, and define 
\[
h(r):= y_1'(r)+\dfrac{2}{r}y_1(r).
\]
By assumption, we have that $y_1,h\in Y$. With $F(r)=r^2y_1(r)$, it follows that \[
F'(r)=r^2\Big(y_1'(r)+\dfrac{2}{r}y_1(r)\Big)=r^2h(r).
\]
In particular, $F\in H^1\big((0,1),dr\big)$ and 
\[
F(r)=F(0)+\int_0^r s^2h(s)\, ds.
\]
If $F(0)\neq0$, by the continuity of $F$ at $r=0$, there exists $\delta\in(0,1)$ small enough such that (recall that $y_1(r)=r^{-2}F(r)$) \[
\int_0^\delta \vert y_1(r)\vert ^2 \, r dr \geq \dfrac{1}{2} |F(0)|^2\int_0^\delta \dfrac{1}{r^4} \, rdr  = +\infty,
\]
which contradicts the fact that $y_1\in Y$.
Therefore, $F(0)=0$ and \[
\vert y_1(r)\vert \leq \dfrac{1}{r^2}\bigg\vert \int_0^r s^2h(s)\, ds\bigg\vert \leq \dfrac{1}{r^2}\bigg( \int_0^r |h(s)|^2\, sds \bigg)^{1/2} \bigg( \int_0^r s^3\, ds\bigg)^{1/2} \lesssim \Vert h\Vert_{L^2((0,\infty),rdr)}.
\]
Thus, $\sup_{r\in(0,1)}\vert y_1(r)\vert < +\infty$, and hence $ry_1(r)\to0$ as $r\to 0+$. The second equation in $\calB^\ast(y_1,y_2)=0$ is
$y_2'=Uy_1$. Since $U(r)=O(r)$ and $y_1$ is bounded near the origin,
it follows that $y_2$ is also bounded there, whence $ry_2(r)\to0$.
The same asymptotic argument as in $(1)$, applied now to the system
$\calB^\ast(y_1,y_2)=0$, gives exponential decay of both components at
infinity.  Multiplying the two equations in
$\calB^\ast(y_1,y_2)=0$ by $r^2\overline{y_1}$ and
$r^2\overline{y_2}$, respectively, subtracting, taking real parts, and
integrating by parts gives
\[
 0=
 \left[\frac{r^2}{2}\bigl(|y_2(r)|^2-|y_1(r)|^2\bigr)\right]_{0}^{\infty}
 -\int_0^\infty b(r)|y_1(r)|^2r^2\,dr
 -\int_0^\infty |y_2(r)|^2r\,dr .
\]
The boundary term vanishes by the behavior just established at zero and
infinity. Since $b\ge0$, we obtain $y_2\equiv0$, and then
$y_2'=Uy_1$ gives $y_1\equiv0$.
\end{proof}

We can now prove Theorem~\ref{maintheo1}.

\begin{proof}[Proof of Theorem~\ref{maintheo1}]
In view of the definition \eqref{eq:ML}, it suffices to work with $\bfL$.
Lemma~\ref{lem:2.1} determines the domain and selfadjointness of~$\bfL$.

Suppose that $\mu$ is an eigenvalue of $\bfL$ with eigenfunction
$\bm\phi\in\calD(\bfL)$. Since $\bfL=\calB^*\calB$ and
$\ker(\calB)=\{0\}$ by Lemma~\ref{lemB}, we have $\mu>0$. Moreover,
\[
 \calB^*\calB\bm\phi=\mu\bm\phi.
\]
Since $\bm\phi\in\calD(\calB)$, the right-hand side belongs to
$\calD(\calB)$. It follows that
$\calB\bm\phi\in\calD(\calB\calB^*)=\calD(\calL)$ and
\[
 \calL(\calB\bm\phi)
 =\calB\calB^*(\calB\bm\phi)
 =\mu\calB\bm\phi.
\]
Thus $\calB\bm\phi$ is an eigenfunction of $\calL$ with eigenvalue $\mu$.

Conversely, suppose that $\mu$ is an eigenvalue of $\calL$ with
 eigenfunction $\bm\psi\in\calD(\calL)$. Since
$\calL=\calB\calB^*$ and $\ker(\calB^*)=\{0\}$ by Lemma~\ref{lemB}, again
$\mu>0$. Since
\[
 \calB\calB^*\bm\psi=\mu\bm\psi
\]
and $\bm\psi\in\calD(\calB^*)$, it follows that
$\calB^*\bm\psi\in\calD(\calB^*\calB)=\calD(\bfL)$ and
\[
 \bfL(\calB^*\bm\psi)
 =\calB^*\calB(\calB^*\bm\psi)
 =\mu\calB^*\bm\psi.
\]
The maps $\calB$ and $\calB^*$ are inverse to one another up to the factor
$\mu$ on the corresponding eigenspaces, and therefore preserve
multiplicity.

Proposition~\ref{prop:lem_p2_t2} shows that $\calL$ has a unique eigenvalue
$\lambda^2\in(0,1)$, with corresponding eigenfunction $(0,\psi)$, and that
its threshold is regular. The preceding correspondence shows that
$\lambda^2$ is the unique eigenvalue of $\bfL$ and has the same
multiplicity. Moreover,
\[
 \calB^*(0,\psi)=(U\psi,-\psi'),
\]
so that, in view of $\bfM=\operatorname{diag}(\bfL,\bfL)$, the vectors
$\bmY_1$ and $\bmY_2$ are eigenfunctions of $\bfM$ as stated, and the
 eigenvalue has multiplicity two. The enclosure
\[
 \lambda^2\in[0.777471875,\,0.77747375]
\]
is proved in Lemma~\ref{lem:fgr_eigenvalue_box}. Finally, the purely
absolutely continuous nature of the spectrum of $\bfM$ on $[1,\infty)$
follows from the distorted Fourier basis in Proposition~\ref{propdFT},
combined with~\cite[Theorem XIII.19]{RS4}.
\end{proof}


\subsection{Distorted Fourier transform}
We denote by 
\begin{align*}
\Phi_j(r,\freq) \quad \hbox{ and } \quad \Theta_j(r,\freq)
\end{align*}
a real-valued fundamental system of generalized eigenfunctions for the problem 
\begin{align}\label{def:Phij}
\mathcal{H}_j f = (1+\freq^2) f, 
\end{align}
with $\Phi_j\in L^2((0,1))$, and let $\Psi_j(r,\freq)$ be the Weyl
solution at infinity, normalized for $\freq>0$ by
\[
 W\bigl(\Psi_j(\cdot,\freq),\overline{\Psi_j(\cdot,\freq)}\bigr)=-2i.
\]
It satisfies $\mathcal H_j\Psi_j=(1+z^2)\Psi_j$, with
$\Psi_j(r,z)\in L^2$ near infinity for $\Im z>0$. In what follows, we
always assume that $\Phi_j$ and $\Theta_j$ are chosen so that
$W(\Theta_j,\Phi_j)=1$, where
\[
 W(f,g)=fg'-f'g.
\]
We define
\begin{align}\label{def:a}
 a_j(\freq)
 =\frac{W(\Phi_j,\overline{\Psi_j})}
        {W(\Psi_j,\overline{\Psi_j})}
 =\frac{i}{2}W(\Phi_j,\overline{\Psi_j}).
\end{align}
With these definitions at hand, we can construct the distorted Fourier transform for the entire matrix operator $\bfM$. We will perform a detailed and rigorous analysis of the asymptotic behavior and structure of these solutions in our second paper \cite{LPPSS2}. We follow the general procedure in \cite{GZ,KST,KMS}. The outcome of this analysis is summarized in the following two propositions. The first of them explains the asymptotic behavior of $\Phi_j$, $\Theta_j$, $\Psi_j$, and $a_j$. The statements are proved in \cite[Sections~2.2 and~2.3]{LPPSS2}. 

\begin{prop}\label{prop:p1_t1_and_t2}
There exists a fundamental system of real-valued solutions of $\mathcal{H}_jf=f$ at the threshold ($\freq=0$) with the following asymptotic behavior 
\begin{equation}\label{Phi_Theta_zero_Op_1and2}
\begin{aligned}
(j=1) \qquad  &    \begin{aligned}
\Phi_{1}^{(0)} (r) & = \begin{cases}
r^{5/2}+O(r^{9/2}), & r\lesssim 1,
\\ c_*r^{3/2}+O(re^{-r}),  &  r \gtrsim 1, 
\end{cases}
\\ \Theta_{1}^{(0)} (r) & = \begin{cases}
 \tfrac{1}{4}r^{-3/2} + O(r^{1/2}), & r\lesssim 1,
\\ \tfrac1{2c_*}r^{-1/2}+O(r^{-1}e^{-r}), &   r \gtrsim 1, 
\end{cases} \end{aligned}
\\ (j=2) \qquad  &  \begin{aligned} \Phi_{2}^{(0)} (r) & = \begin{cases}
r^{1/2}+O(r^{5/2}), & r\lesssim 1,
\\ \widetilde{c}_{1,*}r^{1/2}\log(r)+\widetilde{c}_{2,*}r^{1/2}+O(e^{-r}\log(r)),  &  r \gtrsim 1,
\end{cases}
\\ \Theta_{2}^{(0)} (r) & = \begin{cases}
 -r^{1/2}\log(r) + O(r^{5/2}\log(r)), & r\lesssim 1,
\\ \widetilde{c}_{3,*}r^{1/2}\log(r)+\widetilde{c}_{4,*}r^{1/2}+O(e^{-r}\log(r)), &   r \gtrsim 1.
\end{cases} 
\end{aligned}
\end{aligned}
\end{equation}
for some $c_*\neq0$, where $\widetilde{c}_{1,*}\widetilde{c}_{4,*}-\widetilde{c}_{2,*}\widetilde{c}_{3,*}=1$ and $\widetilde{c}_{1,*}\neq0$. Also, they are normalized so that \[
W[\Theta_{j}^{(0)},\Phi_{j}^{(0)}]=1.
\]
In the case of positive energy $\freq>0$ for all $r>0$ with  $\freq r  \lesssim 1$, we have the expansion
\begin{equation}\label{Phi_xi_Op_1and2}
\begin{aligned}
\Phi_1(r,\freq) & = \Phi_{1}^{(0)}(r) + r^{3/2} \sum_{j\geq 1} (\freq r)^{2j} \Phi_{1,j}(r), 
\\\Phi_2(r,\freq) & = \Phi_2^{(0)}(r) + \sqrt{r} \sum_{j\geq 1} (\freq r)^{2j} \Phi_{2,j}(r),
\end{aligned}
\end{equation}
which converges absolutely. 

For all $r>0$, $\freq \geq 0$ with $\freq  r \gtrsim 1$, a Weyl-Titchmarsh function of $\mathcal{H}_j$ at infinity is given by
\begin{align}
\Psi_1(r,\freq) & = \freq^{-1/2}e^{i\freq r}\sigma_{1,\infty}\left(\freq r,r\right), & & \freq r\gtrsim 1, \label{Psi1Op1_3}
\\ \Psi_{2} (r,\freq) & = \freq^{-1/2}e^{i\freq r}\sigma_{2,\infty}\left(\freq r,r\right), & &  \freq r\gtrsim 1, \label{Fou25}
\end{align}
where $\sigma_{j,\infty}$ satisfy suitable symbol type bounds, see~\cite[Section 2]{LPPSS2} for details. Moreover, the connection coefficients satisfy
\begin{align}
|a_1(\freq)|  &  \simeq \dfrac{1}{\freq \langle \freq \rangle}, \label{SMOp1_3}
\qquad\quad |a_2(\freq)|\simeq
\begin{cases}
 1+|\log\freq|,&0<\freq\le1,\\
 1,&\freq\ge1.
\end{cases}
\end{align}
Furthermore, we have symbol type bounds for all $\Phi_{1,j}$, $\Phi_{2,j}$, $\sigma_{1,\infty}$, $\sigma_{2,\infty}$, $a_1$ and $a_2$.
\end{prop}

\begin{rem}\label{rem:c1*notzero}
The fact that $\widetilde{c}_{1,*}\neq0$, is precisely the statement in Proposition~\ref{prop:lem_p2_t2} that $\mathcal{H}_2$ has no threshold resonance. It is certified by an interval enclosure of a Wronskian, see Subsection~\ref{sec:num_verif_no_resonance}.
\end{rem}

We now state, without proof, the distorted Fourier transform associated with $\bfM$, whose construction is given in \cite[Theorem~1.2]{LPPSS2}.

\begin{prop} \label{propdFT}
The distorted Fourier transform for the selfadjoint matrix operator $\bfM$, applied to a radial function $\bmg \in \calD(\mathbf{M})\subset L^2_{r \ud r}(\bbR_+;\bbC^4)$, is given by 
\begin{equation}
    \wtilcalF[ \bmg ](\freq) \equiv \widetilde{\bmg}(\freq) := \int_0^\infty \overline{\bfE(r,\freq)}^{t}  \bmg(r) r\, \ud r, 
\end{equation}
where $\bfE(r,\freq)$ is  
\begin{equation}
    \begin{aligned}
        \bfE(r,\freq) :=  
        \begin{pmatrix} 
            E(r,\freq) & 0_{2\times 2 } \\ 0_{2 \times 2} & E(r,\freq)
        \end{pmatrix}  ,
    \end{aligned}
\end{equation}
{\setlength{\arraycolsep}{2pt}\small
\begin{align*}
E(r,\freq):= \frac{i}{\sqrt{2\pi r}\,\jk}\begin{pmatrix}
 a_1^{-1}(\freq)\,\begin{aligned}[t]
  &-\partial_r\Phi_1(r,\freq)-\tfrac{1}{2r}\Phi_1(r,\freq)\\
  &-b(r)\Phi_1(r,\freq)
 \end{aligned}
 & a_2^{-1}(\freq)\,U(r)\,\Phi_2(r,\freq) \\
 a_1^{-1}(\freq)\,U(r)\,\Phi_1(r,\freq)
 & a_2^{-1}(\freq)\,\bigl(-\partial_r\Phi_2(r,\freq)+\tfrac{1}{2r}\Phi_2(r,\freq)\bigr)
\end{pmatrix}.
\end{align*}}
Here $a_j(\freq)$, $\Phi_j(r,\freq)$ and $b(r)$ are as in \eqref{def:a}, \eqref{def:Phij} and \eqref{eq:bL1L2}, respectively.
The associated inverse Fourier transform is
\begin{equation}
    \wtilcalF^{-1}[\bmh](r) := \int_0^\infty \bfE(r,\freq)  \bmh(\freq) \freq\, \ud \freq  .
\end{equation}
 Then we have 
\begin{equation}
    \wtilcalF[ \bfM \bmg ](\freq) = (1+\freq^2) \widetilde{\bmg}(\freq),
\end{equation}
the Fourier inversion identity is given by
\begin{equation}
    \bfP_c = \wtilcalF^{-1} \circ \wtilcalF,
\end{equation}
and the Plancherel identity 
\begin{equation}
    \bigl\| \widetilde{\bmg} \bigr\|_{L^2_{\freq \ud \freq}(\bbR_+;\bbC^4)} = \bigl\| \bfP_c\,\bmg \bigr\|_{L^2_{r \ud r}(\bbR_+;\bbC^4)}
\end{equation}
holds for radial functions in $ \calD(\mathbf{M})$ and then extends to $L^2$. In particular, the restriction of $\wtilcalF$ to $\operatorname{Ran}\bfP_c$ is unitary onto $L^2(\freq\,d\freq;\bbC^4)$, equivalently
\[
\wtilcalF\circ\wtilcalF^{-1}=I\quad\text{on }L^2(\freq\,d\freq;\bbC^4).
\] 
\end{prop}
The distorted Fourier transform is crucial for the linear and nonlinear parts of our analysis (see \cites{LPPSS2}{LPPSS3}). On the linear side, it is convenient for various dispersive estimates, in particular weighted estimates. On the nonlinear side, it is crucial for understanding the radiation mechanism of the discrete component and for formulating and verifying the Fermi Golden Rule.



\section{On the Fermi Golden Rule}\label{sec:FGR}

In this section, we focus on the ODE for $Z_j$ and recall from \cite{LPPSS3} the main condition for the FGR. We recall the system of evolution equations for the radiation term and for the internal mode components
\begin{equation}
    \left\{ \begin{aligned}
        (\pt^2 + \bfM) \bmu &= \bfP_c \bmN, \\ 
        (\pt^2 + \lambda^2) z_1 &= \langle \bmY_1, \bmN \rangle, \\ 
        (\pt^2 + \lambda^2) z_2 &= \langle \bmY_2, \bmN \rangle.
    \end{aligned} \right.
\end{equation}
Here $\bmY_j$, $j=1,2$, are two independent orthonormal eigenvectors associated with the unique eigenvalue of $\bfM$, viz. 
\begin{align}\label{eq:internal_eigenvalue_convention}
\bfM \bmY_j = \lambda^2 \bmY_j, \qquad \lambda^2 [0.777471875,0.77747375],
\end{align}
see Theorem~\ref{maintheo1}. 
The interval displayed above is denoted by $\Lambda_{\rm FGR}$ in Lemma~\ref{lem:fgr_eigenvalue_box}. Thus,  $\Lambda_{\rm FGR}$ is an interval for the eigenvalue $\lambda^2$, while $\lambda>0$ denotes the corresponding internal frequency.

For technical reasons, the C++ code uses a different normalization of the internal mode. In fact, let $\psi_0$ denote the positive eigenfunction of $\mathcal L_2$
at the eigenvalue $\lambda^2$, normalized at the origin by
\[
  \psi_0(0)=1,\qquad \psi_0'(0)=0.
\]
Let
\begin{equation}\label{eq:bmYjbd1}
  \bmY^{(0)}_1=(U\psi_0,-\psi_0',0,0)^T,\qquad
  \bmY^{(0)}_2=(0,0,U\psi_0,-\psi_0')^T,
\end{equation}
and define
\[
  \nu_\psi^2
  :=\|\bmY_j^{(0)}\|_{L^2(r\,dr)^4}^2
  =\int_0^\infty
    \bigl(U^2\psi_0^2+(\psi_0')^2\bigr)r\,dr
  =\lambda^2\int_0^\infty\psi_0^2r\,dr.
\]
The orthonormal eigenvectors are then
\[
  \bmY_j:=\nu_\psi^{-1}\bmY_j^{(0)},\qquad j=1,2.
\]
We denote the profiles for the continuous component and the internal modes, respectively, by
\begin{equation}
    \boldsymbol{\mathfrak{f}} := e^{-it\sqrt{\bfM}} (2i\sqrt{\bfM})^{-1}(\partial_t + i\sqrt{\bfM}) \bmu \qquad \hbox{and} \qquad Z_j(t) := e^{-i\lambda t} (2i\lambda)^{-1} (\pt + i \lambda) z_j, \quad j = 1, 2.
\end{equation}

\subsection{The ODE and Fermi coefficient}

In this subsection, we record the leading-order ODE for $Z_j$. Recall that $z_j$ satisfies the equation
\begin{align}\label{FGRODE}
\dot{Z}_j(t) = \dfrac{1}{2i\lambda} e^{-i\lambda t} \big\langle \bmN, \bmY_j \big\rangle,
\end{align}
for a suitable nonlinearity $\bmN$. The Hamiltonian structure of the original equations plays an important role in the radiative damping mechanism that leads to the decay of $Z_j$. Here $\partial^\ast=-\partial_r-r^{-1}$ denotes the adjoint of $\partial_r$ in $L^2((0,\infty),r\,dr)$. Without going into details, we note that the most relevant part of the cubic Hamiltonian is the following:
\begin{align}\label{eq:FGR_cubic_hamiltonian}
    \mathcal{H}_2 & := - \int \biggl( -\frac12 U \alpha^3 - \frac12 U \beta^2 \alpha  
  + \partial^\ast(\mu\beta) \alpha - \mu \beta' \alpha
  + b \zeta \alpha^2 + b \zeta \beta^2
  - U \zeta^2 \alpha - U \mu^2 \alpha \biggr) \, rdr.
\end{align}
The corresponding variational derivative is denoted by 
\begin{align}\label{eq:bmQbddef1}
    \bmQ(\bmr) = -\frac{ \delta \mathcal{H}_2}{\delta \bmr}.
\end{align}
We introduce the symbol $\calH_2$ only to match the notation in \cite{LPPSS3}; from this point on we work with $\bmQ$ instead. This $\calH_2$ is unrelated to the scalar operator $\calH_2$ from~\eqref{eq:calHjdef1}. We denote the associated bilinear form by $\bmQ_s$:
\begin{equation}\label{eq:bmQsdef}
    \bmQ_s(\underline{\bmr},\underline{\bmr}')  = \frac{1}{2}\big[\bmQ(\underline{\bmr}+\underline{\bmr}') - \bmQ(\underline{\bmr})-\bmQ(\underline{\bmr}')\big].
\end{equation}
This is the bilinear expression that appears in the Fermi Golden Rule coefficients below; see \eqref{FGRODE14}. In Lemma~\ref{lem:capd_fgr_tail_upper_bound}
we will use the explicit form of \eqref{eq:bmQsdef} after substituting the
eigenvectors \eqref{eq:bmYjbd1}.
A direct (but technical) computation shows that, to leading order, $\boldsymbol{\mathfrak{f}}$ decomposes as
\begin{align}\label{deffF}
\boldsymbol{\mathfrak{f}}:= \boldsymbol{\mathfrak{f}}_{11}^F+2\boldsymbol{\mathfrak{f}}_{12}^F+\boldsymbol{\mathfrak{f}}_{22}^F+\mathrm{Remainders}, \quad \hbox{ where } \quad   \boldsymbol{\mathfrak{f}}_{ij}^F(t) & := Z_i(t)Z_j(t) \big( \Lambda_{ij} -i \Gamma_{ij} \big),
\end{align}
with
\begin{align}\label{defLamGam}
  \Lambda_{ij} & := e^{-it(\sqrt{\bfP_c\bfM}-2\lambda)}\dfrac{1}{2\sqrt{\bfM}}  \mathrm{p.v.}\frac{1}{\sqrt{\bfP_c\bfM}-2\lambda} \mathbf{P}_c\bmQ_s(\bmY_i,\bmY_j),
  \\ \Gamma_{ij} & :=  \dfrac{\pi}{4\lambda} \delta(\sqrt{\bfP_c\bfM}-2\lambda) \mathbf{P}_c\bmQ_s(\bmY_i,\bmY_j).
\end{align}
\begin{remark}
    The precise sense in which the terms labeled ``Remainders" in the decomposition of $\boldsymbol{\mathfrak{f}}$ in \eqref{deffF} can be handled is described in Sections~9 and~10 of \cite{LPPSS3}. The choice of notation, which seems unmotivated here, is also explained in \cite{LPPSS3}. The terms labeled $\boldsymbol{\mathfrak{f}}_{ij}^F$ here are the terms $\boldsymbol{\mathfrak{f}}_{ij}^F$ in Lemma~9.3 of \cite{LPPSS3}. In comparison with $\bmQ$ in equation~(9.60) of \cite{LPPSS3}, the expression $\bmQ$ in \eqref{eq:bmQbddef1} is missing the contribution of $\calH_{G,2}$ from \cite{LPPSS3}. This is justified because these contributions involve $\Theta_0^{(2)}$ in the notation of \cite{LPPSS3}, and in view of the form of $\bmY_j$  and equation (9.15) in \cite{LPPSS3}, $\Theta_0^{(2)}(\bmY_j,\bmY_k)=0$ for $j,k=1,2$.
\end{remark}
By definition, we can also write \begin{align}
\begin{split}
\bmu  = e^{it\sqrt{\bfP_c\bfM}} \boldsymbol{\mathfrak{f}} + e^{-it\sqrt{\bfP_c\bfM}} \overline{\boldsymbol{\mathfrak{f}}} .
\end{split}
\end{align}
In \cite{LPPSS3}, following similar ideas to \cite{SW} (see also \cite{LP22}), we prove the following result, which is the leading-order ODE for the internal mode. We refer to \cite[Sections 9, 10]{LPPSS3} for a detailed proof and estimates of the ``remainders" in \eqref{FGRODE13}, as well as a more precise statement of the following important proposition. Only the dissipative terms contribute to the evolution of
$|Z_1|^2+|Z_2|^2$: the conservative cubic terms
$i\mathcal P_j(Z,\overline Z)$ cancel in the corresponding energy
identity, while the remainders are of higher order.

\begin{prop}\label{lem4}
In the notation of \cite[Sections~9 and~10]{LPPSS3}, the projection of the
normal-form equation onto the internal modes has the following schematic form:
\begin{align}\label{FGRODE13}
\begin{split}
\dot{Z}_1
&=D_1|Z_1|^2Z_1+2D_{12}|Z_2|^2Z_1
  +C_{12}\overline{Z_1}Z_2^2+i\mathcal P_1(Z,\overline Z)+\mathcal R_1,\\
\dot{Z}_2
&=D_2|Z_2|^2Z_2+2D_{12}|Z_1|^2Z_2
  +C_{12}\overline{Z_2}Z_1^2+i\mathcal P_2(Z,\overline Z)+\mathcal R_2.
\end{split}
\end{align}
The conservative cubic terms satisfy
\[
 \operatorname{Re}\sum_{j=1}^2
 \overline Z_j\,i\mathcal P_j(Z,\overline Z)=0.
\]
Here $\mathcal P_j$ and $\mathcal R_j$ denote, respectively, the conservative
cubic polynomials and the higher-order remainder terms defined and estimated
in \cite[Sections~9 and~10]{LPPSS3}. 
The dissipative coefficients are
\begin{align}\label{FGRODE14}
\begin{split}
D_1 & := -\dfrac{\pi}{4\lambda^2} \big\langle \delta(\sqrt{\bfP_c\bfM}-2\lambda) \bmQ_s\big(\bmY_1, \bmY_1), \bmQ_s\big(\bmY_1, \bmY_1) \big\rangle,
\\
D_{12} & := -\dfrac{\pi}{4\lambda^2} \big\langle \delta(\sqrt{\bfP_c\bfM}-2\lambda) \bmQ_s\big(\bmY_2, \bmY_1), \bmQ_s\big(\bmY_2, \bmY_1) \big\rangle,
\\
C_{12} & := -\dfrac{\pi}{4\lambda^2} \big\langle \delta(\sqrt{\bfP_c\bfM}-2\lambda) \bmQ_s\big(\bmY_2, \bmY_2), \bmQ_s\big(\bmY_1, \bmY_1) \big\rangle,
\\
D_2 &:= -\dfrac{\pi}{4\lambda^2} \big\langle \delta(\sqrt{\bfP_c\bfM}-2\lambda) \bmQ_s\big(\bmY_2, \bmY_2), \bmQ_s\big(\bmY_2, \bmY_2) \big\rangle.
\end{split}
\end{align}
Equivalently, in the scalar distorted Fourier variable $k$, this is the shell
$1+k^2=4\lambda^2$, i.e., $k=k_\lambda:=\sqrt{4\lambda^2-1}$.
\end{prop}
Note that \eqref{FGRODE14} explicitly shows that
$D_1,D_2,D_{12}\leq0$. Moreover, $C_{12}$ is controlled
analytically by Cauchy--Schwarz for the positive spectral measure:
\[
  |C_{12}|^2\leq D_1D_2.
\]
Thus no separate numerical hypothesis on $C_{12}$ is required beyond
the certified strict negativity of $D_1$ and $D_2$.

In the third part of this series \cite[Lemma~9.9, equation (9.67)]{LPPSS3}, we start from~\eqref{FGRODE13} and derive the following differential inequality for $\mathfrak{z}(t):=|Z_1|^2+|Z_2|^2$:
\begin{align}\label{FGRODErho}
- \Gamma_0 \, \mathfrak{z}^2(t)\leq \dot{\mathfrak{z}}(t) +\mathrm{remainders} \leq - \Gamma_1 \, \mathfrak{z}^2(t), \qquad 
\end{align}
for two constants $\Gamma_0,\Gamma_1\geq0$, such that
\[
 D_1,D_2,D_{12}<0
 \quad\Longrightarrow\quad
 \min(\Gamma_0,\Gamma_1)>0.
\]
Thus, to obtain dissipative dynamics from~\eqref{FGRODErho}, it suffices to certify via interval enclosures that
\begin{align}\label{FGRnummain}
D_1,D_2,D_{12} < 0. 
\end{align}
This is one of the main results of this work and will be proved in Section~\ref{sec_NumVer_FGR}. For later comparison with the computer-assisted calculation, let
$D_1^{(0)},D_2^{(0)},D_{12}^{(0)}$, and $C_{12}^{(0)}$ denote the
four expressions in \eqref{FGRODE14} with every occurrence of
$\bmY_\ell$ replaced by the origin-normalized vector
$\bmY_\ell^{(0)}$.  Since
\[
  \bmY_\ell^{(0)}=\nu_\psi\bmY_\ell
\]
and $\bmQ_s$ is bilinear, we have
\[
  D_j^{(0)}=\nu_\psi^4D_j,
  \qquad j=1,2,12,
  \qquad
  C_{12}^{(0)}=\nu_\psi^4C_{12}.
\]
Equivalently,
\[
  D_j=\nu_\psi^{-4}D_j^{(0)},
  \qquad
  C_{12}=\nu_\psi^{-4}C_{12}^{(0)}.
\]

\begin{proposition}\label{prop:FGR1}
Let $D_1,D_2,D_{12}$ be the Fermi Golden Rule coefficients defined
in \eqref{FGRODE14}, formed using the orthonormal eigenvectors
$\bmY_1,\bmY_2$.  Then
\[
  D_1<0,\qquad D_2<0,\qquad D_{12}<0.
\]
More precisely, let 
\[
  k_\lambda:=\sqrt{4\lambda^2-1},
  \qquad
  \widehat D_j^{(0)}
  :=-\frac{k_\lambda}{4\lambda}D_j^{(0)},
  \qquad j\in\{1,2,12\}.
\]
The C++ code proves
\begin{equation}\label{eq:Dhat_lower_bounds}
 \widehat D_1^{(0)}>0.017297889,\qquad
 \widehat D_2^{(0)}>0.047248801,\qquad
 \widehat D_{12}^{(0)}>0.022940050.
\end{equation}
\end{proposition}

\begin{rem}\label{rem:fgr_code_normalization}
The quantities certified by the C++ code are positive scalar
multiples of the negatives of the corresponding coefficients in
\eqref{FGRODE14}.
For
$\alpha\in\{1,2,12\}$, let $(i,j)=(1,1),(2,2),(1,2)$, respectively,
and let $\mathcal A_{ij}\in\bbC^2$ denote the corresponding vector-valued
radial integral after the common scalar factor in the distorted Fourier
kernel has been removed, as in Lemma~\ref{lem:capd_fgr_origin} below.

We have
\[
 \int_0^\infty
 \delta\bigl(\sqrt{1+k^2}-2\lambda\bigr)F(k)\,k\,dk
 =2\lambda F(k_\lambda),
 \qquad
 k_\lambda=\sqrt{4\lambda^2-1}.
\]
At $k=k_\lambda$, the prefactor in Proposition~\ref{propdFT} is
\[
 \frac{i}{2\lambda\sqrt{2\pi r}}.
\]
After integration against $r\,dr$, the common scalar factor outside
$\mathcal A_{ij}$ is therefore
\[
 \frac{1}{2\lambda\sqrt{2\pi}},
\]
whose squared modulus is $1/(8\pi\lambda^2)$.  Consequently,
\[
\begin{aligned}
 D_\alpha^{(0)}
 &=-
 \frac{\pi}{4\lambda^2}\,
 2\lambda\,
 \frac{1}{8\pi\lambda^2}
 \|\mathcal A_{ij}\|_{\bbC^2}^2\\
 &=-
 \frac{1}{16\lambda^3}
 \|\mathcal A_{ij}\|_{\bbC^2}^2.
\end{aligned}
\]
The quantity computed by the C++ code is
\[
 \widehat D_\alpha^{(0)}
 =
 \frac{k_\lambda}{64\lambda^4}
 \|\mathcal A_{ij}\|_{\bbC^2}^2.
\]
It follows that
\[
 \widehat D_\alpha^{(0)}
 =-\frac{k_\lambda}{4\lambda}D_\alpha^{(0)},
 \qquad
 D_\alpha
 =-\frac{4\lambda}{k_\lambda\nu_\psi^4}
  \widehat D_\alpha^{(0)}.
\]
Hence the certified positivity of $\widehat D_\alpha^{(0)}$ implies the
negativity of the coefficients in~\eqref{FGRODE14}.
\end{rem}


\section[Approximating the vortex]{Approximating the vortex profiles \texorpdfstring{$(U,a)$}{(U,a)} } \label{sec_app1}

This section is devoted to studying   the degree-one magnetic vortex in the self-dual case. To start, we establish asymptotics and estimate the value of $U'(0)$ using a combination of classical analysis together with rigorous numerical computations.  By the Bogomolny equations  $U$ and $a$ satisfy the first order system
\begin{align}\label{numerics2}
U'(r) &= \dfrac{1-a(r)}{r} U(r), \qquad a'(r) = \frac{r}{2}(1-U^2(r)).
\end{align}
At $r=0$ one has 
\begin{equation}\label{Uaexp0}
    U(r)=\uprz r(1+o(1)), \qquad a(r)=\frac{r^2}{4}(1+o(1)).
\end{equation}  
Here $\uprz>0$ is known to exist and to be unique with the property that $U(r)\to1$ and $a(r)\to1$ as $r\to\infty$ \cite{bookJT}. We give a self-contained proof of this fact based on the ODE~\eqref{numerics2}.

\begin{prop}
    \label{prop:exuniqUa}
    There exists a unique value of $\uprz>0$ so that \eqref{numerics2} has a
    solution for which \eqref{Uaexp0} holds as $r\to0+$ and such that $(U(r),a(r))\to (1,1)$ as $r\to\infty$. 
\end{prop}
\begin{proof}
It is standard to solve the initial value problem at the singular point $r=0$.
For $\upr\geq0$, write
\[
    U(r,\upr)=\upr rW(r,\upr),
    \qquad
    a(r,\upr)=\frac{r^2}{4}b(r,\upr).
\]
Then \eqref{numerics2} is the same as
\begin{equation}\label{eq:Wbsys}
      W'=-\frac r4 bW,
    \qquad
    b'=\frac2r(1-b)-2\upr^2rW^2,
    \qquad W(0,\upr)=b(0,\upr)=1.
\end{equation}
Equivalently,
\begin{equation}\label{eq:Ua_local_integral}
\begin{aligned}
 W(r,\upr)
 &=\exp\left(-\frac14\int_0^r s b(s,\upr)\,ds\right),\\
 b(r,\upr)
 &=1-\frac{2\upr^2}{r^2}\int_0^r s^3W(s,\upr)^2\,ds .
\end{aligned}
\end{equation}
For completeness, fix $L>0$ and consider $0\leq\upr\leq L$.  On
$C([0,\rho])\times C([0,\rho])$, the right-hand side of
\eqref{eq:Ua_local_integral} maps  
\[
 \|W-1\|_\infty+\|b-1\|_\infty\leq1
\]
into itself if $\rho$ is small enough.  The difference of the images of two pairs is bounded by $C_L\rho^2$ times the difference of the pairs.  Taking $C_L\rho^2<1$ yields a contraction.  Since $L$ is arbitrary, this proves local existence and uniqueness for every $\upr\geq0$, uniformly on compact sets of
parameters.  The same integral equations, differentiated with respect to
$r$ and $\upr$, give smooth dependence on $(r,\upr)$.  Expanding
\eqref{eq:Ua_local_integral} at zero gives
\begin{equation}\label{eq:Ua_local_expansion_proof}
\begin{aligned}
 U(r,\upr)&=\upr r-\frac{\upr}{8}r^3+O(r^5),\\
 a(r,\upr)&=\frac{r^2}{4}-\frac{\upr^2}{8}r^4+O(r^6),
\end{aligned}
\end{equation}
locally uniformly in $\upr$.  Starting at any positive radius, Picard iteration therefore gives a unique maximal continuation on
$0<r<R(\upr)$, where $R(\upr)\leq\infty$.  Continuous dependence implies the
lower semicontinuity
\[
 R(\upr_0)\leq\liminf_{\upr\to\upr_0}R(\upr).
\]
Notice also that $U(r,\upr)>0$ for $\upr>0$ and every
$0<r<R(\upr)$, cf.~\eqref{eq:Ua_local_integral}. 

As usual, we analyze the dependence on the shooting parameter via  the variational equations (or equations of variations).  Set
\[
 V(r,\upr):=\partial_{\upr}U(r,\upr),
 \qquad B(r,\upr):=-\partial_{\upr}a(r,\upr).
\]
The variational equations are
\begin{equation}\label{eq:v_ode}
\begin{aligned}
 V'&=\frac{1-a}{r}V+\frac Ur B,\\
 B'&=rUV,
\end{aligned}
\end{equation}
and by \eqref{eq:Ua_local_expansion_proof}
\[
 V(r,\upr)=r+O(r^3),
 \qquad
 B(r,\upr)=\frac{\upr}{4}r^4+O(r^6).
\]
We now claim that for $\upr>0$, 
\begin{equation}\label{eq:shooting_monotonicity}
 \partial_{\upr}U(r,\upr)>0,
 \qquad
 \partial_{\upr}a(r,\upr)<0
\end{equation}
for as long as the solution exists.  In particular, increasing the shooting
parameter raises $U$ and lowers $a$ on every common interval of existence.
Indeed, by the expansions above we have $V>0$ and $B>0$ for small $r>0$. If $V$ had a first zero at some $\rVzero >0$, then $V>0$ on $(0,\rVzero )$ would imply
$B'=rUV>0$ there and hence $B(\rVzero )>0$.  Evaluating the first line of~\eqref{eq:v_ode} at $\rVzero $ gives $V'(\rVzero )=\frac{U(\rVzero )}{\rVzero }B(\rVzero )>0$, contradicting the
choice of $\rVzero $ as the first zero.  Hence $V>0$, and therefore $B>0$, on every
interval of existence.

We now analyze the possible  crossings of the level $1$ by $a$ and~$U$.  Thus, suppose 
that $U(\rcrossingz ,\upr_0)=1$, where $\rcrossingz >0$ is minimal.  Then $a$ cannot
have reached $1$ at an earlier radius.  Indeed, if $a(\rcrossingo ,\upr_0)=1$ first,
with $\rcrossingo <\rcrossingz $, then $U(\rcrossingo ,\upr_0)<1$ and
\[
 U'(\rcrossingo ,\upr_0)=0,
 \qquad
 U''(\rcrossingo ,\upr_0)
 =-\frac12\bigl(1-U(\rcrossingo ,\upr_0)^2\bigr)U(\rcrossingo ,\upr_0)<0.
\]
Immediately after $\rcrossingo $ one therefore has $a>1$ and $U'<0$.  As long as
$U<1$, one also has $a'>0$, so these inequalities persist and $U$ can never
reach $1$, a contradiction.  Simultaneous equality $U=a=1$ is also
impossible, since uniqueness at the positive radius would make the solution
identically equal to $(1,1)$.  Consequently, for $0<r<\rcrossingz $
\[
 0<U(r,\upr_0)<1,\qquad 0<a(r,\upr_0)<1,\qquad a(\rcrossingz ,\upr_0)<1,
\]
and the first $U$--crossing is transverse:
\begin{equation}\label{eq:U_cross_transverse}
 U'(\rcrossingz ,\upr_0)=\frac{1-a(\rcrossingz ,\upr_0)}{\rcrossingz }>0.
\end{equation}
The same argument with the roles reversed shows that, at a first
$a$--crossing,
\begin{equation}\label{eq:a_cross_transverse}
 U(\rcrossingz ,\upr_0)<1,\qquad
 a'(\rcrossingz ,\upr_0)=\frac{\rcrossingz }{2}\bigl(1-U(\rcrossingz ,\upr_0)^2\bigr)>0.
\end{equation}
It also shows that a solution cannot have both a $U$--crossing and an
$a$--crossing: after a first $U$--crossing, $a$ has a strict maximum below
$1$ and then decreases, whereas after a first $a$--crossing, $U$ has a strict
maximum below $1$ and then decreases.

The monotonicity \eqref{eq:shooting_monotonicity} now gives the crossing dichotomy: {\em  If the solution for $\upr_0>0$ first reaches $U=1$ at $\rcrossingz $, then
every $\upr_1>\upr_0$ reaches $U=1$ at  $\rcrossingo <\rcrossingz $. Similarly, if the solution for $\upr_0$ first reaches
$a=1$ at $\rcrossingz $, then every $0<\upr_1<\upr_0$ reaches $a=1$ at a radius
strictly smaller than $\rcrossingz $.}

To see this, suppose that the
$\upr_1$--solution has not crossed $U=1$ before $\rcrossingz $.  On its common
interval with the $\upr_0$--solution, \eqref{eq:shooting_monotonicity} gives
$a(r,\upr_1)<a(r,\upr_0)<1$.  Before a $U$--crossing we therefore have
$0<U(r,\upr_1)<1$ and $0<a(r,\upr_1)<1$, so the solution remains bounded and
continues up to $\rcrossingz $.  But then
$U(\rcrossingz ,\upr_1)>U(\rcrossingz ,\upr_0)=1$, and the intermediate value theorem gives a
crossing before $\rcrossingz $.  The $a$--crossing assertion is proved similarly.

Define the {\em crossing sets}
\begin{align*}
 \calC_U&:=\{\upr\geq0:\ U(r,\upr)=1
       \text{ for some }0<r<R(\upr)\},\\
 \calC_a&:=\{\upr\geq0:\ a(r,\upr)=1
       \text{ for some }0<r<R(\upr)\}.
\end{align*}
By the preceding, these sets are disjoint, and we claim that they are relatively open in $[0,\infty)$.  For a positive
parameter in either crossing set, choose $r>0$ immediately before and after
the first crossing.  The strict inequalities at these two radii given by
\eqref{eq:U_cross_transverse} or \eqref{eq:a_cross_transverse} persist for
nearby parameters by lower semicontinuity of $R$ and continuous dependence. The intermediate value theorem then gives a crossing for every nearby
parameter.  At the endpoint $\upr=0$ one has
\[
 U(r,0)=0,\qquad a(r,0)=\frac{r^2}{4},
\]
so $a$ crosses $1$ transversely at $r=2$.  The same argument gives a relative
neighborhood $[0,\epsilon)$ contained in $\calC_a$.  This proves the claim.

The set $\calC_a$ is nonempty by the explicit solution at $\upr=0$.  The set
$\calC_U$ is nonempty as well.  Indeed, for as long as $0\leq U\leq1$,
\[
 a(r)\leq\frac{r^2}{4}
\]
and hence, for $0<r_\varepsilon<r$ before the first $U$--crossing,
\[
 \log\frac{U(r)}{U(r_\varepsilon)}
 \geq\int_{r_\varepsilon}^r\frac{1-s^2/4}{s}\,ds.
\]
Letting $r_\varepsilon\to0+$ gives
\[
 U(r)\geq\upr r e^{-r^2/8}.
\]
At $r=2$ the right-hand side equals $2\upr e^{-1/2}$, and consequently
\[
 [\sqrt e/2,\infty)\subset\calC_U.
\]
Define
\[
 \uprzCa:=\sup\calC_a,\qquad \uprzCU:=\inf\calC_U.
\]
By the crossing dichotomy, relative openness, and the disjointness proved
above, one has $0<\uprzCa\leq\uprzCU<\infty$ and
\begin{equation}\label{eq:crossing_intervals}
 \calC_a=[0,\uprzCa),
 \qquad
 \calC_U=(\uprzCU,\infty).
\end{equation}

If $\upr\in[\uprzCa,\uprzCU]$, then neither component crosses $1$.  It follows
that
\[
 0<U(r,\upr)<1,\qquad 0<a(r,\upr)<1
\]
throughout the maximal interval.  Both components are then strictly
increasing and bounded, and the solution extends to every $r>0$.  Let their
limits be $U_\infty$ and $a_\infty$.  If $U_\infty<1$, then
$a'(r)\geq c r$ for all large $r$ and some $c>0$, contradicting $a<1$.
Thus $U_\infty=1$.  If $a_\infty<1$, then for some $c>0$ and all large $r$,
\[
 \frac{U'(r)}{U(r)}=\frac{1-a(r)}r\geq\frac cr,
\]
which forces $U$ to grow past $1$.  Hence $a_\infty=1$.  Every parameter in
$[\uprzCa,\uprzCU]$ therefore produces a solution converging to $(1,1)$.

It remains to prove uniqueness.  Suppose that two parameters
$\upr_2>\upr_1>0$ both produce such a solution, and set
\[
 \Delta U:=U(\,\cdot\,,\upr_2)-U(\,\cdot\,,\upr_1),
 \qquad
 \Delta a:=a(\,\cdot\,,\upr_2)-a(\,\cdot\,,\upr_1).
\]
The expansions at zero show that $\Delta U>0$ and $\Delta a<0$ for small
positive $r$.  Moreover,
\begin{equation}\label{eq:Delta_ode_complete}
\begin{aligned}
 \Delta U'
 &=\frac{1-a(r,\upr_1)}r\Delta U
   -\frac{\Delta a}{r}U(r,\upr_2),\\
 \Delta a'
 &=-\frac r2\bigl(U(r,\upr_2)+U(r,\upr_1)\bigr)\Delta U.
\end{aligned}
\end{equation}
Since $a(r,\upr_1)<1$, an argument involving a first-zero (see above) applied to~\eqref{eq:Delta_ode_complete} shows that
\[
 \Delta U(r)>0,\qquad \Delta a(r)<0
 \quad\text{for every }r>0.
\]
In fact the first equation then gives $\Delta U'(r)>0$ for every $r>0$.
This is impossible because both profiles converge to $1$, and hence
$\Delta U(r)\to0$ as $r\to\infty$.  Thus the parameter is unique.  It follows
that $\uprzCa=\uprzCU=: \uprz$, which proves existence and uniqueness.  
\end{proof}

\begin{rem}\label{rem:J0}
    The proof implies the following practical enclosure criterion, which we will rely on for computations: if $0<\upr_-<\upr_+$ satisfy $a(r_-,\upr_-)>2$ for some $r_->0$ and $U(r_+,\upr_+)>2$ for some $r_+>0$, then the unique stable value $\uprz$ satisfies $\uprz\in [\upr_-,\upr_+]$. The reason we use $>2$ rather than $>1$ is the greater stability of the former.
\end{rem}

We now turn to the task of approximating the critical value of~$\upr=\uprz$ in Proposition~\ref{prop:exuniqUa}. We begin by writing the solutions of~\eqref{numerics2} as Frobenius series near~$r=0$. 
Changing variables $U(r)=\upr rW(r)$ with $W(0)=1$ and $a(r)=\frac{r^2}{4}b(r)$ with $b(0)=1$ leads to the system~\eqref{eq:Wbsys}, viz.
\begin{align*}
    W'(r) &= -\frac{r}{4}b(r) W(r),\\
    b'(r) &= \frac{2}{r}(1-b(r))- 2\upr^2 r W(r)^2.
\end{align*}
Although we can solve for $W$ in terms of $b$, we choose to keep the system in this form. 
We seek solutions in the form of power series 
\begin{equation}
    \label{eq:Wb}
    W(r) = \sum_{j=0}^\infty (-1)^j c_j r^{2j}, \quad b(r)= \sum_{\ell=0}^\infty (-1)^\ell  d_\ell r^{2\ell},
\end{equation}
with $c_0=1$ and $d_0=1$. The equations  yield the recursion
\begin{equation}\label{eq:recur}
    \begin{aligned}
      c_n & =\frac{1}{8n} \sum_{j+\ell=n-1}  c_j d_\ell, \\
      d_n &= \frac{\upr^2}{n+1} \sum_{j+\ell=n-1}  c_j c_\ell,
    \end{aligned}
    \end{equation}
    for all $n\ge1$.  
The range of $\upr$ in the following lemma contains the previously reported estimate $\uprz=U'(0)=0.6032\pm 10^{-4}$ (see \cite{VeSh}). More recent numerical studies (e.g.\ \cite{AIBPMCNOQu}) use the critical-coupling vortex profiles as input but do not, to our knowledge, report a refined standalone value for $U'(0)$.

\begin{lem}\label{lem:Wberror}
   For each $0\le \upr \le \frac{1}{\sqrt{2}}$ the recursion \eqref{eq:recur} has unique nonnegative solutions that satisfy the bounds $c_n\le 2^{-n}$ and $d_n\le 2^{-n}$ for each $n\ge0$. In conclusion, the power series~\eqref{eq:Wb} converge absolutely for $0\le r<\sqrt2$. Moreover, $U(r) = \upr rW(r)$ and $a(r)=\frac{r^2}{4}b(r)$ are analytic solutions of~\eqref{numerics2} in that interval. Finally, for every integer $N\ge0$ and $0\le r\le \frac14$,
   \[
    \Big |W(r) -  \sum_{j=0}^N (-1)^j c_j r^{2j}\Big| \le \frac{ 2^{-5N}}{31},\qquad \Big |b(r) -  \sum_{j=0}^N (-1)^j d_j r^{2j}\Big| \le\frac{ 2^{-5N}}{31}.
   \]
\end{lem} 
\begin{proof}
The coefficients are uniquely determined by \eqref{eq:recur} and are
nonnegative. Suppose inductively that $c_j,d_j\le2^{-j}$ for $j<n$.
Since each sum in \eqref{eq:recur} has $n$ terms,
\[
 c_n\le \frac1{8n}\,n2^{-(n-1)}=2^{-n-2}\le2^{-n},
\]
and, since $\upr^2\le1/2$,
\[
 d_n\le \frac1{2(n+1)}\,n2^{-(n-1)}
      =\frac{n}{n+1}2^{-n}\le2^{-n}.
\]
This closes the induction. The coefficient bounds give absolute convergence
for $r<\sqrt2$, so termwise substitution in \eqref{eq:Wbsys} is justified.
Finally, for $0\le r\le1/4$, put $q=r^2/2\le1/32$. Then
\[
 \sum_{n>N}2^{-n}r^{2n}
 =\frac{q^{N+1}}{1-q}
 \le\frac{2^{-5N}}{31},
\]
which proves both remainder estimates.
\end{proof}

We denote the unique solution of \eqref{numerics2} by
$(U(r,\upr),a(r,\upr))$.  We use the truncations
\[
 U_N(r,\upr):=\upr r\sum_{j=0}^N(-1)^jc_j(\upr)r^{2j},
 \qquad
 a_N(r,\upr):=\frac{r^2}{4}\sum_{j=0}^N(-1)^jd_j(\upr)r^{2j},
\]
for
\begin{align}\label{app_kinterval}
 \upr\in J_0:=[0.6032878545810,\,0.6032878545819].
\end{align}
Proposition~\ref{prop:J0slope} will establish that $\uprz\in J_0$. First, we use the previous Frobenius series to determine initial data for the ODE~\eqref{numerics2} at $\rorigin=0.1$ and all $\upr\in J_0$. 

\begin{lem}\label{lem:I0}
At the radius $\rorigin=0.1$  
\begin{equation}\label{eq:vortex_start_box}
 \begin{aligned}
 U(\rorigin,\upr)&\in[0.0602534900425221,\,0.0602534900426123],\\
 a(\rorigin,\upr)&\in[0.00249545811721475,\,0.00249545811721477],
 \end{aligned}
\end{equation}
uniformly for $\upr\in J_0$.
\end{lem}

\begin{proof}
For \eqref{eq:vortex_start_box}, the code evaluates the Frobenius recursion
through $c_{80}$ for $U$ and through $d_{79}$ for $a$.  With $\upr_+ := 0.6032878545819$, and 
$q_*=\rorigin^2/2=1/200$, Lemma~\ref{lem:Wberror} bounds the two omitted tails by
\begin{align*}
 \bigl|U(\rorigin,\upr)-U_{80}(\rorigin,\upr)\bigr|&\le \upr_+\rorigin\frac{q_*^{81}}{1-q_*}, \qquad \  \bigl|a(\rorigin,\upr)-a_{79}(\rorigin,\upr)\bigr| \le \frac{\rorigin^2}{4}\frac{q_*^{80}}{1-q_*},
\end{align*}
both of which are smaller than $3\cdot10^{-187}$.  Outward-rounded interval
evaluation of the finite recursion over $J_0$ then gives
\eqref{eq:vortex_start_box}.
\end{proof}

The following is a key enclosure of the unique value of $\upr$ that gives us the vortex.

\begin{prop}\label{prop:J0slope}
If $U$ is the true vortex profile, then $\uprz=U'(0)\in J_0$.
\end{prop}

\begin{proof}
Set $\upr_-=0.6032878545810$ and $\upr_+=0.6032878545819$.  Starting
from the certified initial enclosures \eqref{eq:vortex_start_box}, the
CAPD solver propagates the vortex system \eqref{numerics2} with the
shooting parameter frozen at each endpoint. The decimal inputs are first
enclosed by directed interval parsing, and all subsequent arithmetic is
rounded outward. The multiprecision transcript
\cite{AYMHPartICertificateRepo} records
\[
 \begin{aligned}
 a(27.75,\upr_-)-2
 &\in[0.01048373219531,\,0.01048373219532],\\
 U(31.6,\upr_+)-2
 &\in[0.00213149536392,\,0.00213149536394].
 \end{aligned}
\]
The endpoints therefore lie on the two opposite crossing sides described in
Proposition~\ref{prop:exuniqUa}.  Remark~\ref{rem:J0} yields
$\uprz\in[\upr_-,\upr_+]=J_0$.  This use of $J_0$ is not circular:
Lemma~\ref{lem:I0} constructs the initial enclosures uniformly for every
parameter in the prescribed interval $J_0$, without assuming that
$\uprz\in J_0$.
\end{proof}

Next, we solve for $(U,a)$ from $r=\infty$ assuming that $(U(r),a(r))\to (1,1)$
as $r\to\infty$.  The linearized equations about~$(1,1)$ have   the modified Bessel functions $K_\nu(r), I_\nu(r)$ as a fundamental system (see the following proof for details). 
Recall the asymptotic expansions 
\begin{equation}\begin{aligned}
    I_0(r) &= \frac{e^r}{\sqrt{2\pi r}}\Big(1 + \frac{1}{8r}+O(r^{-2})\Big), \quad I_1(r) = \frac{e^r}{\sqrt{2\pi r}}\Big(1 - \frac{3}{8r}+O(r^{-2})\Big), \\
    K_0(r) &= \frac{\sqrt{\pi} e^{-r}}{\sqrt{2 r}}\Big(1 - \frac{1}{8r}+O(r^{-2})\Big), \quad K_1(r) = \frac{\sqrt{\pi}e^{-r}}{\sqrt{2 r}}\Big(1 + \frac{3}{8r}+O(r^{-2})\Big),
\end{aligned}
\label{eq:Bessel_KI}
\end{equation}
as  $r\to\infty$, see \url{https://dlmf.nist.gov/10.40}.

\begin{lem}
    \label{lem:Ua}
Let $(U,a)$ be the unique vortex profile given by
Proposition~\ref{prop:exuniqUa}.  Then 
\begin{equation}\label{eq:mB_definition}
 m_{\rm B}:=
 \lim_{r\to\infty}\frac{1-U(r)}{K_0(r)}
 =
 \lim_{r\to\infty}\frac{1-a(r)}{rK_1(r)}>0
\end{equation}
exists.  
\end{lem}
\begin{proof}
Set
\[
 \widetilde U:=1-U,\qquad \widetilde a:=1-a.
\]
Then \eqref{numerics2} becomes
\begin{equation}\label{eq:sysinfty}
\begin{aligned}
 \widetilde U'
 &=-r^{-1}\widetilde a+r^{-1}\widetilde a\,\widetilde U,\\
 \widetilde a'
 &=-r\widetilde U+\frac r2\widetilde U^2.
\end{aligned}
\end{equation}
We first establish the limits in \eqref{eq:mB_definition}.  Introduce 
\[
 q(r):=-2\log U(r).
\]
By Proposition~\ref{prop:exuniqUa}, one has $q>0$, $q'<0$, and
$q,q'\to0$ as $r\to\infty$.  Moreover,
\[
 q''+\frac1r q'=1-e^{-q}.
\]
Fix $0<\gamma<1$.  For all sufficiently large $r$,
$1-e^{-q}\geq\gamma^2q$, and hence $q''\geq\gamma^2q$.  Therefore
\[
 \frac{d}{dr}\bigl(q'^2-\gamma^2q^2\bigr)
 =2q'\bigl(q''-\gamma^2q\bigr)\leq0.
\]
Since $q'^2-\gamma^2q^2\to0$ at infinity, this nonincreasing quantity is
nonnegative.  Thus $-q'\geq\gamma q$, and
\begin{equation}\label{eq:preliminary_U_decay}
 \widetilde U(r)=O(e^{-\gamma r}).
\end{equation}
The second equation in \eqref{eq:sysinfty} gives
\[
 \widetilde a(r)
 =\int_r^\infty
 s\widetilde U(s)\left(1-\frac{\widetilde U(s)}2\right)\,ds,
\]
so
\begin{equation}\label{eq:preliminary_a_decay}
 \widetilde a(r)=O(re^{-\gamma r}).
\end{equation}
The linear part has the fundamental solutions
\[
 (I_0(r),-rI_1(r)),\qquad (K_0(r),rK_1(r)).
\]
Write
\[
 (\widetilde U,\widetilde a)
 =A(r)(K_0,rK_1)+D(r)(I_0,-rI_1).
\]
Using the Wronskian relation 
\begin{equation}\label{eq:WrKI}
 r\bigl(K_0(r)I_1(r)+K_1(r)I_0(r)\bigr)=1,
\end{equation}
one obtains
\[
 A=rI_1\widetilde U+I_0\widetilde a,\qquad
 D=rK_1\widetilde U-K_0\widetilde a,
\]
and, from \eqref{eq:sysinfty},
\begin{align*}
 A'
 &=rI_1\frac{\widetilde a\,\widetilde U}{r}
   +I_0\frac r2\widetilde U^2,\\
 D'
 &=rK_1\frac{\widetilde a\,\widetilde U}{r}
   -K_0\frac r2\widetilde U^2.
\end{align*}
Choose $\gamma>1/2$ in
\eqref{eq:preliminary_U_decay}, \eqref{eq:preliminary_a_decay}.  The estimates in \eqref{eq:Bessel_KI} then show that $A'$ and $D'$ are integrable.  Hence
$A(r)$ has a finite limit, denoted by $m_{\rm B}$, while the displayed
formula for $D$ gives $D(r)\to0$.  More precisely,
\[
 A(r)-m_{\rm B}
 =O\bigl(r^{1/2}e^{-(2\gamma-1)r}\bigr),
 \qquad
 D(r)=O\bigl(r^{1/2}e^{-(2\gamma+1)r}\bigr).
\]
It follows that
\begin{equation}\label{eq:Ua_Bessel_asymptotic}
\begin{aligned}
 \widetilde U(r)
 &=m_{\rm B}K_0(r)+O(e^{-2\gamma r}),\\
 \widetilde a(r)
 &=m_{\rm B}rK_1(r)+O(re^{-2\gamma r}).
\end{aligned}
\end{equation}
Furthermore, $A(r)>0$ and $A'(r)>0$, so $m_{\rm B}>0$.  This proves
\eqref{eq:mB_definition}.  
\end{proof}

\noindent Using \eqref{eq:Bessel_KI} and choosing $2\gamma>1$ in \eqref{eq:Ua_Bessel_asymptotic} yields
\begin{align}
U(r)&= 1-mr^{-\frac12}e^{-r}+O(r^{-\frac32}e^{-r}), & &    a(r)= 1-mr^{1/2}e^{-r}+O(r^{-\frac12}e^{-r}) , & &   r\to\infty,  \qquad 
\label{Uaexpinfty}
\end{align}
with $m=\sqrt{\pi/2}\,m_{\rm B}$.
\def\rzeroUbound{4}\def\rzeroUparameter{3}
Although the previous lemma determines the asymptotic behavior of $U(r)$
and $a(r)$ to leading order for large~$r$, it does not provide a
certified numerical enclosure for the coefficient $m$. We therefore
prove rigorous exponential upper and lower bounds for
$r>\rzeroUbound$ with explicit constants, which can be seen as an
effective version of Lemma~\ref{lem:Ua}. The only computational
ingredient is the verification, carried out by our C++ code using
certified interval arithmetic, of the inequalities in
\eqref{app_num_condition_a6} at the single, relatively small radius
$\rfour=\rzeroUbound$. Their extension to every $r>\rfour$ is entirely
analytic: once \eqref{app_num_condition_a6} has been certified at $\rfour$,
the comparison argument below yields \eqref{lem_1-U2_statement} for all
$r>\rfour$. 

\begin{lem}\label{lem_1-U2}
For all $r>\rzeroUbound$, one has
\begin{equation}\label{lem_1-U2_statement}
\begin{aligned}
1-\bigl(4r^{-1/2}-2r^{-3/2}+3r^{-5/2}\bigr)e^{-r}
&<U(r)<1-\tfrac75K_0(r)+K_0(r)^2,\\
1-4\sqrt r\,e^{-r}
&<a(r)<1-\tfrac75rK_1(r).
\end{aligned}
\end{equation}
\end{lem}

\begin{proof}
Set $\rfour=\rzeroUbound$ and define, as before,
\begin{equation}\label{1-U2feq}
\begin{aligned}
\mathcal A(r)&:=1-4\sqrt r\,e^{-r},
&\mathsf A(r)&:=1-\tfrac75rK_1(r),\\
\mathcal U_g(r)&:=1-\bigl(4r^{-1/2}-2r^{-3/2}\bigr)e^{-r},
&\mathsf U_g(r)&:=1-\tfrac75K_0(r)+K_0(r)^2,\\
\mathcal U_s(r)&:=1-\bigl(4r^{-1/2}-2r^{-3/2}
  +3r^{-5/2}\bigr)e^{-r},
&\mathsf U_s(r)&:=1-\tfrac75K_0(r)-K_0(r)^2.
\end{aligned}
\end{equation}
Starting from the initial boxes in \eqref{eq:vortex_start_box}, the C++
verifier propagates \eqref{numerics2} from $\rorigin=0.1$ to $\rfour=4$ and
certifies
\begin{equation}\label{app_num_condition_a6}
\mathcal A(\rfour)<a(\rfour,\upr)<\mathsf A(\rfour),
\qquad
\mathcal U_g(\rfour)<U(\rfour,\upr)<\mathsf U_s(\rfour)
\end{equation}
uniformly for $\upr\in J_0$. Proposition~\ref{prop:J0slope} implies that
$\uprz=U'(0)\in J_0$, hence we may specialize to $\upr=\uprz$. 
Proposition~\ref{prop:exuniqUa}
gives $0<U,a<1$ and $(U,a)\to(1,1)$.  All six comparison profiles are
positive on $[\rfour,\infty)$ and converge to $1$.

We first prove the lower bounds.  Direct differentiation gives
\begin{equation}\label{lem_1-U2_ineq}
\begin{aligned}
\mathcal U_g'
&<\frac{1-\mathcal A}{r}\mathcal U_g,
&\qquad
\mathcal A'
&>\frac r2(1-\mathcal U_g^2),\\
\mathcal U_s'
&>\frac{1-\mathcal A}{r}\mathcal U_s,
&
\mathcal A'
&<\frac r2(1-\mathcal U_s^2)
\end{aligned}
\end{equation}
for $r>\rfour$.  For clarity, the four residuals are
\begin{align*}
\frac{1-\mathcal A}{r}\mathcal U_g-\mathcal U_g'
&=3r^{-5/2}e^{-r}
  -\left(16r^{-1}-8r^{-2}\right)e^{-2r}>0,\\
\mathcal A'-\frac r2(1-\mathcal U_g^2)
&=2\left(4-\frac4r+\frac1{r^2}\right)e^{-2r}>0,\\
\mathcal U_s'-\frac{1-\mathcal A}{r}\mathcal U_s
&=\frac{15}{2}r^{-7/2}e^{-r}
  {}+\frac{1-\mathcal A}{r}(1-\mathcal U_s)>0,\\
\frac r2(1-\mathcal U_s^2)-\mathcal A'
&=\frac r2\left[
  6r^{-5/2}e^{-r}
  \right.\notag\\
&\qquad\left.
  -\left(16r^{-1}-16r^{-2}+28r^{-3}
  -12r^{-4}+9r^{-5}\right)e^{-2r}\right]>0.
\end{align*}
For the first inequality, use
$3e^r>16r^{3/2}-8r^{1/2}$ for $r\ge4$.  For the last, after multiplying by
$re^{2r}$, the left-hand side is $6e^rr^{-3/2}>40$, whereas the expression
on the right is less than
$16+28r^{-2}+9r^{-4}<18$.

For the lower comparison, define
\[
x_s:=U-\mathcal U_s,\qquad
x_g:=U-\mathcal U_g,\qquad
y:=a-\mathcal A.
\]
Then $x_s(\rfour)>x_g(\rfour)>0$ and $y(\rfour)>0$, while all three differences tend
to zero at infinity.  From \eqref{lem_1-U2_ineq},
\begin{equation}\label{app_num_condition_a8}
\begin{aligned}
x_g'&>\frac{1-a}{r}x_g-\frac{\mathcal U_g}{r}y,\\
x_s'&<\frac{1-a}{r}x_s-\frac{\mathcal U_s}{r}y,\\
y'&<-\frac r2(U+\mathcal U_g)x_g,\\
y'&>-\frac r2(U+\mathcal U_s)x_s.
\end{aligned}
\end{equation}

Suppose that one of the two lower bounds fails, and let
\[
r_*:=\inf\{r>\rfour:x_s(r)=0\ \hbox{or}\ y(r)=0\}.
\]
If $x_s(r_*)=0<y(r_*)$, then $x_s'(r_*)<0$ and $y'(r_*)>0$.
Moreover, the quadrant $x_s<0$, $y>0$ is forward invariant by the second
and fourth inequalities in \eqref{app_num_condition_a8}.  Thus $x_s$
decreases and $y$ increases for all $r>r_*$, contradicting
$x_s,y\to0$.

Suppose next that $y(r_*)=0<x_s(r_*)$.  If $x_g(r_*)\ge0$, the quadrant
$x_g>0$, $y<0$ is forward invariant by the first and third inequalities,
again contradicting convergence to zero.  It remains to exclude
$x_g(r_*)<0$.  Let
\[
r_{**}:=\sup\{r\in(\rfour,r_*):x_g(r)=0\}.
\]
Then $x_g<0$ on $(r_{**},r_*)$.  The exact difference equations are
\begin{align}
x_g'
&=\frac{1-a}{r}x_g-\frac{\mathcal U_g}{r}y
{}+3r^{-5/2}e^{-r}
-\frac{1-\mathcal A}{r}(1-\mathcal U_g),\label{eq:xg_exact}\\
y'
&=-\frac r2(U+\mathcal U_g)x_g
-2\left(4-\frac4r+\frac1{r^2}\right)e^{-2r}.\label{eq:y_exact}
\end{align}
Since $x_g'(r_{**})\le0$, \eqref{eq:xg_exact} gives
\[
0<\mathcal G(r_{**})\le y(r_{**}),
\]
where
\[
\mathcal G(r):=\frac r{\mathcal U_g(r)}
\left[
3r^{-5/2}e^{-r}
-\frac{1-\mathcal A(r)}r(1-\mathcal U_g(r))
\right].
\]
Also set
\[
\mathcal Q(a,b):=
-\int_a^b2\left(4-\frac4s+\frac1{s^2}\right)e^{-2s}\,ds.
\]
We claim that
\begin{equation}\label{eq:GQ_lower}
-\mathcal G(r)<2\mathcal Q(r,\infty),\qquad r\ge\rfour.
\end{equation}
To verify this, put $p(r)=(2-r^{-1})^2$.  Since $p$ is concave on
$[\rfour,\infty)$,
\[
-2\mathcal Q(r,\infty)
=4\int_r^\infty p(s)e^{-2s}\,ds
\le\bigl(2p(r)+p'(r)\bigr)e^{-2r}.
\]
Since $\mathcal U_g<1$,
\[
\mathcal G(r)>
3r^{-3/2}e^{-r}
-\left(16-\frac8r\right)e^{-2r}.
\]
It therefore suffices to check
\[
3e^r>
h(r):=24r^{3/2}-16r^{1/2}+6r^{-1/2}-2r^{-3/2}.
\]
At $r=\rfour$, the two sides are $163.794\ldots$ and $162.75$.  Moreover,
\[
h(r)-h'(r)
=r^{-5/2}\left(24r^4-52r^3+14r^2+r-3\right)>0
\]
for $r\ge\rfour$, so $3e^r-h(r)$ is increasing once positive.  This proves
\eqref{eq:GQ_lower}.

Integrating \eqref{eq:y_exact} from $r_{**}$ to $r_*$ and using $x_g<0$
gives
\[
\begin{aligned}
0=y(r_*)
&>y(r_{**})+\mathcal Q(r_{**},r_*)\\
&\ge\mathcal G(r_{**})+\mathcal Q(r_{**},r_*)\\
&>\mathcal G(r_{**})+\mathcal Q(r_{**},\infty)>0,
\end{aligned}
\]
a contradiction.  Here the last inequality follows from
\eqref{eq:GQ_lower} and $\mathcal Q<0$.  If $x_s(r_*)=y(r_*)=0$, the strict
inequalities give $x_s'(r_*)<0$ and $y'(r_*)>0$, reducing to the first case.
This proves
\[
\mathcal U_s(r)<U(r),\qquad \mathcal A(r)<a(r),\qquad r>\rfour.
\]

We now prove the upper bounds by the same comparison, recording the required
analytic checks.  Since
\[
\frac{1-\mathsf A}{r}=\frac75K_1(r),\qquad
\mathsf A'=\frac75rK_0(r),\qquad
\mathsf U_g'=\left(\frac75-2K_0(r)\right)K_1(r),\qquad
\mathsf U_s'=\left(\frac75+2K_0(r)\right)K_1(r),
\]
one has
\begin{align*}
\frac{1-\mathsf A}{r}\mathsf U_g-\mathsf U_g'
&=K_0(r)K_1(r)
  \left(\frac1{25}+\frac75K_0(r)\right)>0,\\
\mathsf A'-\frac r2(1-\mathsf U_g^2)
&=rK_0(r)^2
  \left[1+\frac12\left(\frac75-K_0(r)\right)^2\right]>0,
\end{align*}
\begin{align*}
\mathsf U_s'-\frac{1-\mathsf A}{r}\mathsf U_s
&=K_0(r)K_1(r)
  \left(\frac{99}{25}+\frac75K_0(r)\right)>0,\\
\frac r2(1-\mathsf U_s^2)-\mathsf A'
&=rK_0(r)^2
  \left[1-\frac12\left(\frac75+K_0(r)\right)^2\right]>0.
\end{align*}
For the last sign, the integral representation of $K_0$ gives
\[
0<K_0(r)\le\sqrt{\frac{\pi}{2r}}e^{-r}
\le\sqrt{\frac\pi8}e^{-4}<\frac1{75}<\sqrt2-\frac75.
\]
Thus the exact analogues of \eqref{lem_1-U2_ineq} hold for
$(\mathsf U_g,\mathsf U_s,\mathsf A)$.
\begin{figure}
\centering
\includegraphics[scale=0.69]{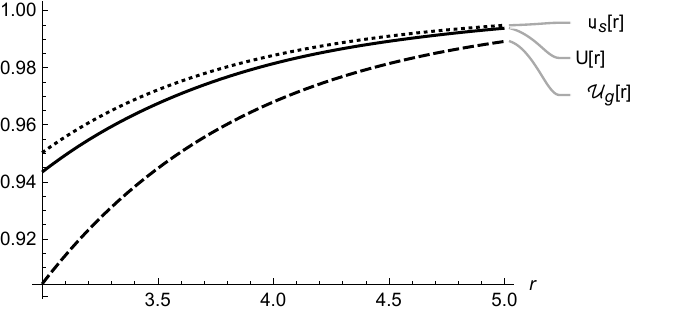}
\qquad 
\includegraphics[scale=0.65]{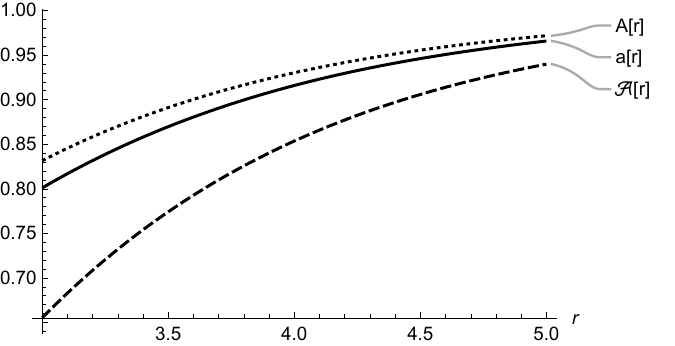}
\caption{Illustration of Lemma~\ref{lem_1-U2}. The solid curves show the vortex profile $(U,a)$, and the dashed and dotted curves show the comparison functions defined in \eqref{1-U2feq}. The boundary inequalities \eqref{app_num_condition_a6} are verified by interval arithmetic. These enclosures become highly accurate for large~$r$.}
\label{fig3LT}
\end{figure}
Set
\[
X_g:=\mathsf U_g-U,\qquad
X_s:=\mathsf U_s-U,\qquad
Y:=\mathsf A-a.
\]
By \eqref{app_num_condition_a6},
$X_g(\rfour)>X_s(\rfour)>0$ and $Y(\rfour)>0$.  The preceding first-exit proof applies
with $\mathsf U_g$ as the target barrier and $\mathsf U_s$ as the auxiliary
barrier.  In its only technical subcase, if $r_{**}$ is the last zero of
$X_s$ before $Y$ first reaches zero, the exact difference equations give
$Y(r_{**})\geq\widetilde{\mathcal G}(r_{**})$; integration then gives the
same contradiction as above provided
\begin{equation}\label{eq:GQ_upper}
-\widetilde{\mathcal G}(r)
<2\widetilde{\mathcal Q}(r,\infty),
\end{equation}
where
\begin{align*}
\widetilde{\mathcal G}(r)
&:=\frac r{\mathsf U_s(r)}
\left[
\mathsf U_s'(r)
-\frac{1-\mathsf A(r)}r\mathsf U_s(r)
\right],\\
\widetilde{\mathcal Q}(a,b)
&:=\int_a^b
\left[
\mathsf A'(s)-\frac s2(1-\mathsf U_s(s)^2)
\right]ds.
\end{align*}
Indeed,
\[
\widetilde{\mathcal G}(r)
=\frac{rK_0(r)K_1(r)}{\mathsf U_s}
\left(\frac{99}{25}+\frac75K_0(r)\right)
>\frac{99}{25}rK_0(r)^2,
\]
while
\[
\mathsf A'-\frac r2(1-\mathsf U_s^2)
=-rK_0(r)^2
 \left[1-\frac12\left(\frac75+K_0(r)\right)^2\right]
>-\frac1{50}rK_0(r)^2.
\]
Since $K_1>K_0$, one has
$K_0(s)\le K_0(r)e^{-(s-r)}$ for $s\ge r$, and hence
\[
\int_r^\infty sK_0(s)^2\,ds
\le K_0(r)^2\left(\frac r2+\frac14\right).
\]
Consequently,
\[
-2\widetilde{\mathcal Q}(r,\infty)
<\frac1{25}\int_r^\infty sK_0(s)^2\,ds
<\widetilde{\mathcal G}(r),
\]
which is \eqref{eq:GQ_upper}.  The same last-zero integration used above now
excludes the technical subcase and proves
\[
U(r)<\mathsf U_g(r),\qquad a(r)<\mathsf A(r),\qquad r>\rfour.
\]
Together with the lower bounds, this is \eqref{lem_1-U2_statement}.
\end{proof}

We found empirically that our
spectral and FGR computations require a narrower interval than~$J_0$ from Proposition~\ref{prop:J0slope}.  Lemma~\ref{lem:Newton}
achieves higher accuracy by  solving  
$U(20,\upr)=1$ via a Newton scheme combined with the exponential tail estimate
from Lemma~\ref{lem_1-U2}. The  interval $J_0$ has width $9.0\cdot10^{-13}$, whereas that lemma 
reduces this to $1.263\cdot10^{-16}$, achieving an improvement  by a factor greater than $7.1\cdot10^3$. 

The Newton iteration lemma uses the equation of variations~\eqref{eq:v_ode}
with respect to the shooting parameter. In analogy with the vortex profiles themselves, we depend on Frobenius series to compute the solutions of this differentiated ODE for small~$r$. Differentiating the
explicit  recurrences \eqref{eq:recur} gives the following direct
analogue of Lemma~\ref{lem:Wberror}.

\begin{lemma}\label{lem:Wbderror}
Let $c_n(\upr),d_n(\upr)$ be the coefficients in \eqref{eq:Wb}, determined by
\eqref{eq:recur}. Set
\[
    \dot c_n(\upr):=\partial_\upr c_n(\upr),\qquad
    \dot d_n(\upr):=\partial_\upr d_n(\upr).
\]
Since $c_0(\upr)=d_0(\upr)=1$, one has
$\dot c_0(\upr)=\dot d_0(\upr)=0$.
For every $0\le \upr\le 1/\sqrt2$ and every $n\ge1$,
\[
    0\le \dot c_n(\upr)\le 4n\,2^{-n},
    \qquad
    0\le \dot d_n(\upr)\le 4n\,2^{-n}.
\]
Consequently, the differentiated Frobenius series for
$\partial_\upr W$ and $\partial_\upr b$ converge absolutely for
$0\le r<\sqrt2$. Moreover, for $0\le r\le1/4$, $q=r^2/2$, and
\[
    E_N(r):=
    4q^{N+1}\frac{(N+1)-Nq}{(1-q)^2},
\]
one has
\[
\left|
  \partial_\upr W(r,\upr)
  -\sum_{j=0}^N(-1)^j\dot c_j(\upr)r^{2j}
\right|
\le E_N(r),
\]
and
\[
\left|
  \partial_\upr b(r,\upr)
  -\sum_{j=0}^N(-1)^j\dot d_j(\upr)r^{2j}
\right|
\le E_N(r).
\]
In particular, with
\[
    U_N(r,\upr)=\upr r\sum_{j=0}^N(-1)^jc_j(\upr)r^{2j},
    \qquad
    a_N(r,\upr)=\frac{r^2}{4}\sum_{j=0}^N(-1)^jd_j(\upr)r^{2j},
\]
the shooting derivatives satisfy
\[
\left|
  \partial_\upr U(r,\upr)-\partial_\upr U_N(r,\upr)
\right|
\le
 r\,\frac{2^{-5N}}{31}+\frac{r}{\sqrt2}E_N(r),
\]
and
\[
\left|
  \partial_\upr a(r,\upr)-\partial_\upr a_N(r,\upr)
\right|
\le
\frac{r^2}{4}E_N(r).
\]
\end{lemma}

\begin{proof}
Differentiating \eqref{eq:recur} gives, for $n\ge1$,
\[
  \dot c_n
  =
  \frac1{8n}
  \sum_{j+\ell=n-1}
  \big(\dot c_jd_\ell+c_j\dot d_\ell\big),
\]
and
\[
  \dot d_n
  =
  \frac{2\upr}{n+1}
  \sum_{j+\ell=n-1}c_jc_\ell
  +
  \frac{\upr^2}{n+1}
  \sum_{j+\ell=n-1}
  \big(\dot c_jc_\ell+c_j\dot c_\ell\big).
\]
All coefficients are nonnegative polynomials in $\upr$. We use
Lemma~\ref{lem:Wberror}, namely $0\le c_n,d_n\le2^{-n}$, and prove the
displayed derivative bounds by induction. The case $n=1$ is immediate. If
the bounds hold up to $n-1$, then
\[
\begin{aligned}
  \dot c_n
  &\le
  \frac1{8n}
  \sum_{j+\ell=n-1}4(j+\ell)2^{-(n-1)}
  =
  (n-1)2^{-n}
  \le4n2^{-n}.
\end{aligned}
\]
Similarly, since $0\le\upr\le1/\sqrt2$,
\[
\begin{aligned}
  \dot d_n
  &\le
  \frac{\sqrt2}{n+1}n\,2^{-(n-1)}
  +
  \frac{1}{2(n+1)}
  \sum_{j+\ell=n-1}4(j+\ell)2^{-(n-1)}  \\
  &=
  \left(
    \frac{2\sqrt2\,n}{n+1}
    +
    \frac{4n(n-1)}{n+1}
  \right)2^{-n}
  \le4n2^{-n}.
\end{aligned}
\]
This proves the coefficient bounds. The tail estimate is 
\[
  \sum_{j>N}4j\left(\frac{r^2}{2}\right)^j
  =
  4q^{N+1}\frac{(N+1)-Nq}{(1-q)^2}.
\]
The estimates for $\partial_\upr U$ and $\partial_\upr a$ follow from
\[
  \partial_\upr U=rW+\upr r\,\partial_\upr W,\qquad
  \partial_\upr a=\frac{r^2}{4}\partial_\upr b,
\]
together with the $W$-tail estimate in Lemma~\ref{lem:Wberror}. This proves
the lemma.
\end{proof}

We can now turn to refining the interval $J_0$ from Proposition~\ref{prop:J0slope}. 

\begin{lemma}\label{lem:Newton}
The shooting parameter satisfies
\[
  \bigl|\uprz-0.6032878545816779\bigr|
  <7\cdot10^{-17}.
\]
In particular,
\begin{equation}\label{eq:Icert}
  \uprz\in I_{\rm cert}
  :=
  [0.6032878545816699,\,0.6032878545816856].
\end{equation}
The interval $I_{\rm cert}$ has width $1.57\cdot10^{-14}$ and is the
fixed working interval used in the subsequent interval\footnote{Not all digits in the endpoints of this interval are significant.} computations.
\end{lemma}

\begin{proof}
Write $(U(r,\upr),a(r,\upr))$ for the solution of the vortex equations
with initial slope $U'(0,\upr)=\upr$. Proposition~\ref{prop:exuniqUa}
gives a unique value $\uprz$ for which
\[
  \bigl(U(r,\uprz),a(r,\uprz)\bigr)\longrightarrow(1,1)
  \qquad\text{as }r\to\infty,
\]
and Proposition~\ref{prop:J0slope} gives the preliminary bound
\begin{equation}\label{eq:Newton_J0}
  \uprz\in J_0
  =[0.6032878545810,\,0.6032878545819].
\end{equation}
The condition defining $\uprz$ is imposed at infinity. We therefore first
locate the zero of the auxiliary function
\[
  F_{20}(\upr):=U(20,\upr)-1
\]
and then compare that zero with $\uprz$ using the vortex estimate at infinity provided by Lemma~\ref{lem_1-U2}.

Let $m$ be the midpoint\footnote{To be more precise, let \(m\in I_{\rm cert}\) be the 220-bit representable point used by the
C++ code. It is chosen as a floating-point approximation
to the midpoint of \(I_{\rm cert}\).
} of $I_{\rm cert}$. The C++ calculation starts at
$\rorigin=0.1$. The initial intervals for $U$ and $a$ are obtained from the
Frobenius estimates in Lemma~\ref{lem:Wberror}, while those for their
derivatives with respect to $\upr$ are obtained from
Lemma~\ref{lem:Wbderror}. The vortex equations and the variational
system~\eqref{eq:v_ode} are then integrated from $\rorigin=0.1$ to $r=20$ by
interval arithmetic.
The full outward-rounded intervals can be found in the 
transcript\footnote{\label{fn:transcript}We remind the reader that the
``transcript'' is a file in our repository containing the principal
numerical results produced by the C++ code.} contained in the Github repository~\cite{AYMHPartICertificateRepo}. They imply the following
slightly enlarged bounds:
\[
  F_{20}(m)
  \in[-4.087672,\,-4.087671]\cdot10^{-9}
\]
and
\[
  F_{20}'(\upr)
  =\partial_\upr U(20,\upr)
  \in[2.853286,\,2.853290]\cdot10^7,
  \qquad \upr\in I_{\rm cert}.
\]
In particular, $F_{20}'$ is positive throughout $I_{\rm cert}$.

Let $[F_{20}(m)]$ and $[F_{20}'(I_{\rm cert})]$ denote the intervals
computed by the code,  and form the interval 
\[
  \mathcal N_{20}
  :=
  m-\frac{[F_{20}(m)]}{[F_{20}'(I_{\rm cert})]},
\]
which is needed by the Newton method. 
Using the full intervals recorded in the transcript, the program proves
that $\mathcal N_{20}$ has width less than $1.089\cdot10^{-22}$ and that
\[
  \mathcal N_{20}
  \subset
  0.6032878545816779
  +[-6.739,\,-6.738]\cdot10^{-18}
  \Subset I_{\rm cert}.
\]
We recall the scalar interval Newton criterion: if $F\in C^1(I)$,
$0\notin[F'(I)]$, and
\[
  m-\frac{[F(m)]}{[F'(I)]}\subset\operatorname{int}I,
\]
then $F$ has a zero in $I$. Applying this criterion gives a zero
$\upr_{20}\in I_{\rm cert}$ of $F_{20}$. Since $F_{20}'>0$ on
$I_{\rm cert}$, this zero is unique. Moreover, the mean value theorem
shows that every zero in $I_{\rm cert}$ belongs to the interval Newton
image. Consequently,
\begin{equation}\label{eq:Newton_c20}
  F_{20}(\upr_{20})=0,
  \qquad
  \bigl|\upr_{20}-0.6032878545816779\bigr|
  <6.739\cdot10^{-18}.
\end{equation}
It remains to compare the finite-radius zero $\upr_{20}$ with the true
shooting parameter $\uprz$. Integrating the same variational system
uniformly for $\upr\in J_0$, the C++ calculation gives the outward-rounded
lower bound
\begin{equation}\label{eq:Newton_broad_derivative}
  \partial_\upr U(20,\upr)>2.853194\cdot10^7,
  \qquad \upr\in J_0.
\end{equation}
More explicitly, to obtain \eqref{eq:Newton_broad_derivative} the C++ code
integrates the vortex equations together with their derivatives with respect
to the shooting parameter. Writing
\[
 U_\upr:=\partial_\upr U,\qquad a_\upr:=\partial_\upr a,
\]
the system passed to CAPD is
\[
\begin{aligned}
 r'&=1,
&
 U'&=\frac{1-a}{r}U,
&
 a'&=\frac r2(1-U^2),\\
 U_\upr'&=\frac{1-a}{r}U_\upr-\frac Ur a_\upr,
&
 a_\upr'&=-rUU_\upr,
&
 \upr'&=0.
\end{aligned}
\]
Notice that
$V=U_\upr$ and $B=-a_\upr$ reduce the fourth and fifth equations to
\eqref{eq:v_ode}.
The initial intervals at $\rorigin=0.1$, uniform for $\upr\in J_0$, are obtained
from the Frobenius estimates in
Lemmas~\ref{lem:Wberror} and~\ref{lem:Wbderror}. CAPD propagates the
six-component interval vector
\[
 (r,U,a,U_\upr,a_\upr,\upr)
\]
to $r=20$ and obtains
\[
 U_\upr(20,J_0)\subset
 [2.85319410965073633711311481846,\,
  2.85331836120358078129566571744]\cdot10^7.
\]
This implies \eqref{eq:Newton_broad_derivative}. 

The hypotheses at $r=\rfour=\rzeroUbound$ needed for Lemma~\ref{lem_1-U2} were already
verified uniformly on $J_0$. Hence that lemma applies to the true vortex,
in view of \eqref{eq:Newton_J0}, and gives
\[
  0<1-U(20,\uprz)\le B_{20},
\]
where
\[
  B_{20}
  :=
  \bigl(4\cdot20^{-1/2}-2\cdot20^{-3/2}
        +3\cdot20^{-5/2}\bigr)e^{-20}
  <1.80092\cdot10^{-9}.
\]
Since $I_{\rm cert}\subset J_0$, both $\upr_{20}$ and $\uprz$ belong to
$J_0$. Using $F_{20}(\upr_{20})=0$, the mean value theorem and~\eqref{eq:Newton_broad_derivative} yield
\[
\begin{aligned}
  |\uprz-\upr_{20}|
  &\le
  \frac{|F_{20}(\uprz)-F_{20}(\upr_{20})|}
       {\displaystyle\inf_{\upr\in J_0}F_{20}'(\upr)}<
  \frac{1.80092\cdot10^{-9}}
       {2.853194\cdot10^7}
  <6.312\cdot10^{-17}.
\end{aligned}
\]
Combining this estimate with \eqref{eq:Newton_c20}, we obtain
\[
\begin{aligned}
  \bigl|\uprz-0.6032878545816779\bigr|
  &\le
  |\uprz-\upr_{20}|
  +\bigl|\upr_{20}-0.6032878545816779\bigr|\\
  &<
  6.312\cdot10^{-17}
  +6.739\cdot10^{-18}\\
  &<7\cdot10^{-17}.
\end{aligned}
\]
Finally,
\[
  0.6032878545816779
  +[-7\cdot10^{-17},\,7\cdot10^{-17}]
  \Subset I_{\rm cert},
\]
which also proves the inclusion in~\eqref{eq:Icert}.
\end{proof}

\section{Spectral properties of the linearized operators \texorpdfstring{$\calL_1,\calL_2$}{L1, L2}}

The following consequence of Lemma~\ref{lem_1-U2} supplies the explicit
tail norms used in the threshold and FGR computations below.
\begin{corollary}
\label{cor:vortex_tail_bridge}
For every $r\ge \rfour$, 
\begin{equation}\label{eq:capd_vortex_tail_bridge}
    |1-U(r)^2|\le 8\,\frac{e^{-r}}{\sqrt r}.
\end{equation}
Consequently, for $ R\ge\rfour$,
\begin{equation}\label{eq:capd_vortex_tail_norms}
    \int_R^\infty s\,|1-U(s)^2|\log\frac{s}{R}\,ds
    \le 8 e^{-R}\frac{R+2}{R},
    \qquad
    \int_R^\infty s\,|1-U(s)^2|\,ds
    \le 8e^{-R}(R+1).
\end{equation}
\end{corollary}

\begin{proof}
Lemma~\ref{lem_1-U2} gives, for $r>\rfour$,
\[
    U(r)>1-\big(4r^{-1/2}-2r^{-3/2}+3r^{-5/2}\big)e^{-r}.
\]
Together with the already established upper bound $U(r)<1$, this implies
\[
    0<1-U(r)^2
    \le 2(1-U(r)).
\]
Hence
\[
    1-U(r)^2
    \le
    \big(8r^{-1/2}-4r^{-3/2}+6r^{-5/2}\big)e^{-r}
    =
    8\frac{e^{-r}}{\sqrt r}\Big(1-\frac{1}{2r}+\frac{3}{4r^2}\Big).
\]
The factor in parentheses is at most $1$ for $r\ge 3/2$, and therefore for all $r\ge R\ge\rfour$. This proves \eqref{eq:capd_vortex_tail_bridge}. For the first bound in \eqref{eq:capd_vortex_tail_norms} we use, for $s\ge R\ge\rfour$, the elementary inequalities $\sqrt{s}\le s$ and $\log(s/R)\le s/R-1$. Thus
\[
    \int_R^\infty s\,|1-U(s)^2|\log\frac{s}{R}\,ds
    \le 8\int_R^\infty \sqrt{s}e^{-s}\log\frac{s}{R}\,ds
    \le 8\int_R^\infty s e^{-s}\Big(\frac{s}{R}-1\Big)\,ds
    =8e^{-R}\frac{R+2}{R}.
\]
Similarly,
\[
    \int_R^\infty s\,|1-U(s)^2|\,ds
    \le 8\int_R^\infty s e^{-s}\,ds
    =8e^{-R}(R+1),
\]
as claimed.
\end{proof}

\subsection{Absence of discrete spectrum and of a threshold resonance for \texorpdfstring{$\mathcal{L}_1$}{L1}}

Next, we give a rigorous computer-assisted proof of the fact that $V_1>1$. The argument is analytic
on $(0,\sqrt6)$ and $(\rfour,\infty)$; interval arithmetic is needed only on
the compact interval $[\sqrt6,\rfour]$.

\begin{figure}
\centering
\includegraphics[scale=0.7
]{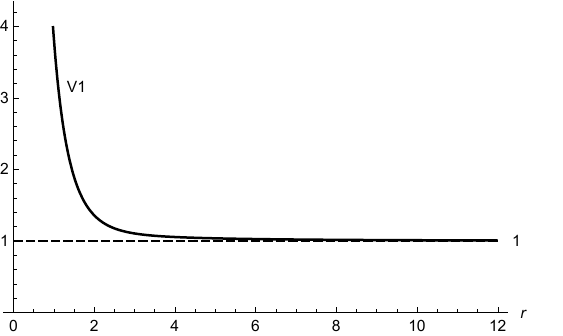}
\quad \includegraphics[scale=0.7
]{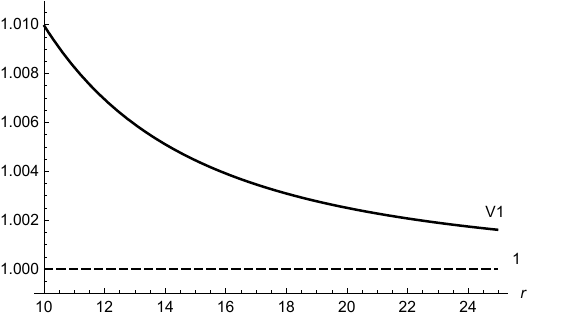}
\caption{Graph of $V_1(r)$ on $(0,12)$ and $(10,25)$ respectively.}
\label{figlemV1}
\end{figure}

\begin{lem}\label{lem_appV1g1}
The potential $V_1$ satisfies
\begin{align}\label{lem_app_V1}
V_1(r)-1
=\frac{(2-a(r))^2}{r^2}-\frac12\bigl(1-U(r)^2\bigr)>0,
\qquad r\in(0,\infty).
\end{align}
Hence, $\mathcal L_1$ has no discrete spectrum or threshold resonance.
\end{lem}

\begin{proof}
  We first record the inequality
\begin{equation}\label{eq:U_greater_a}
U(r)>a(r),\qquad r>0,
\end{equation}
which will be used below.  Indeed, since $0<U<1$, the second equation in
\eqref{numerics2} gives
\[
a(r)=\int_0^r\frac{s}{2}\bigl(1-U(s)^2\bigr)\,ds<\frac{r^2}{4}.
\]
Set $F=U-a$.  The expansions \eqref{Uaexp0} give $F(r)>0$ for all
sufficiently small $r>0$, while \eqref{Uaexpinfty}, with $m>0$, gives
\[
F(r)=m\bigl(r^{1/2}-r^{-1/2}\bigr)e^{-r}
  +O\bigl(r^{-1/2}e^{-r}\bigr)>0
\]
for all sufficiently large $r$.  If $F(r)=0$ at some $r>0$, and if
$x:=U(r)=a(r)$, then $0<x<1$, $r^2>4x$, and \eqref{numerics2} yields
\[
F'(r)
=\frac{1-x}{r}\left(x-\frac{r^2}{2}(1+x)\right)<0.
\]
Thus every zero of $F$ would be a strict crossing from positive to negative.
Since $F$ is positive both near zero and for all sufficiently large $r$, no
such zero can occur.  This proves \eqref{eq:U_greater_a}.

We now prove \eqref{lem_app_V1}.  By \eqref{eq:U_greater_a},
\begin{align}\label{lem_app_V12}
\frac{(2-a)^2}{r^2}-\frac12(1-U^2)
>
\frac{(2-U)^2}{r^2}-\frac12(1-U^2)
=:g(r).
\end{align}
For $0<x<1$,
\[
\frac{2(2-x)^2}{1-x^2}-6
=\frac{2(2x-1)^2}{1-x^2}\ge0.
\]
Consequently, if $g(r)\le0$, then
\[
r^2\ge \frac{2(2-U(r))^2}{1-U(r)^2}\ge6.
\]
It follows that $g(r)>0$, and hence $V_1(r)-1>0$, for
$0<r<\sqrt6$.

On the compact interval $[\sqrt6,\rfour]$, the interval arithmetic computation
described in Section~\ref{sec:capd_certificate} gives the certified bound
\[
V_1(r)-1>0.0548.
\]
It remains to consider $r>\rfour=\rzeroUbound$.  Put
\[
p(r):=4r^{-1/2}-2r^{-3/2}+3r^{-5/2}.
\]
The lower comparison functions in Lemma~\ref{lem_1-U2} are positive for
$r\ge\rfour$.  Indeed, direct differentiation shows that
$4\sqrt r\,e^{-r}$ and $p(r)e^{-r}$ are decreasing there, and their values
at $r=\rfour$ are $8e^{-4}<1$ and $\frac{59}{32}e^{-4}<1$, respectively.
Hence that lemma, together with $1-a>0$, gives
\begin{align*}
V_1(r)-1
&=\frac{4}{r^2}(1-a)+\frac{a^2}{r^2}
  +\frac12\bigl(U^2-1\bigr)\notag\\
&>\frac1{r^2}\bigl(1-4\sqrt r\,e^{-r}\bigr)^2
 +\frac12\left(\bigl(1-p(r)e^{-r}\bigr)^2-1\right)\notag\\
&=\frac{1}{r^2}
 -\bigl(4r^{-1/2}+6r^{-3/2}+3r^{-5/2}\bigr)e^{-r}\notag\\
&\quad
 +\bigl(24r^{-1}-8r^{-2}+14r^{-3}-6r^{-4}
   +\tfrac92r^{-5}\bigr)e^{-2r}.
\end{align*}
The coefficient of $e^{-2r}$ in the last line is
\[
16r^{-1}+\frac12p(r)^2>0.
\]
Moreover, if
\[
P(r):=4r^{3/2}+6r^{1/2}+3r^{-1/2},
\]
then
\[
\frac{d}{dr}\bigl(P(r)e^{-r}\bigr)
=-e^{-r}\left(4r^{3/2}+\frac32r^{-3/2}\right)<0,
\]
and
\[
P(4)e^{-4}=\frac{91}{2}e^{-4}<1.
\]
Therefore $P(r)e^{-r}<1$ for every $r\ge\rfour$, and already the first two
terms in the preceding lower bound have positive sum.  This proves
\eqref{lem_app_V1} for $r>\rfour$, and hence on all of $(0,\infty)$.

Finally, since $V_1(r)\to1$, one has
$\sigma_{\rm ess}(\mathcal L_1)=[1,\infty)$.  For every $f$ in the form
domain of $\mathcal L_1$,
\[
\big\langle(\mathcal L_1-1)f,f\big\rangle
=\int_0^\infty
\left(|f'(r)|^2+\bigl(V_1(r)-1\bigr)|f(r)|^2\right)r\,dr\ge0.
\]
Thus $\mathcal L_1$ has no eigenvalue below $1$, and hence no discrete
spectrum.

To exclude a threshold resonance, suppose that $\varphi=\overline{\varphi}$ is a radial
threshold solution of
\[
-\varphi''-\frac1r\varphi'+(V_1-1)\varphi=0,
\]
which is globally subordinate. Recall that this refers to the smaller branch relative to the asymptotic behavior as $r\to0+$, respectively $r\to\infty$. This branch is unique up to a scalar multiple. 
Here
\[
V_1(r)-1=\frac4{r^2}+O(1)\quad(r\to0+),
\qquad
V_1(r)-1=\frac1{r^2}+O\bigl(r^{-1/2}e^{-r}\bigr)
\quad(r\to\infty).
\]
Thus the subordinate asymptotics are
\[
\varphi(r)=O(r^2),\quad \varphi'(r)=O(r)
\qquad(r\to0+),
\]
and
\[
\varphi(r)=O(r^{-1}),\quad \varphi'(r)=O(r^{-2})
\qquad(r\to\infty).
\]
Multiplying the equation by $r{\varphi}$, and integrating over
$(\varepsilon,R)$ gives
\[
\int_\varepsilon^R
\left(|\varphi'|^2+(V_1-1)|\varphi|^2\right)r\,dr
=
\bigl[r\varphi'(r)\varphi(r)\bigr]_{\varepsilon}^{R}.
\]
The boundary terms vanish as $\varepsilon\to0+$ and $R\to\infty$.
Thus the nonnegative integral on the left is zero.  Since $V_1-1>0$, this
forces $\varphi\equiv0$, contradicting the definition of a resonance.
Therefore the threshold is not resonant.
\end{proof}

\subsection{Absence of a threshold resonance for \texorpdfstring{$\mathcal{L}_2$}{L2}}\label{sec:num_verif_no_resonance}
In this subsection, we give the computer-assisted interval proof of the absence of threshold resonances for $\mathcal{L}_2$, already stated in Proposition~\ref{prop:lem_p2_t2}.

\begin{lem}\label{Lemspec0}
The threshold of $\mathcal L_2$, equivalently of its conjugate
$\mathcal H_2$, is not a resonance.
\end{lem}

\begin{proof}
By \eqref{eq:LtoH}, it suffices to work with $\mathcal H_2$. Since
$U^2-1$ decays exponentially, the standard zero-energy construction gives
at infinity the fundamental system described in
\eqref{Phi_Theta_zero_Op_1and2}, with asymptotics
\[
 r^{1/2}\log r+O(e^{-r}\log r),
 \qquad
 r^{1/2}+O(e^{-r}).
\]
Let $\Phi_{2,0}^{(0)}$ denote the nontrivial solution of
\begin{equation}
     (\mathcal H_2-1)\Phi=0 
\label{eq:Phi200}
\end{equation}
which is regular at the origin, with its scaling fixed by the normalization~\eqref{eq:Phi200 rsmall} as $r\to0+$. Here the superscript $(0)$ indicates that we are at the
threshold spectral value, that is, $k=0$. Let
$\Phi_{2,\infty}^{(0)}$ denote the threshold solution normalized at infinity.
We write
\[
 W[f,g]:=fg'-f'g.
\]
Since the equation contains no first-order derivative, this Wronskian is
independent of $r$. 
The two solutions used in the computation are determined by the asymptotic
normalizations
\begin{equation}\label{eq:Phi200 rsmall}
    \Phi_{2,0}^{(0)}(r)\sim r^{1/2},
 \qquad
 \frac{d}{dr}\Phi_{2,0}^{(0)}(r)\sim\frac{1}{2r^{1/2}},
 \qquad r\to0+,
\end{equation}
and
\[
 \Phi_{2,\infty}^{(0)}(r)\sim-r^{1/2},
 \qquad
 \frac{d}{dr}\Phi_{2,\infty}^{(0)}(r)\sim-\frac{1}{2r^{1/2}},
 \qquad r\to\infty.
\]
By Lemma~\ref{lem:domains}, the normalization at the origin selects
the solution satisfying the Friedrichs boundary condition and excludes the
$r^{1/2}\log r$ branch.
The regular solution is approximated numerically with initial conditions in the value-and-derivative interval supplied by Lemma~\ref{lem:Acheck} at $\rorigin=0.1$ (obtained by a truncated Frobenius series), and propagated forward by the C++ code. See Lemma~\ref{lem:Acheck} below, applied
with threshold energy $E=1$, for the Frobenius approximation at~$\rorigin=0.1$. 

The solution $\Phi_{2,\infty}^{(0)}$, normalized at infinity, will be
propagated backward from $R=16$. To obtain rigorous initial data for
this propagation, namely intervals enclosing its value and derivative
at $R=16$, we use a Volterra equation. For this purpose, define
\[
 q(r):=1-U(r)^2,
 \qquad
 \Phi_{2,\infty}^{(0)}(r)=r^{1/2}w(r).
\]
The threshold equation and normalization at infinity give
\[
 (rw')'=-rq(r)w,
 \qquad
 w(r)\to-1,
 \qquad
 rw'(r)\to0
 \quad(r\to\infty).
\]
Integrating twice from $r$ to infinity and applying Fubini's theorem in
the second integration gives
\begin{equation}\label{eq:threshold_wprime_volterra}
 w'(r) =\frac1r\int_r^\infty s q(s)w(s)\,ds, \qquad \ 
 w(r)
 =-1-\int_r^\infty
 s q(s)\log\frac{s}{r}\,w(s)\,ds.
\end{equation}
For $R\ge\rfour$, define
\[
 (\mathcal T_Rv)(r)
 :=\int_r^\infty
 s q(s)\log\frac{s}{r}\,v(s)\,ds.
\]
By Corollary~\ref{cor:vortex_tail_bridge},
\[
 \|\mathcal T_R\|_{L^\infty\to L^\infty}
 \le L_R,
 \qquad
 L_R:=8e^{-R}\frac{R+2}{R},
\]
and
\[
 \int_R^\infty s q(s)\,ds
 \le M_R,
 \qquad
 M_R:=8e^{-R}(R+1).
\]
For $R\ge16$, one has $L_R\le L_{16}<1$, so the affine map
$v\mapsto-1-\mathcal T_Rv$ is a contraction. Thus
\begin{equation}\label{eq:threshold_w_error_bounds}
 \|w+1\|_{L^\infty([R,\infty))}
 \le\frac{L_R}{1-L_R},
 \qquad
 |w'(R)|
 \le\frac{M_R}{R(1-L_R)},
\end{equation}
where the derivative estimate follows from
\eqref{eq:threshold_wprime_volterra} and
$\|w\|_{L^\infty([R,\infty))}\le(1-L_R)^{-1}$.
At $R=16$, the multiprecision transcript records the outward-rounded
bounds
\[
 L_R<1.012817\cdot10^{-6},
 \qquad
 \frac{L_R}{1-L_R}<1.012818\cdot10^{-6},
 \qquad
 \frac{M_R}{R(1-L_R)}<9.56550\cdot10^{-7}.
\]
Consequently, the exact Cauchy data satisfy
\begin{equation}\label{eq:threshold_infinity_Cauchy_box}
\begin{aligned}
 \Phi_{2,\infty}^{(0)}(R)
 &\in R^{1/2}
 \left(-1+[-1.012818,1.012818]\cdot10^{-6}\right),\\
 (\Phi_{2,\infty}^{(0)})'(R)
 &\in
 \frac{-1+[-1.012818,1.012818]\cdot10^{-6}}{2R^{1/2}}
 +R^{1/2}[-9.56550,9.56550]\cdot10^{-7}.
\end{aligned}
\end{equation}
The C++ verifier, using CAPD's validated integrator, then propagates
\eqref{eq:threshold_infinity_Cauchy_box} backward from $R=16$ to the
common Wronskian matching radius $r_{\rm m}=10$. 
If a threshold resonance existed, the solution regular at the origin would
also have size $O(r^{1/2})$ at infinity. It would therefore be a multiple of
$\Phi_{2,\infty}^{(0)}$, and their Wronskian would vanish. The interval
computation instead gives
\begin{align}\label{noresonance_wronskian}
 W\bigl[\Phi_{2,0}^{(0)},\Phi_{2,\infty}^{(0)}\bigr](10)
 \in[-0.8003267,\,-0.8001801].
\end{align}
Since this interval excludes zero, $\Phi_{2,0}^{(0)}$ is not a multiple of
$\Phi_{2,\infty}^{(0)}$. Equivalently, in its expansion at infinity, the
coefficient of the logarithmic solution
\[
 \Theta_{2,\infty}^{(0)}(r)\sim r^{1/2}\log r,
 \qquad r\to\infty,
\]
is nonzero. Thus $\widetilde c_{1,*}\neq0$, and the threshold of
$\mathcal L_2$ is not a resonance.
\end{proof}


\subsection{The discrete spectrum of \texorpdfstring{$\mathcal{L}_2$}{L2}}\label{secLT}
We now prove that $\mathcal{L}_2=-\Delta_r+U^2$ has at most one
eigenvalue in $(0,1)$. A rigorous enclosure of this eigenvalue is obtained
in Lemma~\ref{lem:fgr_eigenvalue_box} below. As a consequence of some
Set\^o-type estimates~\cite{seto}, summarized in
\eqref{eq:seto_eigenvalue_bound}, it suffices to estimate the quantity
$\Lambda$ defined in \eqref{eq:seto_Lambda_definition}. The C++ code certifies upper enclosures of the compact double integral and of the two
compact one-variable moments; the regions near zero and
infinity are controlled analytically using the rigorous vortex tail
bounds.

\begin{prop}\label{theoLT2}
The operator $\mathcal{L}_2$ has at most one eigenvalue in $(0,1)$.
\end{prop}

\begin{proof}
Set
\[
  V(r):=1-U^2(r)\geq0,
  \qquad
  \mathcal L_2-1=-\Delta_r-V.
\]
Since $V(r)\leq1$ near the origin, while
Lemma~\ref{lem_1-U2} gives
\[
  V(r)\leq 8r^{-1/2}e^{-r},\qquad r>\rfour,
\]
we have
\[
  \int_0^\infty r\bigl(1+|\ln r|\bigr)V(r)\,dr<\infty.
\]
Thus the hypotheses of Set\^o's two-dimensional zero-angular-momentum
bound~\cite[Theorem~5.1, equation~(5.6)]{seto} are satisfied. If
$\mathcal N$ denotes the number of negative eigenvalues of the radial
operator $\mathcal L_2-1$, counted with multiplicity, then Set\^o's bound
gives
\begin{align}\label{eq:seto_eigenvalue_bound}
\mathcal{N}<1+\dfrac{\int_0^\infty
 \left(\int_0^\infty V(s)\big\vert\ln\tfrac rs\big\vert s\,ds\right)
 V(r)r\,dr}
 {2\int_0^\infty V(r)r\,dr}.
\end{align}
First, the Bogomolny equations
\eqref{equ:onevortex_ODEs_smooth} and the boundary values of $a$ give
\begin{equation}\label{eq:vortex_mass}
  \int_0^\infty V(r)r\,dr
  =\int_0^\infty 2a'(r)\,dr
  =2\bigl(a(\infty)-a(0)\bigr)=2.
\end{equation}
It therefore remains to prove that
\begin{align}\label{eq:seto_Lambda_definition}
\Lambda:=\int_0^\infty
 \left(\int_0^\infty V(s)\big\vert\ln\tfrac rs\big\vert s\,ds\right)
 V(r)r\,dr<3.2.
\end{align}

Write
\[
  f(r):=rV(r)=r\bigl(1-U^2(r)\bigr).
\]
Put $\varepsilon:=10^{-30}$ and $R:=15$. Since $0\leq U\leq1$, we have
\[
  0\leq f(t)=t\bigl(1-U^2(t)\bigr)\leq t.
\]
On each interval $[\alpha,\beta]\subset[\varepsilon,0.1]$, the
rectangle computation therefore uses the pointwise bound
$f(t)\leq\beta$. On $[0.1,R]$ it uses tubular neighborhoods
(``enclosures") of the vortex profile generated by the C++ code,
initialized at $\rorigin=0.1$ by the Frobenius approximations.

We estimate analytically the region in which at least one of the two
variables lies in $(0,\varepsilon)$. Set
\[
  m_\varepsilon:=\int_0^\varepsilon f(t)\,dt,
  \qquad
  \ell_\varepsilon:=\int_0^\varepsilon f(t)|\ln t|\,dt.
\]
The bound $f(t)\leq t$ gives
\[
  m_\varepsilon\leq\frac{\varepsilon^2}{2},
  \qquad
  \ell_\varepsilon
  \leq\frac{\varepsilon^2}{2}|\ln\varepsilon|
       +\frac{\varepsilon^2}{4}.
\]
Moreover, by the Bogomolny identity 
\[
  M_\infty:=\int_0^\infty f(t)\,dt=2,
\]
see \eqref{eq:vortex_mass}.
Using again $f(t)\leq t$ on $(0,R)$, and
Lemma~\ref{lem_1-U2} together with $|\ln t|\leq t$ for $t\geq R$, we
also have
\[
\begin{aligned}
  L_\infty
  :=\int_0^\infty f(t)|\ln t|\,dt
  &\leq \int_0^R t|\ln t|\,dt
       +8\int_R^\infty t^2e^{-t}\,dt  \\
  &<\frac{R^2}{2}\ln R+1.
\end{aligned}
\]
Consequently, symmetry and
$\bigl|\ln(r/s)\bigr|\leq|\ln r|+|\ln s|$ show that the union of the
two strips where $r<\varepsilon$ or $s<\varepsilon$ contributes at
most
\[
\begin{aligned}
  2\bigl(m_\varepsilon L_\infty
       +\ell_\varepsilon M_\infty\bigr)
  &\leq
  2\left[
    \frac{\varepsilon^2}{2}
       \left(\frac{R^2}{2}\ln R+1\right)
    +2\left(
       \frac{\varepsilon^2}{2}|\ln\varepsilon|
       +\frac{\varepsilon^2}{4}\right)
  \right] \\
  &=4.448107532\ldots\cdot10^{-58}
   <4.449\cdot10^{-58}.
\end{aligned}
\]
The factor $2$ harmlessly counts the intersection of the two strips
twice. The positive cutoff $\varepsilon$ is introduced only because
the rectangle formula for a diagonal cell contains the logarithm of
its left endpoint and therefore cannot be applied with that endpoint
equal to zero. The resulting
rectangle computation gives
\begin{align}\label{mainintegral_LT}
\int_0^{15}\left(\int_0^{15}
 \bigl(1-U^2(s)\bigr)\big\vert\ln\tfrac rs\big\vert s\,ds\right)
 \bigl(1-U^2(r)\bigr)r\,dr<3.17965.
\end{align}
The same computation gives the certified bounds over $[0,15]$
\[
  M_{15}:=\int_0^{15}f(r)\,dr<2.008204,
  \qquad
  L_{15}:=\int_0^{15}f(r)|\ln r|\,dr<1.476031.
\]

For $r\geq15$, Lemma~\ref{lem_1-U2} gives
\begin{align}\label{1-U2}
  f(r)
  &<2\sqrt r\bigl(4-2r^{-1}+3r^{-2}\bigr)e^{-r}
   \leq8\sqrt r\,e^{-r}
   \leq8r e^{-r}.
\end{align}
Set
\[
\begin{aligned}
  m_0&:=8\int_{15}^\infty r e^{-r}\,dr
       =8(15+1)e^{-15}<3.91555\cdot10^{-5},\\
  m_1&:=8\int_{15}^\infty r^2e^{-r}\,dr
       =8(15^2+2\cdot15+2)e^{-15}<6.28936\cdot10^{-4}.
\end{aligned}
\]
Thus
\[
  \int_{15}^\infty f(r)\,dr\leq m_0,
  \qquad
  \int_{15}^\infty r f(r)\,dr\leq m_1.
\]
Moreover, if $0<r<15\leq s$, then
\[
  \big\vert\ln\tfrac rs\big\vert
  =\ln\tfrac sr\leq s+|\ln r|,
\]
whereas, for $r,s\geq15$,
\[
  \big\vert\ln\tfrac rs\big\vert
  \leq\ln r+\ln s\leq r+s.
\]
Consequently, the two mixed regions and the tail--tail region satisfy
\begin{align*}
&\iint_{(0,\infty)^2\setminus(0,15)^2}
 f(r)f(s)\big\vert\ln\tfrac rs\big\vert\,dr\,ds\\
&\qquad\leq
 2\bigl(M_{15}m_1+L_{15}m_0\bigr)+2m_0m_1
 <0.002642.
\end{align*}
This is the combined tail estimate recorded by the
certificate; in particular, it already includes the tail--tail region.

Combining this estimate with \eqref{mainintegral_LT}, we obtain
\[
  \Lambda<3.17965+0.002642=3.182292<3.2.
\]
Therefore
\[
  \mathcal N<1+\frac14\Lambda<1.796<2.
\]
Since $\mathcal N$ is an integer, it follows that $\mathcal N\leq1$.
In particular, $\mathcal L_2$ has at most one eigenvalue in $(0,1)$.
\end{proof}

\section{Approximating the internal mode   and the resonant dFT basis elements   }\label{sec:internal_mode_dft_data}

In this section we locate the internal eigenvalue $\lambda^2$ and its eigenfunction with certified high accuracy. In addition, since the resonant terms live at square-root spectral frequency $2\lambda$, i.e., at spectral energy $4\lambda^2$, we also compute the corresponding distorted Fourier basis elements at that frequency with rigorous error bounds. This is needed in the FGR verification. 

The eigenfunction $\psi$ satisfies the scalar eigenvalue equation with the exact vortex profile as a coefficient:
\begin{align}\label{eq:internal_mode}
(\calL_2\psi)(r)=-\psi''(r)-\dfrac{1}{r}\psi'(r)+U(r)^2\psi(r)=\lambda^2 \psi(r). 
\end{align}
We reserve $\psi$ for the $L^2$-normalized positive eigenfunction, while the
  C++ code rescales $\psi$ to $\psi_0$ which satisfies the normalization $\psi_0(0)=1, \psi_0'(0)=0$. At infinity, \eqref{eq:internal_mode} is an exponentially decaying perturbation of the modified Bessel equation with parameter $\sqrt{1-\lambda^2}$, and square integrability excludes the $I_0$ branch, leaving the $K_0\big(\sqrt{1-\lambda^2}\,r\big)$ behavior. In particular,
\[
  \psi(r)=O\!\left(r^{-1/2}e^{-\sqrt{1-\lambda^2}\,r}\right),
  \qquad r\to\infty.
\]
 For the Fermi Golden Rule, we need to solve the following equations for the
regular distorted Fourier basis at spectral energy $4\lambda^2$ (see
\eqref{eq:lemSd2}):
\begin{equation}\label{eq:Phi12}
\begin{aligned}
\calH_1\Phi_1= -\Phi_1''(r)-\dfrac{1}{4r^2}\Phi_1(r)
  +\dfrac{1+U(r)^2}{2}\Phi_1(r)
  +\dfrac{(2-a(r))^2}{r^2}\Phi_1(r)
  &=4\lambda^2\Phi_1(r),\\
\calH_2\Phi_2 = -\Phi_2''(r)-\dfrac{1}{4r^2}\Phi_2(r)+U(r)^2\Phi_2(r)
  &=4\lambda^2\Phi_2(r)
\end{aligned}
\end{equation}
with small $r$ asymptotics
\begin{align*}
\Phi_1(r) &=r^{1/2}\Big(r^{2}+a_4r^4+O(r^{6})\Big),
& \Phi_1'(r)&=r^{1/2}\Big(\dfrac{5}{2}r+\dfrac{9}{2}a_4r^3+O(r^{5})\Big),
\\
\Phi_2(r) &=r^{1/2}\Big(1+b_2r^2+O(r^{4})\Big),
& \Phi_2'(r)&=r^{1/2}\Big(\dfrac{1}{2}r^{-1}+\dfrac{5}{2} b_2r+O(r^{3})\Big).
\end{align*}
These expansions are the only admissible ones by Lemma~\ref{lem:domains}.

As in the case of the vortex, we use Frobenius series to approximate these
solutions on \([0,\rorigin]\).  Here, the coefficients of the ODEs themselves contain the Frobenius series of \(U\) and \(a\), so the
 recurrences governing the coefficients contain convolutions.  Lemmas~\ref{lem:capd_origin_frobenius} and~\ref{lem:Acheck} control the
resulting truncation errors by combining an
analytic majorant valid for all orders with finite outward-rounded interval calculations 
performed by the C++ code.

In the following lemma, a Frobenius factor means the analytic
power-series part left after removing the prescribed singular or
vanishing power at a regular singular point. The relevance of this
factorization to the lemma is that, after the resulting expressions are
substituted into the differential equations and equal powers are
compared, the coefficients obey triangular recurrences of convolution
type. In our setting near $r=0$, the regular solutions have the form
\[
    Y(r)=r^\nu h(r^2),
\]
where $r^\nu$ is the leading Frobenius power dictated by the indicial
equation, and
\[
    h(x)=\sum_{n\ge0}c_nx^n,\qquad x=r^2,
\]
is analytic at $x=0$. Thus, for example,
\[
    U(r)=r f(r^2),
\]
so the Frobenius factor is $f(x)$, while
\[
    \psi_0(r,\mu)=h_\mu(r^2)
\]
has leading power $\nu=0$ and Frobenius factor $h(x)$. Similarly,
\[
    \Phi_1(r)=r^{5/2}h_1(r^2),\qquad
    \Phi_2(r)=r^{1/2}h_2(r^2),
\]
so $h_1$ and $h_2$ are the corresponding Frobenius factors.

We use the following lemma to turn bounds for finitely many coefficients,
together with a bootstrap estimate, into geometric bounds for the entire
coefficient tail and hence for the truncation error. Unlike
Lemmas~\ref{lem:Wberror} and~\ref{lem:Wbderror}, it requires finite
interval-arithmetic checks of its hypotheses, especially of the
bootstrap in part~(2). The concrete recurrences and finite checks are
given in Lemma~\ref{lem:Acheck}.

More precisely, suppose that the Frobenius series is truncated after degree
$M$, and write
\[
  A_n:=|c_n|x_*^n,\qquad T:=A_M.
\]
Our goal is to prove the geometric bound
\[
  A_{M+m}\leq T\rho^m,\qquad m\geq1,
\]
for some fixed $\rho<1$. The code first computes the next $K$
coefficients and verifies this inequality separately for
$m=1,\ldots,K$. These $K$ coefficients are not included in the
polynomial used at $x_*$. Rather, they provide the initial segment of
the induction and move the starting index of the more delicate argument
from $M+1$ to $M+K+1$.

The remaining coefficients are controlled by their triangular recurrences, which contain coefficient convolutions in the variable $x=r^2$. The $n$th coefficient depends only on coefficients with smaller indices. The C++ code inserts the geometric bounds for all preceding coefficients into a nonnegative majorant of the $n$th coefficient, obtained by placing absolute values inside the recurrence. We show that the resulting expression is at most $\beta T\rho^{n-M}$ with $\beta\leq1$. Hence the same geometric bound holds for the $n$th coefficient, and the induction continues.

In the applications below, there are several reasons why this induction closes. Multiplication of the coefficient recurrences by $x_*^n$ produces explicit factors $x_*$ or $x_*^2$. Moreover, the coefficients of the vortex profile decay geometrically with
\[
  \sigma=\frac{x_*}{2}=0.005,
\]
and the divisors in the recurrences grow quadratically with $n$. Taking $K=20$ separates the coefficients retained in the polynomial from the first index of the infinite induction. The convolution terms containing those retained coefficients consequently carry powers at least $\sigma^{K-1}$ or $\sigma^{K+1}$. Lemma~\ref{lem:Acheck} verifies these facts for the four families used below: $\psi_0(\cdot,\mu)$, the regular solutions $\Phi_1$ and $\Phi_2$
at the FGR energy $E=4\mu$, and the regular threshold solution for
$\Phi_2$ at $E=1$.

\begin{lemma}\label{lem:capd_origin_frobenius}
Let 
\[
    h(x)=\sum_{n\ge0}c_nx^n
\]
and define
\[
    A_n:=|c_n|x_*^n, \qquad x_*:=\rorigin^2=10^{-2}.
\]
Assume that 
\begin{equation}\label{eq:frob_majorant_recurrence}
    A_n\le P_n(A_0,\ldots,A_{n-1}),\qquad n\ge1,
\end{equation}
where each $P_n\ge0 $ is monotone increasing in each argument.
Let $M,K$ be positive integers,  $T>0$, and  $0<\rho<1$.  Define
\[
  \widehat A_n:=
  \begin{cases}
    A_n, & 0\le n\le M+K,\\
    T\rho^{\,n-M}, & n>M+K.
  \end{cases}
\]
Suppose the following two inequalities hold:
\begin{enumerate}[leftmargin=*]
\item The finite coefficients satisfy
\begin{equation}\label{eq:certified_frob_finite_slices}
     A_{M+m}\le T\rho^m,\qquad 1\le m\le K.
\end{equation}
\item For every $n\ge M+K+1$,
\begin{equation}\label{eq:frob_tail_bootstrap}
  P_n(\widehat A_0,\ldots,\widehat A_{n-1})
  \le \beta\,T\rho^{\,n-M}
\end{equation}
with a constant $\beta\le 1$.
\end{enumerate}
Then the exact Frobenius coefficients satisfy
\begin{equation}\label{eq:frob_geometric_tail_bound}
 A_{M+m}\le T\rho^m,\qquad m\ge1.
\end{equation}
Consequently, with
\[
    h_M(x):=\sum_{n=0}^{M}c_nx^n,
    \qquad
    h_M'(x):=\sum_{n=1}^{M}n c_nx^{n-1},
\]
one has
\begin{equation}
    \label{eq:value}
    \left|h(x_*)-h_M(x_*)\right|
  \le
  E_0:=\frac{T\rho}{1-\rho}
\end{equation}
and
\begin{equation}\label{eq:deriv}
  \left|h'(x_*)-h_M'(x_*)\right|
  \le
  E_1:=
  \frac{T}{x_*}
  \left(
    \frac{M\rho}{1-\rho}
    +
    \frac{\rho}{(1-\rho)^2}
  \right).
\end{equation}
Thus, if $Y(r)=r^\nu h(r^2)$, $\nu\ge0$, and $\rorigin:=\sqrt{x_*}=0.1$, then
\[
Y(\rorigin)\in \rorigin^\nu\big(h_M(x_*)+[-E_0,E_0]\big),
\]
and
\[
Y'(\rorigin)\in
\nu \rorigin^{\nu-1}\big(h_M(x_*)+[-E_0,E_0]\big)
+2\rorigin^{\nu+1}\big(h_M'(x_*)+[-E_1,E_1]\big),
\]
with the first term omitted if $\nu=0$.
\end{lemma}

\begin{proof}
The finitely many bounds in item~(1) start the induction. Suppose that
$A_{M+j}\le T\rho^j$ is already known for $1\le j<m$, with $m>K$, and
set $n=M+m$. Every coefficient entering
\eqref{eq:frob_majorant_recurrence} is then bounded by the corresponding
$\widehat A_j$. Hence, by the monotonicity of $P_n$ and
\eqref{eq:frob_tail_bootstrap},
\[
\begin{aligned}
A_{M+m} \le P_{M+m}(A_0,\ldots,A_{M+m-1}) \le P_{M+m}(\widehat A_0,\ldots,\widehat A_{M+m-1}) \le \beta T\rho^m
\le T\rho^m.
\end{aligned}
\]
Induction gives the claimed estimate for every $m\ge1$, proving
\eqref{eq:frob_geometric_tail_bound}.

Hence, the remainder at $x=x_*$  satisfies
\[
\begin{aligned}
\left|h(x_*)-h_M(x_*)\right| \le \sum_{m\ge1}|c_{M+m}|x_*^{M+m} =\sum_{m\ge1}A_{M+m} \le T\sum_{m\ge1}\rho^m
=\frac{T\rho}{1-\rho},
\end{aligned}
\]
which proves \eqref{eq:value}. Similarly,
\[
\begin{aligned}
\left|h'(x_*)-h_M'(x_*)\right| & \le
\sum_{m\ge1}(M+m)|c_{M+m}|x_*^{M+m-1} =\frac{1}{x_*}
\sum_{m\ge1}(M+m)A_{M+m}\\
&\le
\frac{T}{x_*}\sum_{m\ge1}(M+m)\rho^m =
\frac{T}{x_*}
\left(
\frac{M\rho}{1-\rho}
+
\frac{\rho}{(1-\rho)^2}
\right),
\end{aligned}
\]
which proves \eqref{eq:deriv}.

Finally, since $x_*=\rorigin^2$, the bounds
\eqref{eq:value} and \eqref{eq:deriv} give the stated interval for
$Y(\rorigin)=\rorigin^\nu h(x_*)$. Differentiating
$Y(r)=r^\nu h(r^2)$ gives
$Y'(r)=\nu r^{\nu-1}h(r^2)+2r^{\nu+1}h'(r^2)$, and applying the same
bounds at $r=\rorigin$ gives the stated interval for
$Y'(\rorigin)$. When $\nu=0$, the first term vanishes.
\end{proof}

The next lemma applies
Lemma~\ref{lem:capd_origin_frobenius} to the Frobenius series for
$\psi_0(\cdot,\mu)$, $\Phi_1$, $\Phi_2$, and the regular threshold solution
$\Phi_{2,0}^{(0)}$ defined in~\eqref{eq:Phi200}. The subsequent validated ODE integrations require interval enclosures containing the exact values and derivatives of these solutions at $\rorigin=0.1$. As in the preceding interval computations, the C++ code produces rigorous enclosures for finitely many coefficients and the resulting truncated polynomials. To obtain enclosures for the exact solutions, however, the omitted infinite tails must also be bounded analytically. Since the coefficient recurrences contain convolutions with the vortex coefficients, obtaining uniform tail bounds is less direct than for the vortex series. The purpose of the argument below is therefore to verify the hypotheses of Lemma~\ref{lem:capd_origin_frobenius} and thereby obtain explicit uniform bounds $\mathcal E_0,\mathcal E_1$ for the omitted value and derivative tails. Adding these bounds to the finite polynomial enclosures gives rigorous intervals for the exact solutions and their derivatives at $\rorigin=0.1$.

To obtain these bounds, we first substitute the Frobenius series into
the differential equations, which leads to triangular recurrences for
their coefficients. The
previously established coefficient bounds for $U$ and $a$ then give
nonnegative majorants for every term in these recurrences. Second,
 interval arithmetic is used to compute the coefficients
through degree $M+K$ and to verify directly
\[
  A_{M+m}\leq A_M\rho^m,\qquad 1\leq m\leq K,
\]
with $K=20$ and $\rho=0.95$. The adjacent coefficient ratios printed
in the transcript are not used to estimate the infinite tail.

Finally, the proof estimates the recurrence for every
$j\geq j_0:=M+K+1$. The quantities $H_\alpha$ introduced below measure
the contribution of the coefficients with
$\ell\leq M$, and $N_\alpha$ bounds the contribution of the
coefficients with $\ell>M$. The
proof shows that the corresponding normalized bounds do not increase
for $j\geq j_0$. The interval computation verifies
\[
  \beta_\alpha:=H_\alpha+N_\alpha<1
\]
at $j=j_0$, which proves the all-order estimate required in part~(2)
of Lemma~\ref{lem:capd_origin_frobenius}. The remainder bounds
$\mathcal E_0,\mathcal E_1$ then follow from the two geometric sums in
that lemma.

The wider interval $J$ is used in
Lemma~\ref{lem:fgr_eigenvalue_box} to locate the internal eigenvalue
and obtain the sharper enclosure
$\lambda^2\in\Lambda_{\rm FGR}\subset J$. This is why both parameter
intervals occur in the following estimates. The resulting intervals
for the solutions and their derivatives at $\rorigin=0.1$ provide the initial
data for the subsequent validated ODE integrations. In the following lemma, $\psi_0(r,\mu)$ denotes the solution of~\eqref{eq:internal_mode} with eigenvalue~$\lambda^2$ replaced by~$\mu$. In other words, we solve
\begin{equation}\label{eq:evalmu}
(\calL_2\psi_0(\cdot,\mu))(r)=-\psi_0''(r,\mu)-\dfrac{1}{r}\psi_0'(r,\mu)+U(r)^2\psi_0(r,\mu)=\mu \psi_0(r,\mu)
\end{equation}
with normalizations $\psi_0(0,\mu)=1$, $\psi_0'(0,\mu)=0$. Note that there is a unique value of~$\mu\in (0,1)$, namely the eigenvalue~$\lambda^2$, for which $\psi_0(\cdot,\mu)$ decays (and then $\psi_0(r):=\psi_0(r,\lambda^2)$). In fact, in Lemma~\ref{lem:fgr_eigenvalue_box} below, $\psi_0(\cdot,\mu)$ plays an essential role in proving that $\lambda^2$ is contained in $J$ as well as in the much smaller interval~$\Lambda_{\rm FGR}$. This is why the following lemma controls the Frobenius series of~$\psi_0(r,\mu)$ near $r=0$ uniformly for $\mu\in J$.

\begin{lemma}
\label{lem:Acheck}
Define
$
 x=r^2$, $x_*:=\rorigin^2$, $\rorigin=0.1,
$
and let $
 I_{\rm cert}$ be as in~\eqref{eq:Icert}. Further, define $
 J:=[0.77747,\,0.77753]$, and  $
 \Lambda_{\rm FGR}:=[0.777471875,\,0.77747375]\subsetneq J$.
Set 
\[
 U(r)=r f(x),\qquad a(r)=g(x),\qquad
 f(x)=\sum_{n\ge0}p_nx^n,\qquad
 g(x)=\sum_{n\ge1}q_nx^n.
\]
All the estimates below are uniform for the shooting parameter
$\upr\in I_{\rm cert}$, and hence apply to the vortex shooting parameter
$\uprz$.
We write, cf.~\eqref{eq:evalmu}
\[
 \psi_0(r,\mu)=h_\mu(x),\qquad h_\mu(0)=1,
\]
and
\[
 \Phi_{1,E}(r)=r^{5/2}h_{1,E}(x),\qquad
 \Phi_{2,E}(r)=r^{1/2}h_{2,E}(x),\qquad
 h_{1,E}(0)=h_{2,E}(0)=1.
\]
Here $\mu\in J$ for $h_\mu$, $E\in4\Lambda_{\rm FGR}$ for $h_{1,E}$ and
$h_{2,E}$, and $E=1$   for 
$\Phi_{2,0}^{(0)}$. 
For a Frobenius factor $h(x)=\sum_{n\ge0}\gamma_nx^n$, let
\[
 h^{(M)}(x):=\sum_{n=0}^{M}\gamma_nx^n,
\]
where $M=100$ for $h_\mu$ and $M=60$ for $h_{1,E}$ and $h_{2,E}$.
Uniformly in the parameter ranges just specified, the exact remainders at
$x_*$ satisfy
\[
 |h(x_*)-h^{(M)}(x_*)|<\mathcal E_0,
 \qquad
 |h'(x_*)-\partial_x h^{(M)}(x_*)|<\mathcal E_1,
\]
where
\[
{\renewcommand{\arraystretch}{1.25}%
\begin{array}{c|c|c|c}
\text{Frobenius factor and parameter range}
 &M&\mathcal E_0&\mathcal E_1\\ \hline
h_\mu,\ \mu\in J
 &100&2.744\cdot10^{-279}&3.293\cdot10^{-275}\\
h_{1,E},\ E\in4\Lambda_{\rm FGR}
 &60&2.948\cdot10^{-167}&2.358\cdot10^{-163}\\
h_{2,E},\ E\in4\Lambda_{\rm FGR}
 &60&1.096\cdot10^{-166}&8.762\cdot10^{-163}\\
h_{2,E},\ E=1
 &60&3.122\cdot10^{-167}&2.497\cdot10^{-163}
\end{array}}
\]
Consequently, if $Y(r)=r^\nu h(r^2)$ is any of these regular solutions,
then
\[
 Y(\rorigin)\in
 \rorigin^\nu\bigl(h^{(M)}(x_*)+[-\mathcal E_0,\mathcal E_0]\bigr)
\]
and
\[
 Y'(\rorigin)\in
 \nu \rorigin^{\nu-1}
 \bigl(h^{(M)}(x_*)+[-\mathcal E_0,\mathcal E_0]\bigr)
 +2\rorigin^{\nu+1}
 \bigl(\partial_x h^{(M)}(x_*)+[-\mathcal E_1,\mathcal E_1]\bigr),
\]
with the first term omitted when $\nu=0$.
\end{lemma}
\begin{proof}
Put
\[
  \sigma:=\frac{x_*}{2}=0.005,
  \qquad
  C:=\sup_{\upr\in I_{\rm cert}}|\upr|.
\]
We first relate the vortex coefficients above to those estimated in
Lemma~\ref{lem:Wberror}.  From
\[
  U(r)=\upr r\sum_{n\ge0}(-1)^nc_nr^{2n},
  \qquad
  a(r)=\frac{r^2}{4}\sum_{n\ge0}(-1)^nd_nr^{2n},
\]
one has
\begin{equation}\label{eq:pnqn}
  p_n=\upr(-1)^nc_n,
  \qquad
  q_n=\frac14(-1)^{n-1}d_{n-1}\quad(n\ge1).
\end{equation}
Since $|c_n|,|d_n|\le2^{-n}$,
\begin{equation}\label{eq:certified_vortex_coefficients}
  |p_n|x_*^n\le C\sigma^n,
  \qquad
  |q_n|x_*^n\le\frac12\sigma^n.
\end{equation}
The combinations of vortex coefficients that enter the Frobenius
recurrences are $f^2$ and $(2-g)^2-4$. We therefore define
\begin{equation}\label{eq:certified_vortex_product_coefficients}
\begin{aligned}
 B(x):=f(x)^2=\sum_{k\ge0}B_kx^k,
  \qquad \
  S(x):=(2-g(x))^2-4=\sum_{i\ge1}S_ix^i.
\end{aligned}
\end{equation}
Writing out the coefficient convolutions in $B=f^2$ and $S=-4g+g^2$ and applying \eqref{eq:certified_vortex_coefficients} gives
\begin{equation}\label{eq:certified_vortex_kernels}
\begin{aligned}
  |B_k|x_*^k
  &\le C^2(k+1)\sigma^k=:\overline B_k,
  && k\ge0,\\
  |S_i|x_*^i
  &\le\left(2+\frac{i-1}{4}\right)\sigma^i
   =:\overline S_i,
  && i\ge1.
\end{aligned}
\end{equation}
Indeed, the first estimate follows from
\[
 \sum_{m=0}^k(C\sigma^m)(C\sigma^{k-m})
 =C^2(k+1)\sigma^k.
\]
For the second, use $(2-g)^2-4=-4g+g^2$: the two contributions are
bounded by $2\sigma^i$ and $(i-1)\sigma^i/4$, respectively.

We next derive the three recurrences used by the code.  Write
\[
 h_\mu(x)=\sum_{j\ge0}\gamma_jx^j,
  \qquad
  h_1(x)=\sum_{j\ge0}\gamma_{1,j}x^j,
  \qquad
  h_2(x)=\sum_{j\ge0}\gamma_{2,j}x^j.
\]
   We replace $4\lambda^2$ by~$E=4\mu$ for the
solutions in \eqref{eq:Phi12}, whereas we set $E=1$ for the threshold
solution in \eqref{eq:Phi200}.  We substitute
\[
 U(r)=rf(r^2),\qquad a(r)=g(r^2),\qquad
 \psi_0(r,\mu)=h_\mu(r^2)
\]
into \eqref{eq:evalmu}, and
\[
 \Phi_1(r)=r^{5/2}h_1(r^2),\qquad
 \Phi_2(r)=r^{1/2}h_2(r^2)
\]
into the first and second equations of \eqref{eq:Phi12}, respectively.
For the threshold solution we substitute
$\Phi_{2,0}^{(0)}(r)=r^{1/2}h_2(r^2)$ into \eqref{eq:Phi200}.
With $x=r^2$, and with primes in the following display denoting
$x$--derivatives, the resulting equations are
\begin{equation}\label{eq:certified_frob_factor_equations}
\begin{aligned}
 4\bigl(h_\mu'+xh_\mu''\bigr)
   &=(xB-\mu)h_\mu,\\
 4x^2h_1''+12xh_1'
   &=\left(S+\left(\frac12-E\right)x+\frac12x^2B\right)h_1,\\
 4\bigl(h_2'+xh_2''\bigr)
   &=(xB-E)h_2.
\end{aligned}
\end{equation}
The last equation in \eqref{eq:certified_frob_factor_equations} also
holds for the threshold factor, with $E=1$. Substituting the power
series for $B,S,h_\mu,h_1,h_2$ into
\eqref{eq:certified_frob_factor_equations}, and comparing the
coefficient of $x^{j-1}$ in the first and third equations and that of
$x^j$ in the second, gives, for $j\ge1$,
\begin{equation}\label{eq:certified_frob_recurrences}
\begin{aligned}
  4j^2\gamma_j
  &=-\mu\gamma_{j-1}
    +\sum_{\substack{k+\ell=j-2\\k,\ell\ge0}}B_k\gamma_\ell,\\
  (4j^2+8j)\gamma_{1,j}
  &=\sum_{i=1}^{j}S_i\gamma_{1,j-i}
    +\left(\frac12-E\right)\gamma_{1,j-1}
    +\frac12\sum_{\substack{k+\ell=j-2\\k,\ell\ge0}}
       B_k\gamma_{1,\ell},\\
  4j^2\gamma_{2,j}
  &=-E\gamma_{2,j-1}
    +\sum_{\substack{k+\ell=j-2\\k,\ell\ge0}}B_k\gamma_{2,\ell}.
\end{aligned}
\end{equation}
As usual, a sum with $j-2<0$ is empty.  The two distinct factors by
which one divides to determine the $j$th coefficient are $4j^2$ for
$h_\mu$ and $h_2$, and $4j^2+8j$ for $h_1$.  Indeed,
\[
 \frac{d^2}{dr^2}\bigl(r^\nu h(r^2)\bigr)
 =r^{\nu-2}\left(
   \nu(\nu-1)h+2(2\nu+1)xh'+4x^2h''
 \right).
\]
For $\Phi_1$, $\nu=5/2$ and the constant $4$ in $(2-g)^2$
cancels $\nu(\nu-1)+1/4=4$, leaving
$4j(j-1)+12j=4j^2+8j$.  For $\Phi_2$, $\nu=1/2$ and
$\nu(\nu-1)+1/4=0$, leaving
$4j(j-1)+4j=4j^2$.  The first equation above likewise gives
$4j^2$ for $h_\mu$.  This proves
\eqref{eq:certified_frob_recurrences}.  Notice in particular that
$\psi_0$ and $\Phi_2$ use only the series for $U$, whereas $\Phi_1$ uses
the series for both $U$ and $a$.

We now specify the monotone maps to which
Lemma~\ref{lem:capd_origin_frobenius} is applied.  For a fixed parameter
in the relevant interval, let
\begin{equation}\label{eq:certified_frob_weighted_coefficients}
 a_j^\psi:=|\gamma_j|x_*^j,\qquad
 a_j^1:=|\gamma_{1,j}|x_*^j,\qquad
 a_j^2:=|\gamma_{2,j}|x_*^j.
\end{equation}
The absolute values of $\mu$, $E$, and $\frac12-E$ below denote uniform
upper bounds on the corresponding parameter interval.  Multiplying \eqref{eq:certified_frob_recurrences} by $x_*^j$ and using
\eqref{eq:certified_vortex_kernels} gives the following concrete
majorant maps. For $z=(z_0,\ldots,z_{j-1})\in[0,\infty)^j$, set
\begin{equation}\label{eq:certified_frob_majorant_maps}
\begin{aligned}
 \mathcal P_{\psi,j}(z)
 &:=
 \frac{1}{4j^2}\left(
 x_*|\mu|z_{j-1}
 +x_*^2\sum_{\ell=0}^{j-2}
       \overline B_{j-2-\ell}z_\ell\right),\\
 \mathcal P_{1,j}(z)
 &:=
 \frac{1}{4j^2+8j}\left(
 \sum_{\ell=0}^{j-1}\overline S_{j-\ell}z_\ell
 +x_*\left|\frac12-E\right|z_{j-1}
 +\frac{x_*^2}{2}\sum_{\ell=0}^{j-2}
       \overline B_{j-2-\ell}z_\ell\right),\\
 \mathcal P_{2,j}(z)
 &:=
 \frac{1}{4j^2}\left(
 x_*|E|z_{j-1}
 +x_*^2\sum_{\ell=0}^{j-2}
       \overline B_{j-2-\ell}z_\ell\right).
\end{aligned}
\end{equation}
In other words, $\mathcal P_{\alpha,j}$ is obtained from the
corresponding recurrence in
\eqref{eq:certified_frob_recurrences} by taking absolute values
term by term and applying the coefficient bounds
\eqref{eq:certified_vortex_kernels}. Consequently,
\begin{equation}\label{eq:certified_frob_majorant_inequality}
 a_j^\alpha
 \le
 \mathcal P_{\alpha,j}
 \bigl(a_0^\alpha,\ldots,a_{j-1}^\alpha\bigr),
 \qquad \alpha\in\{\psi,1,2\}.
\end{equation}
Each $\mathcal P_{\alpha,j}$ is increasing in every argument, since it
is a linear combination with nonnegative coefficients. The maps in
\eqref{eq:certified_frob_majorant_maps} are the concrete maps denoted by
$P_n$ in Lemma~\ref{lem:capd_origin_frobenius}.

We next verify the two hypotheses of
Lemma~\ref{lem:capd_origin_frobenius} that imply the geometric tail
bound \eqref{eq:frob_geometric_tail_bound}. We begin with the finite
family of coefficient bounds
\eqref{eq:certified_frob_finite_slices}. The
argument is applied separately to the Frobenius factor $h_\mu$ for
$\mu\in J$, to $h_1$ and $h_2$ for
$E\in4\Lambda_{\rm FGR}$, and to the threshold factor $h_2$ for $E=1$.
In every case, take
\[
 K:=20,\qquad \rho:=0.95,
\]
with $M=100$ for $h_\mu$ and $M=60$ for $h_1,h_2$.

The finite computational step uses
\eqref{eq:certified_frob_recurrences} to construct successively interval
enclosures for the coefficients through index $M+K$. Starting from the
normalized zeroth coefficient, each enclosure is computed from the
preceding ones, with $\upr$ and the spectral parameter represented by
their full input intervals and all arithmetic rounded outward. For
$\alpha=\psi,1,2$, respectively, let
\[
 [\gamma_{\alpha,j}^-,\gamma_{\alpha,j}^+]
\]
denote the resulting enclosure of $\gamma_j,\gamma_{1,j}$, or
$\gamma_{2,j}$. By
\eqref{eq:certified_frob_weighted_coefficients}, multiplication by
$x_*^j$ gives an enclosure of $a_j^\alpha$. Thus, for
$0\le j\le M+K$, let $A_j$ be the outward-rounded upper bound produced
by the code, so that, uniformly over the corresponding parameter range,
\begin{equation}\label{eq:certified_frob_interval_coefficient_bounds}
 a_j^\alpha
 \le x_*^j\max\bigl\{|\gamma_{\alpha,j}^-|,
                         |\gamma_{\alpha,j}^+|\bigr\}
 \le A_j
\end{equation}

For each computation, set
\[
 T:=A_M,\qquad
 j_0:=M+K+1,\qquad
 \vartheta:=\frac{\sigma}{\rho}=\frac1{190}.
\]
Thus $T$ bounds the contribution
$a_M^\alpha=x_*^M|\gamma_{\alpha,M}|$ of the last coefficient retained in the polynomial, while
$j_0$ is the first index treated by the infinite-tail induction. The
code verifies
\begin{equation}\label{eq:certified_frob_finite_interval_check}
 A_{M+m}\le T\rho^m,\qquad 1\le m\le K.
\end{equation}
Only the coefficients through $M$ are retained in the polynomial
evaluated at $x_*$. The next $K$ coefficients are used solely to verify
\eqref{eq:certified_frob_finite_interval_check}. Together with
\eqref{eq:certified_frob_interval_coefficient_bounds}, this verifies
the finite-slice hypothesis \eqref{eq:certified_frob_finite_slices} of
Lemma~\ref{lem:capd_origin_frobenius}; all indices $j>M+K$ are then
controlled by the analytic induction below.

It remains to verify the all-order bootstrap estimate
\eqref{eq:frob_tail_bootstrap}. To obtain a convenient majorant for the
sequence $\widehat A$ appearing there, combine the certified bounds
$A_\ell$ from
\eqref{eq:certified_frob_interval_coefficient_bounds} for the retained
coefficients $0\le\ell\le M$ with the target geometric profile
$T\rho^{\ell-M}$ in \eqref{eq:frob_geometric_tail_bound}, and define
\begin{equation}\label{eq:certified_frob_comparison_sequence}
 \widetilde A_\ell:=
 \begin{cases}
   A_\ell,&0\le\ell\le M,\\
   T\rho^{\ell-M},&\ell>M.
 \end{cases}
\end{equation}
By \eqref{eq:certified_frob_finite_interval_check}, this sequence also
majorizes the interval bounds $A_\ell$ from
\eqref{eq:certified_frob_interval_coefficient_bounds} for
$M<\ell\le M+K$. For $j\ge j_0$, split each convolution in
\eqref{eq:certified_frob_recurrences} according to whether
$\ell\le M$ or $\ell>M$.  At $j=j_0$, we first isolate in each majorant map the terms involving
$A_0,\ldots,A_M$, namely the coefficients retained in the polynomial
$h_M$. These terms are, respectively,
\[
\begin{aligned}
 \mathcal P_{\psi,j_0}^{\rm fin}
 &=\frac{x_*^2}{4j_0^2}
   \sum_{\ell=0}^{M}\overline B_{j_0-2-\ell}A_\ell,\\
 \mathcal P_{1,j_0}^{\rm fin}
 &=\frac{1}{4j_0^2+8j_0}\left(
   \sum_{\ell=0}^{M}\overline S_{j_0-\ell}A_\ell
   +\frac{x_*^2}{2}\sum_{\ell=0}^{M}
      \overline B_{j_0-2-\ell}A_\ell\right),\\
 \mathcal P_{2,j_0}^{\rm fin}
 &=\frac{x_*^2}{4j_0^2}
   \sum_{\ell=0}^{M}\overline B_{j_0-2-\ell}A_\ell.
\end{aligned}
\]
Here the first and third formulas have the same form, but the
coefficients $A_\ell$ and the values of $M,T,j_0$ belong to different
Frobenius factors.  The terms involving the immediately preceding
coefficient have index $j_0-1=M+K>M$, so they belong to the part with
$\ell>M$ and will be included in $N_\psi,N_1,N_2$.  We define $H_\psi,H_1,H_2$ by dividing these three contributions by
the proposed bound $T\rho^{j_0-M}$. Thus
\begin{equation}\label{eq:certified_frob_heads}
\begin{aligned}
 H_\psi&:=
 \dfrac{x_*^2}
      {4j_0^2T\rho^{j_0-M}}\displaystyle\sum_{\ell=0}^{M}
       \overline B_{j_0-2-\ell}A_\ell,
\\
H_1&:= \frac{1}{(4j_0^2+8j_0)T\rho^{j_0-M}}\left(\displaystyle\sum_{\ell=0}^{M}  \overline S_{j_0-\ell}A_\ell +\frac{x_*^2}{2}\displaystyle\sum_{\ell=0}^{M} \overline B_{j_0-2-\ell}A_\ell\right),
\\
H_2&:= \dfrac{x_*^2}{4j_0^2T\rho^{j_0-M}} \displaystyle\sum_{\ell=0}^{M} \overline B_{j_0-2-\ell}A_\ell.
\end{aligned}
\end{equation}
For $j\ge j_0$, let $H_\alpha(j)$ denote the corresponding
right-hand side with $j_0$ replaced by $j$.  Thus
$H_\alpha=H_\alpha(j_0)$.
Indeed, since $j_0=M+K+1$ and $0\le\ell\le M$, we have
\begin{equation}\label{eq:certified_frob_head_index_gaps}
 j_0-2-\ell\ge K-1=19,
 \qquad
 j_0-\ell\ge K+1=21.
\end{equation}
Thus every occurrence of $\overline B_k$ has $k\ge0$, and every
occurrence of $\overline S_i$ has $i\ge1$. Note that if $\mathcal P_{\alpha,j}^{\rm fin}$ denotes the sum of the terms in
\[
 \mathcal P_{\alpha,j}
 \bigl(\widetilde A_0,\ldots,\widetilde A_{j-1}\bigr)
\]
that involve $A_0,\ldots,A_M$, then, by definition,
\[
 \mathcal P_{\alpha,j}^{\rm fin}
 =H_\alpha(j)T\rho^{j-M}.
\]
Before estimating these terms, we indicate the sources of smallness in
the argument and record precisely the effect of normalization by
$T\rho^{j-M}$. After multiplication of the recurrences by $x_*^j$, a
term involving the immediately preceding coefficient carries a factor
$x_*$, while a $\overline B$--term carries a factor $x_*^2$. A
$\overline S$--term carries no such exterior factor; however, since
$S_0=0$, every term in the $\overline S$--sum contains at least one
factor $\sigma$. More precisely, for $j\ge j_0$ and $\ell>M$,
\eqref{eq:certified_frob_comparison_sequence} and
\eqref{eq:certified_vortex_kernels} give
\begin{equation}\label{eq:certified_frob_normalized_tail_factors}
\begin{aligned}
 \frac{\widetilde A_{j-1}}{T\rho^{j-M}}
 =\frac1\rho, \qquad 
 \frac{\overline B_{j-2-\ell}\widetilde A_\ell}
      {T\rho^{j-M}}
 &=C^2(j-1-\ell)
   \frac{\vartheta^{j-2-\ell}}{\rho^2}, \qquad 
 \frac{\overline S_{j-\ell}\widetilde A_\ell}
      {T\rho^{j-M}}
=\left(2+\frac{j-\ell-1}{4}\right)
   \vartheta^{j-\ell}.
\end{aligned}
\end{equation}
In the majorant maps
\eqref{eq:certified_frob_majorant_maps}, these three ratios are
multiplied, respectively, by $x_*$, $x_*^2$, and $1$. Thus the factors
$\rho^{-1}$ and $\rho^{-2}$ are accompanied by explicit powers of
$x_*$ and by the geometric decay of the kernel coefficients away from
the diagonal. For the normalized finite contributions $H_\alpha$ defined in
\eqref{eq:certified_frob_heads}, one has $\ell\le M$. Hence, at the
first uncomputed index $j_0=M+K+1$, the stronger index separations
\eqref{eq:certified_frob_head_index_gaps} provide at least the factors
$\sigma^{K-1}$ in the $\overline B$--sums and $\sigma^{K+1}$ in the
$\overline S$--sum. Finally, all these terms are
divided by the increasing indicial factors $4j^2$ or $4j^2+8j$. The
estimates below make this bookkeeping precise.

It remains to estimate the terms containing coefficients with
$\ell>M$ when the majorant map is evaluated on the comparison sequence
$\widetilde A$. Define
\begin{equation}\label{eq:certified_frob_indicial_factors}
 d_\psi(j)=d_2(j):=4j^2,
 \qquad
 d_1(j):=4j^2+8j.
\end{equation}
For the $\overline B$--convolution, putting $k=j-2-\ell$ gives
\begin{equation}\label{eq:certified_frob_B_tail_estimate}
\begin{aligned}
 &\frac{x_*^2}
 {d_\alpha(j)T\rho^{j-M}}\displaystyle\sum_{\ell=M+1}^{j-2}
 \overline B_{j-2-\ell}T\rho^{\ell-M} 
 \le
 \frac{x_*^2C^2}{d_\alpha(j)\rho^2}
 \sum_{k\ge0}(k+1)\vartheta^k
 =
 \frac{x_*^2C^2}
 {d_\alpha(j)\rho^2(1-\vartheta)^2},
\end{aligned}
\end{equation}
where $\alpha=\psi,1,2$, as appropriate. In the $\Phi_1$ recurrence,
the $\overline B$--contribution in
\eqref{eq:certified_frob_B_tail_estimate} carries the additional factor
$1/2$ from \eqref{eq:certified_frob_majorant_maps}. Similarly, the
$\overline S$--convolution in the $\Phi_1$ recurrence satisfies
\begin{equation}\label{eq:certified_frob_S_tail_estimate}
\begin{aligned}
 \frac{1}
 {(4j^2+8j)T\rho^{j-M}}\displaystyle\sum_{\ell=M+1}^{j-1}
 \overline S_{j-\ell}T\rho^{\ell-M}  &  \le   
 \frac{1}{4j^2+8j}
 \sum_{i\ge1}\left(2+\frac{i-1}{4}\right)\vartheta^i
 \\
 &=
 \frac{1}{4j^2+8j}
 \left(
 \frac{2\vartheta}{1-\vartheta}
 +\frac{\vartheta^2}{4(1-\vartheta)^2}
 \right).
\end{aligned}
\end{equation}
Finally, the term involving the immediately preceding coefficient in
the $h_\mu$ recurrence gives
\begin{equation}\label{eq:certified_frob_preceding_tail_estimate_psi}
 \frac{x_*|\mu|\widetilde A_{j-1}}
      {4j^2T\rho^{j-M}}
 =\frac{x_*|\mu|}{4j^2\rho},
\end{equation}
and the corresponding terms in the $h_1$ and $h_2$ recurrences give,
respectively,
\begin{equation}\label{eq:certified_frob_preceding_tail_estimates_12}
\begin{aligned}
 \frac{x_*\left|\frac12-E\right|\widetilde A_{j-1}}
      {(4j^2+8j)T\rho^{j-M}}
 &=
 \frac{x_*\left|\frac12-E\right|}
      {(4j^2+8j)\rho}, \qquad \ \ 
 \frac{x_*|E|\widetilde A_{j-1}}
      {4j^2T\rho^{j-M}}
 &=
 \frac{x_*|E|}{4j^2\rho}.
\end{aligned}
\end{equation}
Accordingly, we collect the right-hand sides of \eqref{eq:certified_frob_B_tail_estimate}--\eqref{eq:certified_frob_preceding_tail_estimates_12}, evaluated at $j=j_0$, by defining
\begin{equation}\label{eq:certified_frob_tail_norms}
\begin{aligned}
 N_\psi&:=
  \frac{x_*|\mu|}{4j_0^2\rho}
  +\frac{x_*^2C^2}{4j_0^2\rho^2(1-\vartheta)^2},\\
 N_1&:=\frac{1}{4j_0^2+8j_0}\left(
   \frac{2\vartheta}{1-\vartheta}
   +\frac{\vartheta^2}{4(1-\vartheta)^2}
   +\frac{x_*\left|\frac12-E\right|}{\rho}
   +\frac{x_*^2C^2}{2\rho^2(1-\vartheta)^2}
 \right),\\
 N_2&:=\frac{1}{4j_0^2}\left(
   \frac{x_*|E|}{\rho}
   +\frac{x_*^2C^2}{\rho^2(1-\vartheta)^2}
 \right).
\end{aligned}
\end{equation}
For $j\ge j_0$, let $N_\alpha(j)$ denote the same expression with
$j_0$ replaced by $j$.  Thus $N_\alpha=N_\alpha(j_0)$.
These are precisely the quantities obtained from
\[
 \sum_{i\ge1}\vartheta^i=\frac{\vartheta}{1-\vartheta},
 \qquad
 \sum_{i\ge1}(i-1)\vartheta^i
 =\frac{\vartheta^2}{(1-\vartheta)^2},
 \qquad
 \sum_{k\ge0}(k+1)\vartheta^k
 =\frac1{(1-\vartheta)^2}.
\]
If $\mathcal P_{\alpha,j}^{\rm tail}$ denotes the remaining terms after
the majorant map is evaluated on $\widetilde A$, then the preceding
estimates state precisely that
\[
 \mathcal P_{\alpha,j}^{\rm tail}
 \le N_\alpha(j)T\rho^{j-M}.
\]
We now explain why it suffices to evaluate the preceding bounds at
$j=j_0$.  The kernel majorants satisfy
\[
 \frac{\overline B_{k+1}}{\overline B_k}
 =\sigma\frac{k+2}{k+1}\le2\sigma,
 \qquad
 \frac{\overline S_{i+1}}{\overline S_i}
 =\sigma\frac{i+8}{i+7}
 \le\frac98\sigma<2\sigma.
\]
For the normalized finite parts, passing from $j$ to $j+1$ therefore
costs a factor $\rho^{-1}$ from the normalization, while the kernel
majorant contributes a factor at most $2\sigma$. Since the
denominators~$d_\alpha(j)$ increase,
\[
 H_\alpha(j+1)
 \le
 \frac{2\sigma}{\rho}
 \frac{d_\alpha(j)}{d_\alpha(j+1)}
 H_\alpha(j)
 \le
 \frac{2\sigma}{\rho}H_\alpha(j),
 \qquad \alpha\in\{\psi,1,2\}.
\]
Here
\[
 \frac{2\sigma}{\rho}=\frac{0.01}{0.95}=\frac1{95}<1.
\]
Consequently,
\[
 H_\alpha(j)\le H_\alpha(j_0)=H_\alpha,
 \qquad j\ge j_0.
\]
Likewise,
\[
 N_\alpha(j)\le N_\alpha(j_0)=N_\alpha,
 \qquad j\ge j_0,
\]
because the geometric sums have already been extended to infinity and
the denominators~$d_\alpha(j)$ increase with~$j$.

Set
\[
 \beta_\psi:=H_\psi+N_\psi,\qquad
 \beta_1:=H_1+N_1,\qquad
 \beta_2:=H_2+N_2.
\]
It follows that, for every $j\ge j_0$, the corresponding monotone map
satisfies
\begin{equation}\label{eq:certified_frob_all_order}
 \mathcal P_{\alpha,j}
 \bigl(\widetilde A_0,\ldots,\widetilde A_{j-1}\bigr)
 \le
 \bigl(H_\alpha(j)+N_\alpha(j)\bigr)T\rho^{j-M}
 \le
 \beta_\alpha T\rho^{j-M}.
\end{equation}
It remains to transfer \eqref{eq:certified_frob_all_order}, which was
proved for the comparison sequence $\widetilde A$, to the hybrid
sequence required in \eqref{eq:frob_tail_bootstrap}. Since
Lemma~\ref{lem:capd_origin_frobenius} applies to one exact coefficient
sequence at a time, choose any parameter value in the range of the
current computation: $\mu\in J$ for $h_\mu$,
$E\in4\Lambda_{\rm FGR}$ for $h_1$ and $h_2$, or $E=1$ for the
threshold factor $h_2$. The bounds below are uniform over the
corresponding range.

For the exact weighted coefficients $a_\ell^\alpha$ defined in
\eqref{eq:certified_frob_weighted_coefficients}, define the sequence
denoted by $\widehat A$ in Lemma~\ref{lem:capd_origin_frobenius}, here
written as $\widehat a$ to distinguish it from the computed bounds
$A_\ell$, by
\begin{equation}\label{eq:certified_frob_hybrid_sequence}
 \widehat a_\ell:=
 \begin{cases}
   a_\ell^\alpha,&0\le\ell\le M+K,\\
   T\rho^{\ell-M},&\ell>M+K.
 \end{cases}
\end{equation}
For $0\le\ell\le M$,
\eqref{eq:certified_frob_interval_coefficient_bounds} and
\eqref{eq:certified_frob_comparison_sequence} give
\[
 a_\ell^\alpha\le A_\ell=\widetilde A_\ell.
\]
For $M<\ell\le M+K$,
\eqref{eq:certified_frob_interval_coefficient_bounds},
\eqref{eq:certified_frob_finite_interval_check}, and
\eqref{eq:certified_frob_comparison_sequence} give
\[
 a_\ell^\alpha\le A_\ell
 \le T\rho^{\ell-M}
 =\widetilde A_\ell.
\]
For $\ell>M+K$,
\eqref{eq:certified_frob_hybrid_sequence} and
\eqref{eq:certified_frob_comparison_sequence} agree. Hence
\begin{equation}\label{eq:certified_frob_hybrid_comparison}
 \widehat a_\ell\le\widetilde A_\ell,
 \qquad \ell\ge0.
\end{equation}
By the coordinatewise monotonicity of $\mathcal P_{\alpha,j}$,
\eqref{eq:certified_frob_hybrid_comparison}, and
\eqref{eq:certified_frob_all_order},
\[
 \mathcal P_{\alpha,j}
 \bigl(\widehat a_0,\ldots,\widehat a_{j-1}\bigr)
 \le
 \mathcal P_{\alpha,j}
 \bigl(\widetilde A_0,\ldots,\widetilde A_{j-1}\bigr)
 \le
 \beta_\alpha T\rho^{j-M}.
\]
This is precisely the inequality required in
\eqref{eq:frob_tail_bootstrap}; the interval computation below verifies
that $\beta_\alpha<1$.

In every case the finite verification uses $K=20$ and $\rho=0.95$.
The interval calculation first computes the coefficients through
$M+K$, sets $T=A_M$, and verifies
\eqref{eq:certified_frob_finite_slices}. It then evaluates the explicit
bounds $H_\alpha$ and $N_\alpha$ in
\eqref{eq:certified_frob_heads} and
\eqref{eq:certified_frob_tail_norms}. The transcript\footnotemark[\getrefnumber{fn:transcript}]
\cite{AYMHPartICertificateRepo} records
\begin{equation}\label{eq:certified_frob_bootstrap_constants}\begin{aligned}
 \beta_\psi:=H_\psi+N_\psi&<0.358546
 &&\text{for }\mu\in J,\\
 \beta_1:=H_1+N_1&<1.417\cdot10^{-6}
 &&\text{for }E\in4\Lambda_{\rm FGR},\\
 \beta_2:=H_2+N_2&<1.249\cdot10^{-6}
 &&\text{for }E\in4\Lambda_{\rm FGR},\\
 \beta_2:=H_2+N_2&<4.027\cdot10^{-7}
 &&\text{for }E=1.
\end{aligned}
\end{equation}
In particular, \eqref{eq:certified_frob_bootstrap_constants} shows that all four bootstrap constants are strictly smaller than one. The coefficient enclosures \eqref{eq:certified_frob_interval_coefficient_bounds}, the finite checks \eqref{eq:certified_frob_finite_interval_check}, and the bounds \eqref{eq:certified_frob_bootstrap_constants} are the only numerical inputs to this bootstrap argument. The passage from these finite inputs to the infinite coefficient tails is the analytic induction above.

Together with \eqref{eq:certified_frob_finite_slices}, these inequalities verify the hypotheses of Lemma~\ref{lem:capd_origin_frobenius}. Indeed, the first $K$ coefficient bounds start the induction. If the proposed estimate has been proved through index $j-1$, then coordinatewise monotonicity gives
\[
 a_j^\alpha  \le  \beta_\alpha T\rho^{j-M}  \le  T\rho^{j-M}.
\]
It follows by induction that
\[
 a_{M+m}^\alpha\le T\rho^m,  \qquad m\ge1.
\]
At this point no further approximation of the Frobenius series is involved. Lemma~\ref{lem:capd_origin_frobenius} gives the exact remainder bounds
\[
 \mathcal E_0
 =
 \frac{T\rho}{1-\rho},
 \qquad
 \mathcal E_1
 =
 \frac{T}{x_*}
 \left(
   \frac{M\rho}{1-\rho}
   +\frac{\rho}{(1-\rho)^2}
 \right).
\]
Our C++ code evaluates these two expressions with outward-rounded multiprecision interval arithmetic. Thus, the values of $\mathcal E_0,\mathcal E_1$ in the transcript and stated in the lemma are consequences of the geometric majorant. Their small size reflects the small value of $T=A_M=|\gamma_M|x_*^M$ at $x_*=10^{-2}$.

For $h_\mu$, the calculation is performed both on the initial interval $J$ used in the eigenvalue search and, subsequently, on the smaller interval $\Lambda_{\rm FGR}$ determined by Lemma~\ref{lem:fgr_eigenvalue_box}. On $J$, the bootstrap calculation gives
\[
 \beta_\psi<0.358538353958974,
\]
and substitution of the corresponding value of $T$ into the preceding
formulas gives
\[
 \mathcal E_0<2.74366499059244\cdot10^{-279},
 \qquad
 \mathcal E_1<3.29239798871093\cdot10^{-275}.
\]
On $\Lambda_{\rm FGR}$, the corresponding calculation gives
\[
 \beta_\psi<0.358545152034321,
\]
and the same formulas give
\[
 \mathcal E_0<2.74355048425486\cdot10^{-279},
 \qquad
 \mathcal E_1<3.29226058110583\cdot10^{-275}.
\]
The $h_\mu$ row in the statement uses the outward-rounded bounds obtained
on the wider interval $J$, and therefore also applies on
$\Lambda_{\rm FGR}\subset J$. The calculation is repeated on
$\Lambda_{\rm FGR}$ before the subsequent FGR computations, where the
slightly smaller remainder bounds are used. The common estimate
$\beta_\psi<0.358546$ applies to both parameter intervals.

The case $E=1$ concerns the regular threshold solution
$\Phi_{2,0}^{(0)}$ of
\[
  (\mathcal H_2-1)\Phi_{2,0}^{(0)}=0.
\]
Finally, the corresponding bounds for the radial derivative follow from
the exact identity
\[
 \frac{d}{dr}\bigl(r^\nu h(r^2)\bigr)
 =
 \nu r^{\nu-1}h(r^2)+2r^{\nu+1}h'(r^2).
\]
Consequently, the intervals constructed at $\rorigin=0.1$ contain the exact
regular solution and its radial derivative. Starting from these
intervals, the CAPD solver propagates $U,a$ together with $\psi_0$ or the
distorted Fourier basis elements and their derivatives.
\end{proof}

The following result locates the internal eigenvalue to high precision.  Near the
origin, the Frobenius estimates developed above give a rigorous enclosure of
the regular solution at $\rorigin=0.1$.  At infinity, a Volterra argument gives a
rigorous enclosure at $R=16$ of the solution normalized to decay like
$K_0(\sqrt{1-\mu}\,r)$. The C++ code evaluates these bounds using outward-rounded interval
arithmetic. It encloses the Bessel functions $K_0(\alpha_\mu R)$ and $K_1(\alpha_\mu R)$ by interval quadrature using their integral
representations. It then  bounds the Volterra remainders by elementary functions, all with
outward rounding. Finally, the code propagates the two solution enclosures to a common
matching radius and computes their Wronskian there as a function of the
spectral parameter. The vanishing of this Wronskian is equivalent to
linear dependence of the origin-regular and infinity-decaying solutions, yielding an eigenfunction.
By certifying opposite signs of the Wronskian at the endpoints, the code locates the eigenvalue within the chosen interval. 

\begin{lemma}\label{lem:fgr_eigenvalue_box}
Set
\[
  \Lambda_{\rm FGR}:=[0.777471875,\,0.77747375].
\]
Then $\mathcal L_2$ has an eigenvalue in $\Lambda_{\rm FGR}$. In fact, this
eigenvalue is the unique internal eigenvalue. With the notation of
Section~\ref{sec:FGR}, this means
$\lambda^2\in\Lambda_{\rm FGR}$.
\end{lemma}

\begin{proof}
We first specify the two solutions and their normalizations. Put
\[
  J:=[0.77747,\,0.77753],\qquad
  \rorigin:=0.1,\qquad r_{\rm m}:=10,\qquad R:=16,
\]
and, for $\mu\in J$, set
\[
  \alpha_\mu:=\sqrt{1-\mu}.
\]
Let $\psi_{0}(\cdot,\mu)$ be the solution of~\eqref{eq:evalmu}  normalized by
\[
  \psi_{0}(0,\mu)=1,\qquad \psi_{0}'(0,\mu)=0
\]
as before. 
Below, we construct the solution $\psi_{\infty}(\cdot,\mu)$ normalized by
\[
  \lim_{r\to\infty}
  \frac{\psi_{\infty}(r,\mu)}{K_0(\alpha_\mu r)}=1.
\]
With the convention
\[
  W[f,g]:=fg'-f'g,
\]
we define the Wronskian 
\[
  D(\mu):=
  W[\psi_{0}(\cdot,\mu),\psi_{\infty}(\cdot,\mu)](r_{\rm m}).
\]
For two solutions of $\mathcal L_2\psi=\mu\psi$, we have 
\[
  W'(r)=-\frac1rW(r),
  \qquad\text{and hence}\qquad
  \bigl(rW(r)\bigr)'=0.
\]
Thus the zero set of the Wronskian is independent of the matching radius,
although the unweighted Wronskian itself is not constant.

\smallskip
\noindent\emph{The regular solution at the origin.}
Write
\[
  U(r)=rf(r^2),\qquad
  \psi_{0}(r,\mu)=h_\mu(r^2),\qquad
  h_\mu(x)=\sum_{n\ge0}h_nx^n,\qquad h_0=1.
\]
Substitution of these series in $\mathcal L_2\psi=\mu\psi$ gives the first
triangular recurrence in \eqref{eq:certified_frob_recurrences}. More
explicitly, by \eqref{eq:certified_vortex_product_coefficients} and
\eqref{eq:pnqn},
\[
 B_m=\sum_{j=0}^{m}p_jp_{m-j},
 \qquad m\ge0,
\]
and hence
\[
 h_1=-\frac{\mu}{4},
 \qquad
 h_{n+1}
 =
 \frac{1}{4(n+1)^2}
 \left(
 \sum_{m=0}^{n-1}B_mh_{n-1-m}-\mu h_n
 \right),
 \qquad n\ge1.
\]
The vortex shooting parameter is carried by the code as the interval parameter
\[
  I_{\rm cert}
  =[0.6032878545816699,\,0.6032878545816856],
\]
which contains the true shooting parameter $\uprz$ by
Lemma~\ref{lem:Newton}.
Define the degree--$100$ truncation
\[
 \boldsymbol{\mathfrak h}_{100}(x)
 :=\sum_{n=0}^{100}h_nx^n.
\]
The code uses $\boldsymbol{\mathfrak h}_{100}$ for the initial
enclosure and computes the additional coefficients
$h_{101},\ldots,h_{120}$ needed for the $K=20$ finite bootstrap check
in Lemma~\ref{lem:Acheck}. That lemma verifies the hypotheses of
Lemma~\ref{lem:capd_origin_frobenius}, uniformly for $\mu\in J$. The
multiprecision transcript in the certificate repository
\cite{AYMHPartICertificateRepo} records
\[
  M=100,\qquad K=20,\qquad \rho=0.95,
\]
and
\[
  \beta
  \le 0.358538353958974<1.
\]
It also records the tail bounds at $x_*=10^{-2}$,
\[
  E_0
  \le 2.74366499059244\cdot10^{-279},
  \qquad
  E_1
  \le 3.29239798871093\cdot10^{-275},
\]
where $E_0$ bounds the value tail and $E_1$ the $x$-derivative tail. Hence the exact initial data satisfy 
\[
  \psi_0(\rorigin,\mu)
  \in \boldsymbol{\mathfrak{h}}_{100}(10^{-2})+[-E_0,E_0]
\]
and
\[
  \psi_0'(\rorigin,\mu)
  \in
  0.2\boldsymbol{\mathfrak{h}}_{100}'(10^{-2})+[-0.2E_1,0.2E_1].
\]
In particular, the truncation error in the initial value of
$\psi_{0}'(\cdot,\mu)$ is less than $6.585\cdot10^{-276}$. Thus the interval
started at $\rorigin$ contains the initial data of the exact regular solution, not merely
those of a truncated polynomial.

\smallskip
\noindent\emph{The decaying solution at infinity.}
Set
\[
  q(r):=1-U(r)^2.
\]
The eigenvalue equation can be written as
\[
  \left(-\partial_r^2-\frac1r\partial_r+\alpha_\mu^2\right)\psi
  =q\psi.
\]
Corollary~\ref{cor:vortex_tail_bridge} gives, for $r\ge R$,
\[
  |q(r)|\le 8\,\frac{e^{-r}}{\sqrt r}.
\]
Consequently, define
\[
\begin{aligned}
  \ell_R&:=
  \int_R^\infty s|q(s)|\log\frac{s}{R}\,ds,
  &
  m_R&:=
  \int_R^\infty s|q(s)|\,ds.
\end{aligned}
\]
Since $R\ge1$, we have $\sqrt{s}\le s$ for $s\ge R$, and
\[
  \log\frac{s}{R}\le\frac{s-R}{R}.
\]
Therefore,
\begin{align}
  m_R
  &\le8\int_R^\infty s e^{-s}\,ds
  =8e^{-R}(R+1)=:M_R,
  \label{eq:internal_tail_moment}\\
  \ell_R
  &\le\frac8R\int_R^\infty s(s-R)e^{-s}\,ds
  =8e^{-R}\frac{R+2}{R}=:L_R.
  \label{eq:internal_tail_log_moment}
\end{align}
At $R=16$ the outward-rounded values used by the code satisfy
\[
  L_R
  \le1.012816572473333\cdot10^{-6},
  \qquad
  M_R
  \le1.530478376181925\cdot10^{-5}.
\]
The transcript also gives the consistency check
\[
  |1-U(16)^2|
  \le1.288693152516394\cdot10^{-7}
  <2.250703494385183\cdot10^{-7}
  =8\frac{e^{-16}}{\sqrt{16}}.
\]
The global estimate for every $r\ge16$ is proved in 
Corollary~\ref{cor:vortex_tail_bridge}.

We now give the Volterra argument underlying the initial box at $R$.
The modified Bessel functions satisfy
\begin{equation}\label{eq:internal_bessel_wronskian}
\begin{aligned}
 &I_0(\alpha_\mu r)\,
   \partial_rK_0(\alpha_\mu r)
 -\partial_rI_0(\alpha_\mu r)\,
   K_0(\alpha_\mu r)
 =-\frac1r, \qquad \ \frac{d}{dr}
 \left(
 \frac{I_0(\alpha_\mu r)}{K_0(\alpha_\mu r)}
 \right)
 =
 \frac{1}{rK_0(\alpha_\mu r)^2}.
\end{aligned}
\end{equation}
Variation of constants for the solution with the prescribed
normalization at infinity yields
\begin{equation}\label{eq:internal_decaying_variation}
\begin{aligned}
 \psi_{\infty,\mu}(r)
 &=K_0(\alpha_\mu r) -\int_r^\infty s
 \Bigl(
 K_0(\alpha_\mu r)I_0(\alpha_\mu s)
 -I_0(\alpha_\mu r)K_0(\alpha_\mu s)
 \Bigr)
 q(s)\psi_{\infty,\mu}(s)\,ds.
\end{aligned}
\end{equation}
Thus, writing $\psi_{\infty,\mu}(r)=K_0(\alpha_\mu r)z(r)$, we obtain
\begin{equation}\label{eq:internal_decaying_volterra}
 z(r)=1-(\mathcal T_\mu z)(r),
\end{equation}
where
\begin{equation}\label{eq:internal_decaying_volterra_operator}
\begin{aligned}
 (\mathcal T_\mu z)(r)
 &:=
 \int_r^\infty sK_0(\alpha_\mu s)^2 \left( \frac{I_0(\alpha_\mu s)}{K_0(\alpha_\mu s)}
 -
 \frac{I_0(\alpha_\mu r)}{K_0(\alpha_\mu r)}
 \right)
 q(s)z(s)\,ds.
\end{aligned}
\end{equation}
For $s\ge r$, the function $K_0(\alpha_\mu s)$ is positive and
decreasing. Hence, by \eqref{eq:internal_bessel_wronskian},
\[
\begin{aligned}
 &K_0(\alpha_\mu s)^2
 \left(
 \frac{I_0(\alpha_\mu s)}{K_0(\alpha_\mu s)}
 -
 \frac{I_0(\alpha_\mu r)}{K_0(\alpha_\mu r)}
 \right) =
 K_0(\alpha_\mu s)^2
 \int_r^s\frac{dt}{tK_0(\alpha_\mu t)^2} \le
 \int_r^s\frac{dt}{t}
 =
 \log\frac{s}{r}.
\end{aligned}
\]
Consequently, on $C_b([R,\infty))$,
\[
 \|\mathcal T_\mu\| \le
 \sup_{r\ge R}
 \int_r^\infty s|q(s)|\log\frac{s}{r}\,ds
 \le L_R<1
\]
by \eqref{eq:internal_tail_log_moment}. Thus
$I+\mathcal T_\mu$ is invertible, and
\eqref{eq:internal_decaying_volterra} has a unique bounded solution.
Moreover,
\begin{equation}\label{eq:internal_decaying_z_bounds}
 \|z\|_\infty\le\frac1{1-L_R},
 \qquad
 \|z-1\|_\infty\le\frac{L_R}{1-L_R}.
\end{equation}
The tail integral tends to zero with $r$, so $z(r)\to1$ as
$r\to\infty$. This constructs the uniquely normalized decaying
solution
\[
 \psi_{\infty,\mu}(r)=K_0(\alpha_\mu r)z(r).
\]

Differentiating \eqref{eq:internal_decaying_volterra_operator} gives
\[
 z'(r)
 =
 \frac{1}{rK_0(\alpha_\mu r)^2}
 \int_r^\infty
 sK_0(\alpha_\mu s)^2q(s)z(s)\,ds.
\]
Since $K_0(\alpha_\mu s)\le K_0(\alpha_\mu r)$ for $s\ge r$,
\eqref{eq:internal_tail_moment} and
\eqref{eq:internal_decaying_z_bounds} give
\[
 |z'(R)|
 \le
 \frac{M_R}{R(1-L_R)}.
\]
Using
\[
 \partial_rK_0(\alpha_\mu r)
 =-\alpha_\mu K_1(\alpha_\mu r),
\]
we therefore obtain
\begin{equation}\label{eq:internal_decaying_raw_errors}
\begin{aligned}
 \left|\psi_{\infty,\mu}(R)-K_0(\alpha_\mu R)\right|
 &\le
 K_0(\alpha_\mu R)\frac{L_R}{1-L_R},\\
 \left|\psi_{\infty,\mu}'(R)
 +\alpha_\mu K_1(\alpha_\mu R)\right|
 &\le
 \alpha_\mu K_1(\alpha_\mu R)\frac{L_R}{1-L_R}  + 
 K_0(\alpha_\mu R)\frac{M_R}{R(1-L_R)}.
\end{aligned}
\end{equation}

We next pass to the slightly larger bounds used by the code. Set
\begin{equation}\label{eq:internal_eta_scale}
 \eta_\mu:=\frac{L_R}{\alpha_\mu},
 \qquad
 S_\mu:=
 \sqrt{\frac{\pi}{2\alpha_\mu R}}e^{-\alpha_\mu R}.
\end{equation}
Using the lower endpoint of $\alpha_\mu$ on $J$ and the bound on $L_R$,
outward-rounded evaluation gives
\[
 0<\eta_\mu\le2.148\cdot10^{-6}<1,
 \qquad \mu\in J.
\]
The integral representation of $K_0$, together with
$\cosh t\ge1+t^2/2$, gives
\[
 K_0(\alpha_\mu R)\le S_\mu.
\]
Moreover,
\[
 \frac{K_1(x)}{K_0(x)}
 <1+\frac1{2x},
\]
proved in \cite[Corollary~3.3]{YangChu2017}, yields, uniformly for
$\mu\in J$,
\[
 \alpha_\mu
 \frac{K_1(\alpha_\mu R)}{K_0(\alpha_\mu R)}
 <
 \alpha_\mu+\frac1{2R}
 <0.503<1.
\]
Thus, $\alpha_\mu K_1(\alpha_\mu R)
 \le K_0(\alpha_\mu R)
 \le S_\mu$. Define
\begin{equation}\label{eq:internal_decaying_error_radii}
 \varepsilon_{0,\mu}
 :=
 S_\mu\frac{\eta_\mu}{1-\eta_\mu},
 \qquad
 \varepsilon_{1,\mu}
 :=
 \alpha_\mu\varepsilon_{0,\mu}
 +S_\mu\frac{M_R}{1-\eta_\mu}.
\end{equation}
Since $L_R=\alpha_\mu\eta_\mu$, $0<L_R\le\eta_\mu<1$, and
$x\mapsto x/(1-x)$ is increasing on $(0,1)$, the first estimate in
\eqref{eq:internal_decaying_raw_errors} gives
\[
 \left|\psi_{\infty,\mu}(R)-K_0(\alpha_\mu R)\right|
 \le
 S_\mu\frac{\eta_\mu}{1-\eta_\mu}
 =
 \varepsilon_{0,\mu}.
\]
For the derivative estimate,
\[
\begin{aligned}
 \alpha_\mu K_1(\alpha_\mu R)\frac{L_R}{1-L_R}
 &\le
 S_\mu\frac{L_R}{1-L_R} \le
 S_\mu\frac{L_R}{1-\eta_\mu}
 =
 \alpha_\mu\varepsilon_{0,\mu},
\end{aligned}
\]
whereas, since $R\ge1$,
\[
 K_0(\alpha_\mu R)\frac{M_R}{R(1-L_R)}
 \le
 S_\mu\frac{M_R}{1-\eta_\mu}.
\]
The second estimate in \eqref{eq:internal_decaying_raw_errors}
therefore gives
\[
 \left|\psi_{\infty,\mu}'(R)
 +\alpha_\mu K_1(\alpha_\mu R)\right|
 \le
 \varepsilon_{1,\mu}.
\]
These are exactly the error radii used by the C++ routine:
\begin{equation}\label{eq:internal_decaying_Cauchy_box}
\begin{aligned}
 \psi_{\infty,\mu}(R)
 &\in
 K_0(\alpha_\mu R)
 +[-\varepsilon_{0,\mu},\varepsilon_{0,\mu}],\\
 \psi_{\infty,\mu}'(R)
 &\in
 -\alpha_\mu K_1(\alpha_\mu R)
 +[-\varepsilon_{1,\mu},\varepsilon_{1,\mu}].
\end{aligned}
\end{equation}

The central Bessel values $K_0(\alpha_\mu R)$ and $K_1(\alpha_\mu R)$ in
\eqref{eq:internal_decaying_Cauchy_box} are also enclosed by
outward-rounded interval arithmetic. Uniformly for all possible values
$x=\alpha_\mu R$ with $\mu\in J$, the code uses
\begin{equation}\label{eq:internal_Bessel_integral_representation}
 K_\nu(x)
 =
 \int_0^\infty e^{-x\cosh t}\cosh(\nu t)\,dt,
 \qquad \nu=0,1.
\end{equation}
For $\nu=0$ the integrand is nonincreasing in $t$. For $\nu=1$ its
derivative is
\[
 e^{-x\cosh t}\sinh t\,(1-x\cosh t)\le0,
\]
because $x>1$. The code therefore computes lower and upper Riemann
sums on the uniform partition of $[0,8]$ into $12000$ subintervals.
For the remaining tail it uses
\begin{equation}\label{eq:internal_Bessel_quadrature_tail}
 0\le
 \int_8^\infty e^{-x\cosh t}\cosh(\nu t)\,dt
 \le
 \frac2x\exp\left(-\frac{xe^8}{2}\right),
 \qquad \nu=0,1.
\end{equation}
Indeed, for $\nu=1$ this follows from
$\cosh t\le e^t$ and $\cosh t\ge e^t/2$, while the integrand for
$\nu=0$ is smaller.\footnote{In particular, the code relies only on
interval bounds for the Bessel functions. No non-interval
floating-point evaluation of the Bessel functions is used.}

\smallskip
\noindent\emph{Validated propagation and matching.}
Starting with the certified Frobenius box at $\rorigin=0.1$, the code propagates
the regular solution forward to $r_{\rm m}=10$. Starting from a rigorous interval enclosure at $R=16$ provided by the Volterra argument, it propagates the decaying solution backward to
$r_{\rm m}=10$. In both directions, the CAPD solver propagates the full vector
\[
  (r,U,a,\psi_0,\psi_0',\upr,\mu).
\]
This means that the vortex profile $(U,a)$ is  part of the same  ODE system
and is not just passed to the code as a non-rigorous numerical approximation. In the forward
direction the system used by the code is
\[
\begin{aligned}
  \dot r&=1,
  &\dot U&=\frac{1-a}{r}U,
  &\dot a&=\frac r2(1-U^2),\\
  \dot\psi_0&=\psi_0',
  &\dot\psi_0'&=-\frac1r\psi_0'+(U^2-\mu)\psi_0,
  &\dot\upr&=\dot\mu=0,
\end{aligned}
\]
and for backward propagation the sign of the vector field is reversed. The true shooting
parameter is enclosed by $I_{\rm cert}$, see Lemma~\ref{lem:Newton}. For each Wronskian evaluation,
the chosen value of $\mu$ is represented by an interval containing that
exact value. 

On the interval $J$ the code computes
\[
\begin{aligned}
  D(0.77747)
  &\in
  [-4.074288757607958\cdot10^{-7},
   -3.539855515662376\cdot10^{-7}],\\
  D(0.77753)
  &\in
  [7.716058532830248\cdot10^{-6},
   7.768215008615830\cdot10^{-6}].
\end{aligned}
\]
It then performs bisection five times. At the sixth midpoint, the interval enclosure for that value of~$D$ contains $0$, so its sign cannot be certified. The code therefore retains the last bracket whose endpoint
Wronskians have strictly opposite signs:
\[
  \Lambda_{\rm FGR}
  =[0.777471875,\,0.77747375].
\]
Its width is $1.875\cdot10^{-6}$.   At its endpoints
$\mu_-=0.777471875$ and $\mu_+=0.77747375$, the code finds that 
\[
  D(\mu_-)
  \in
  [-1.536081253326721\cdot10^{-7},
   -1.002050065350549\cdot10^{-7}]
\]
and
\[
  D(\mu_+)
  \in
  [1.002150911811321\cdot10^{-7},
   1.535780040260618\cdot10^{-7}].
\]
Hence, $D$ has a zero
$\mu_*\in\Lambda_{\rm FGR}$ (using standard continuous dependence on $\mu$). In view of Proposition~\ref{theoLT2}, $\mu_*$ is the unique eigenvalue of
$\mathcal L_2$. 
\end{proof}

The FGR tail estimate in Lemma~\ref{lem:capd_fgr_tail_upper_bound} requires an explicit numerical constant in the
exponential decay bound for the internal mode. The information needed for that is provided by Lemma~\ref{lem:psi_K0}.

\begin{figure}[!htbp]
\centering
\includegraphics[scale=0.7]{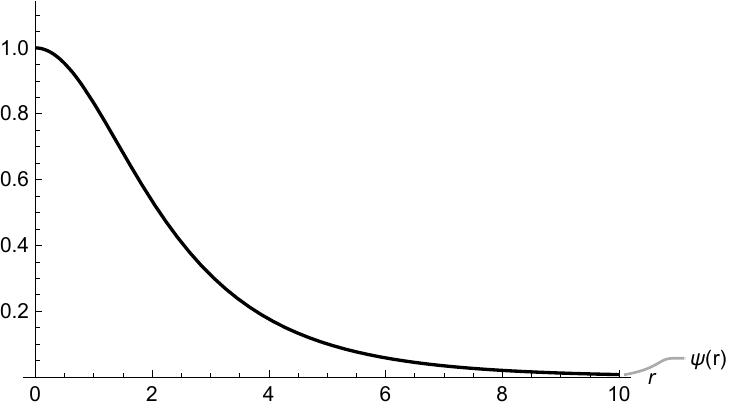}
\caption{Visualization of the internal mode associated with $\mathcal{L}_2$. }
\label{fig4LT}
\end{figure}

\begin{lem} \label{lem:psi_K0}
Let $\psi_0=\psi_0(r)$ be the unique radial eigenfunction\footnote{As ground state it is necessarily of definite sign, thus positive.} characterized by 
\begin{equation}\label{eq:radial_eig}
(\calL_2\psi_0)(r)=-\psi_0''(r)-\frac{1}r\psi_0'(r)+U^2(r)\psi_0(r)=\lambda^2\psi_0(r),\qquad r>0,
\end{equation}
with $\lambda^2\in(0,1)$, $\psi_0(0)=1$, $\psi_0'(0)=0$, and $\psi_0(r)\to0$ as $r\to\infty$.
Fix $\rabstract >0$. Assume that there exists a constant $\delta_0\in\bigl(0,1-\lambda^2\bigr)$ such that
\begin{equation}\label{eq:U_lower_bound}
U^2(r)\ge 1-\delta_0,\qquad \forall\,r\ge \rabstract .
\end{equation}
Define
\begin{equation}\label{eq:kappa_def}
\kappa_\lambda :=\sqrt{(1-\delta_0)-\lambda^2}>0,
\end{equation}
and let $K_0$ denote the modified Bessel function of the second kind.
Assume the one-point comparison at $r=\rabstract $ for some constant $C>0$: 
\begin{equation}\label{eq:one_point_comp}
\psi_0(\rabstract )\le CK_0(\kappa_\lambda  \rabstract ).
\end{equation}
Then
\begin{equation}\label{eq:psi_upper_K0}
0<\psi_0(r)\le C\,K_0(\kappa_\lambda  r),\qquad \forall\,r\ge \rabstract .
\end{equation}
Moreover, for all $r\ge \rabstract $,
\begin{equation}\label{eq:psi_upper_explicit}
\psi_0(r)\le C\,\sqrt{\frac{\pi}{2\kappa_\lambda  r}}\,e^{-\kappa_\lambda  r},
\end{equation}
and the tail admits the explicit bound
\begin{equation}\label{eq:psi_tail_bound_final}
\int_{\rabstract }^{\infty}\psi_0^2(r)\,r\,dr
\le
\frac{C^2\pi}{4\kappa_\lambda ^2}\,e^{-2\kappa_\lambda  \rabstract }.
\end{equation}
\end{lem}

\begin{proof}
The function $K_0(\kappa_\lambda  r)$ satisfies the following equation: 
\begin{equation}\label{eq:K0_identity}
-\frac{d^2}{dr^2}K_0(\kappa_\lambda  r)-\frac{1}r\frac{d}{dr}K_0(\kappa_\lambda  r)+\kappa_\lambda ^2 K_0(\kappa_\lambda  r)=0,
\qquad r>0.
\end{equation}
Hence, for $\wt{\psi_0}(r):=CK_0(\kappa_\lambda  r)$,
\begin{equation}\label{eq:L_on_psibar}
\Bigl(-\partial_r^2-\frac{1}r\partial_r+U^2-\lambda^2\Bigr)\wt{\psi_0}
=
\Bigl(U^2-\lambda^2-\kappa_\lambda ^2\Bigr)\wt{\psi_0}.
\end{equation}
Using \eqref{eq:U_lower_bound} and \eqref{eq:kappa_def}, we have for all $r\ge \rabstract $:
\begin{equation}\label{eq:supersolution}
U^2(r)-\lambda^2-\kappa_\lambda ^2 \ge (1-\delta_0)-\lambda^2-\kappa_\lambda ^2 = 0,
\end{equation}
so $\wt{\psi_0}$ is a \emph{supersolution} of \eqref{eq:radial_eig} on the interval $[\rabstract ,\infty)$.
Let $e(r):=\wt{\psi_0}(r)-\psi_0(r)$. Subtracting \eqref{eq:radial_eig} from~\eqref{eq:L_on_psibar} using
\eqref{eq:supersolution}  yields
\begin{equation}\label{eq:e_ineq}
-\;e''(r)-\frac{1}r e'(r)+\bigl(U^2(r)-\lambda^2\bigr)e(r)\ge 0,\qquad r\ge \rabstract .
\end{equation}
From \eqref{eq:one_point_comp} we have $e(\rabstract )\ge0$.
Also $\psi_0(r)\to0$ and $K_0(\kappa_\lambda  r)\to0$ as $r\to\infty$, hence $e(r)\to0$. 
If $e$ were negative somewhere on $(\rabstract ,\infty)$, it would attain a negative minimum
at some interior point $r_*>\rabstract $, where $e(r_*)<0$, $e'(r_*)=0$, $e''(r_*)\ge0$.
Evaluating \eqref{eq:e_ineq} at $r_*$ gives
\[
0\le -e''(r_*)-\frac{1}{r_*}e'(r_*)+\bigl(U^2(r_*)-\lambda^2\bigr)e(r_*)
\le \bigl(U^2(r_*)-\lambda^2\bigr)e(r_*).
\]
For $r\ge \rabstract $, \eqref{eq:U_lower_bound} implies $U^2(r)-\lambda^2\ge (1-\delta_0)-\lambda^2=\kappa_\lambda ^2>0$,
hence the last term is strictly negative because $e(r_*)<0$, a contradiction.
Therefore, $e(r)\ge0$ for all $r\ge \rabstract $, proving \eqref{eq:psi_upper_K0}.

The integral representation of $K_0$ and the elementary inequality
$\cosh t\ge1+t^2/2$ give, for $x>0$,
\begin{equation}\label{eq:K0_bound}
 K_0(x)=\int_0^\infty e^{-x\cosh t}\,dt
 \le e^{-x}\int_0^\infty e^{-xt^2/2}\,dt
 =\sqrt{\frac{\pi}{2x}}\,e^{-x},
\end{equation}
combined with \eqref{eq:psi_upper_K0} yields \eqref{eq:psi_upper_explicit}.
Then
\[
\psi_0^2(r)\,r \le C^2\,\frac{\pi}{2\kappa_\lambda }\,e^{-2\kappa_\lambda  r},
\]
and integrating from $\rabstract $ to $\infty$ gives \eqref{eq:psi_tail_bound_final}.
\end{proof}

We determine the constant $C$ in \eqref{eq:one_point_comp} by numerical computation, see the tables in Section~\ref{sec:capd_certificate}.For this computation, we take the comparison radius in
Lemma~\ref{lem:psi_K0} to be $\reight:=8.001$; thus
$\rabstract=\reight$ in this application. We also set
$\delta_0=0.0005$. The C++ code verifies the required inequalities at this
radius and records the margins
\begin{equation}\label{eq:psi_K0_endpoint_margins}
\begin{aligned}
 1.4K_0(\kappa_\lambda  \reight )-\psi_0(\reight )&>5.5262\cdot10^{-4},\\
 \psi_0'(\reight )-\bigl(1.4K_0(\kappa_\lambda  \reight )\bigr)'&>2.8955\cdot10^{-4},\\
 U(\reight )^2-(1-\delta_0)&>3.8992\cdot10^{-7}.
\end{aligned}
\end{equation}
The first and third inequalities, together with the monotonicity of $U$, verify the hypotheses of the lemma on $[\reight ,\infty)$.  The derivative margin is an additional endpoint check recorded by the code; it is not needed in the maximum-principle argument above.  Thus the admissible choice is (see Figure~\ref{fig6_cst_psi})
\begin{align}\label{eq:cst_psi_num}
C=1.4
\end{align} 
Note that the abstract parameters $\rabstract$ and $\delta_0$ can be adjusted as needed (provided that \eqref{eq:U_lower_bound} is maintained). For simplicity, we choose $C=1.4$.

\begin{figure}[!htbp]
\centering
\includegraphics[scale=0.9]{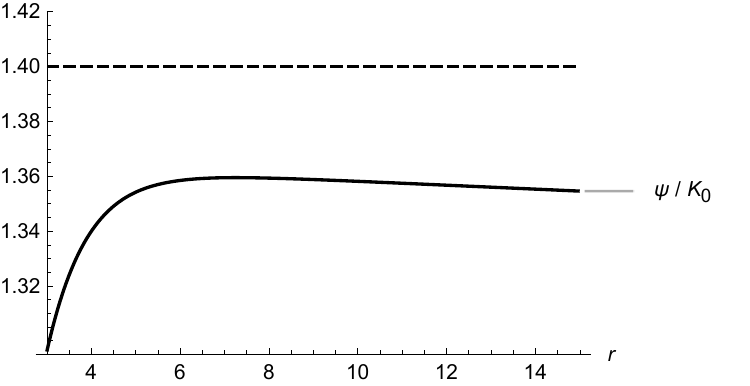}
\caption{Graph of $\psi_0(r)/K_0(\kappa_\lambda r)$ with $\delta_0=0.0005$.}\label{fig6_cst_psi}
\end{figure}

The comparison above controls the value of the internal mode on the entire
tail. We next derive the corresponding derivative estimate needed in the FGR
calculation. This is a separate consequence of the eigenvalue equation.

\begin{lemma}\label{lem:capd_psi_prime_tail}
Let $\psi_0$ and $\lambda^2$ be as in the previous lemma, and set
\[
  \kappa_\lambda=\sqrt{1-\delta_0-\lambda^2},
  \qquad \delta_0=0.0005.
\]
Then, 
\[
  0<\psi_0(r)\le 1.4K_0(\kappa_\lambda r),
  \qquad
  |\psi_0'(r)|\le 1.4K_0(\kappa_\lambda r)
\]
for all $r\ge16$.
\end{lemma}

\begin{proof}
The pointwise comparison is Lemma~\ref{lem:psi_K0}, with the certified
endpoint check made at $\reight=8.001$. For $\delta_0=0.0005$, the interval
certificate proves
\[
 U(\reight )^2-(1-\delta_0)>3.8992\cdot10^{-7}.
\]
Since \(U\) is increasing, the third inequality in
\eqref{eq:psi_K0_endpoint_margins} implies
\[
 U(s)^2-\lambda^2\ge 1-\delta_0-\lambda^2>0,
 \qquad s\ge \reight.
\]
Thus \(0<U(s)^2-\lambda^2\le1-\lambda^2\) for \(s\ge16\).
Multiplying the eigenvalue equation by $r$ gives
\[
  (r\psi_0'(r))'=r\bigl(U(r)^2-\lambda^2\bigr)\psi_0(r).
\]
The right-hand side is integrable on $[16,\infty)$ by
Lemma~\ref{lem:psi_K0}, since $0\le U\le1$. Hence $r\psi_0'(r)$ has a finite
limit as $r\to\infty$. That limit must be zero: otherwise $\psi_0'(r)$ would
have a fixed sign and magnitude comparable to $r^{-1}$ for all sufficiently
large $r$, contradicting $\psi_0(r)\to0$. We may therefore integrate from
$r$ to infinity and obtain
\[
  -r\psi_0'(r)
  =
  \int_r^\infty s\bigl(U(s)^2-\lambda^2\bigr)\psi_0(s)\,ds,
  \qquad r\ge16.
\]
Using $0\le U^2-\lambda^2\le1-\lambda^2$ and
Lemma~\ref{lem:psi_K0}, we get
\[
  |\psi_0'(r)|
  \le
  1.4\,\frac{1-\lambda^2}{r}
  \int_r^\infty sK_0(\kappa_\lambda s)\,ds
  =
  1.4\,\frac{1-\lambda^2}{\kappa_\lambda}
  K_1(\kappa_\lambda r),
\]
where we used \((xK_1(x))'=-xK_0(x)\).
For $x>0$ we use the sharp ratio bound
\[
  \frac{K_1(x)}{K_0(x)}<1+\frac{1}{2x},
\]
proved by Yang and Chu~\cite[Corollary~3.3]{YangChu2017}. The C++ code
verifies uniformly for $\mu\in\Lambda_{\rm FGR}$, with
$\kappa_\mu:=\sqrt{1-\delta_0-\mu}$, that
\[
  \frac{1-\mu}{\kappa_\mu }
  \left(1+\frac1{2\kappa_\mu \cdot16}\right)
  \le 0.50359<1.
\]
In particular, this holds at $\mu=\lambda^2$ and
$\kappa_\mu =\kappa_\lambda$. Since $r\ge16$, it follows that
\[
  \frac{1-\lambda^2}{\kappa_\lambda}
  \frac{K_1(\kappa_\lambda r)}{K_0(\kappa_\lambda r)}<1,
\]
which proves the derivative comparison.
\end{proof}

\section{Certification of the Fermi Golden Rule}\label{sec_NumVer_FGR}

We now turn to the FGR computations. 
With the two functions from~\eqref{eq:Phi12} in hand, we can then pull them back and define the matrix Weyl solutions at zero $\upi_1$ and $\upi_2$, namely, \begin{equation}\label{FGR_num_app_upi}\begin{aligned}
\upi_1&:=(\upi_{1,1},\upi_{1,2})^T=r^{1/2}\mathcal{B}^*\big( r^{-1/2}\Phi_1(r,\freq),\, 0\big)^T=\big(-\partial_r\Phi_1-\tfrac{1}{2r}\Phi_1-b\Phi_1,\, U\Phi_1\big)^T,
\\ \upi_2&:=(\upi_{2,1},\upi_{2,2})^T=r^{1/2}\mathcal{B}^*\big(0, \, r^{-1/2}\Phi_2(r,\freq)\big)^T=\big(U\Phi_2,\, -\partial_r\Phi_2+\tfrac{1}{2r}\Phi_2\big)^T.
\end{aligned}\end{equation}
Here $\freq=k_\lambda:=\sqrt{4\lambda^2-1}$. 
The next step is to calculate the connection coefficients $a_j(\freq)=W(\Phi_j,\overline{\Psi_j})/W(\Psi_j,\overline{\Psi_j})$. For this we need the scalar Weyl
solutions at infinity at the same frequency. The C++ code uses certified outgoing\footnote{This refers to the sign in the complex exponential, which ensures Sommerfeld's radiation condition in the plane.}  initial conditions at $\rsixteen =16$.
These are obtained from interval enclosures of the outgoing asymptotic
expansions, including residual bounds and Volterra remainder estimates for the truncation errors:
\begin{equation}
    \label{eq:outgoingW}
    \begin{aligned}
\Psi_1(\rsixteen )&=\dfrac{1}{\sqrt{k_\lambda}}e^{ik_\lambda \rsixteen }\Big(1+\dfrac{3i}{8}\dfrac{1}{k_\lambda \rsixteen }+\dfrac{15}{128}\dfrac{1}{(k_\lambda \rsixteen )^2}+\ldots\Big), 
\\  \partial_r\Psi_1(\rsixteen )&=\dfrac{1}{\sqrt{k_\lambda}} e^{ik_\lambda \rsixteen }\Big(ik_\lambda-\dfrac{3}{8\rsixteen }-\dfrac{33i}{128}\dfrac{1}{k_\lambda \rsixteen ^2}+\ldots\Big),
\\ 
\Psi_2(\rsixteen )&=\dfrac{1}{\sqrt{k_\lambda}}e^{ik_\lambda \rsixteen }\Big(1-\dfrac{i}{8}\dfrac{1}{k_\lambda \rsixteen } -\dfrac{9}{128}\dfrac{1}{(k_\lambda \rsixteen )^2}+\ldots\Big), 
\\ \partial_r\Psi_2(\rsixteen )&=\dfrac{1}{\sqrt{k_\lambda}} e^{ik_\lambda \rsixteen }\Big(ik_\lambda+\dfrac{1}{8\rsixteen }+\dfrac{7i}{128}\dfrac{1}{k_\lambda \rsixteen ^2}+\ldots\Big),
\end{aligned}
\end{equation}
and then propagates backward on the compact interval. The expansion is differentiated term by term, and the truncation error at $\rsixteen $ is enclosed before it is used as initial data. Within the compact interval the values of $\Psi_j$ are dictated by the linear equation, not by these truncated expressions.

We then compute $a_j(\freq)$. The proof uses the interval enclosures from the C++ code.
For the calculations below, we also introduce the matrix Weyl solutions at infinity, defined analogously to~\eqref{FGR_num_app_upi} (see Proposition~\ref{prop:p1_t1_and_t2} and Proposition~\ref{propdFT}):
\begin{equation}\label{FGR_num_app_upsi_2}\begin{aligned}
    \upsi_1&:=(\upsi_{1,1},\upsi_{1,2})^T=r^{1/2}\mathcal{B}^*\big( r^{-1/2}\Psi_1(r,\freq),\, 0\big)^T=\big(-\Psi_1'-\tfrac{1}{2r}\Psi_1-b\Psi_1,\, U\Psi_1\big)^T,
    \\ \upsi_2&:=(\upsi_{2,1},\upsi_{2,2})^T=r^{1/2}\mathcal{B}^*\big(0, \, r^{-1/2}\Psi_2(r,\freq)\big)^T=\big(U\Psi_2,\, -\Psi_2'+\tfrac{1}{2r}\Psi_2\big)^T.
\end{aligned}\end{equation}
With all of these, we can calculate the distorted Fourier kernel
$\mathbf{E}$ from Proposition~\ref{propdFT}, and hence the Fermi Golden
Rule coefficients in \eqref{FGRODE14}.  The calculation uses the bilinear
form $\bmQ_s$ from \eqref{eq:bmQsdef}, the eigenvectors \eqref{eq:bmYjbd1},
the interval enclosure of the vortex profile, the internal mode, and the distorted
Fourier kernel at $\sqrt{\bfM}=2\lambda$, equivalently at
$k_\lambda=\sqrt{4\lambda^2-1}$ in the scalar Fourier variable.

We now turn to the second, independent tail input for the FGR calculation:
the outgoing scalar Weyl solutions. All interval estimates involving the
spectral parameter are verified uniformly for
$\mu\in\Lambda_{\rm FGR}$  and then applied at $\mu=\lambda^2$. For the
solutions in~\eqref{eq:outgoingW}, the following lemma gives certified initial data
at $R=16$ and a uniform tail bound. These intervals are then propagated backward over the finite interval $[10,16]$.
The C++ code certifies both the initial intervals and the Volterra remainder
bounds.

\begin{lemma}\label{lem:capd_outgoing_start}
Let $\mu\in\Lambda_{\rm FGR}$ and put $k=\sqrt{4\mu-1}$.  For $j=1,2$ let $y_j$ solve 
\[
  y_j''+Q_j(r,\mu)y_j=0,\qquad k^{1/2}e^{-ikr}y_j(r)\longrightarrow1
\qquad(r\to\infty), 
\]
which are equivalent to $\mathcal H_jy_j=4\mu y_j$, with
\[
\begin{aligned}
Q_1(r,\mu)&=4\mu-\frac{1+U(r)^2}{2}-\frac{(2-a(r))^2}{r^2}+\frac{1}{4r^2},\\
Q_2(r,\mu)&=4\mu-U(r)^2+\frac{1}{4r^2}.
\end{aligned}
\]
Define the order $N$ truncated approximation of $y_j$ by
\[
  y_{j,N}^0(r)=k^{-1/2}e^{ikr}\sum_{n=0}^{N}c_{j,n}(kr)^{-n},\qquad N=10,
\]
where $c_{j,0}=1$ and
\[
  c_{j,n}=-i\,\frac{n(n-1)+\alpha_j}{2n}\,c_{j,n-1},
  \qquad
  \alpha_1=-\frac34,\quad \alpha_2=\frac14 .
\]
Then, for  $\mu\in\Lambda_{\rm FGR}$, one has at $R=16$:
\[
\begin{aligned}
  y_1(R)&\in y_{1,N}^0(R)+[-\varepsilon_{1,0},\varepsilon_{1,0}]+i[-\varepsilon_{1,0},\varepsilon_{1,0}],\\
  y_1'(R)&\in (y_{1,N}^0)'(R)+[-\varepsilon_{1,1},\varepsilon_{1,1}]+i[-\varepsilon_{1,1},\varepsilon_{1,1}],\\
  y_2(R)&\in y_{2,N}^0(R)+[-\varepsilon_{2,0},\varepsilon_{2,0}]+i[-\varepsilon_{2,0},\varepsilon_{2,0}],\\
  y_2'(R)&\in (y_{2,N}^0)'(R)+[-\varepsilon_{2,1},\varepsilon_{2,1}]+i[-\varepsilon_{2,1},\varepsilon_{2,1}],
\end{aligned}
\]
with
\[
\varepsilon_{1,0}=1.350273\cdot10^{-7},\ \   \varepsilon_{1,1}=3.922666\cdot10^{-7},\ \ 
\varepsilon_{2,0}=1.306810\cdot10^{-7},\ \  \varepsilon_{2,1}=3.796404\cdot10^{-7}.
\]
Moreover,  the  Weyl vectors $\upsi_j$ in
\eqref{FGR_num_app_upsi_2} satisfy
\[
  \sup_{r\ge16}\max_j |\upsi_j(r)|\le1.253,
\]
in the maximum norm for vectors in $\bbC^2$. 
\end{lemma}

\begin{proof}
Let
\[
  Q_1^0(r,\mu)=k^2-\frac{3}{4r^2},\qquad
  Q_2^0(r,\mu)=k^2+\frac{1}{4r^2}.
\]
For $r\ge 16$, Corollary~\ref{cor:vortex_tail_bridge} gives the explicit   bound
\[
  0\le 1-U(r)^2\le 8e^{-r}r^{-1/2}.
\]
For $j=2$ this implies
\[
  |Q_2(r,\mu)-Q_2^0(r,\mu)|=|1-U(r)^2|\le 8e^{-r}r^{-1/2}.
\]
For $j=1$ we also need a bound on $1-a(r)$. Using the ODE $a'(r)=\frac r2(1-U(r)^2)$ and $a(\infty)=1$, we have
\[
  1-a(r)=\int_r^\infty a'(s)\,ds
  =\frac12\int_r^\infty s\bigl(1-U(s)^2\bigr)\,ds
  \le 4\int_r^\infty e^{-s}s^{1/2}\,ds.
\]
From  $\int_r^\infty e^{-s}s^{1/2}\,ds\le e^{-r}r^{1/2}\bigl(1+\frac{1}{2r}\bigr)$ we infer, for $r\ge16$, that
\[
  0\le 1-a(r)\le 4e^{-r}r^{1/2}\Bigl(1+\frac{1}{2r}\Bigr)
  \le \frac{33}{8}\,e^{-r}r^{1/2}.
\]
Consequently,
\[
\begin{aligned}
|Q_1(r,\mu)-Q_1^0(r,\mu)|
&=\left|\frac{1-U(r)^2}{2}+\frac{1-(2-a(r))^2}{r^2}\right|\\
&\le 4e^{-r}r^{-1/2}+\frac{3(1-a(r))}{r^2}
\le 4e^{-r}r^{-1/2}+\frac{99}{8}e^{-r}r^{-3/2}\\
&\le 4e^{-r}r^{-1/2}+\frac{99}{128}e^{-r}r^{-1/2}
\le 8e^{-r}r^{-1/2}.
\end{aligned}
\]
Write
\[
  y(r)=k^{-1/2}e^{ikr}z(r).
\]
Then $z$ solves
\[
  z''+2ikz'+\left(\frac{\alpha_j}{r^2}+V_j(r)\right)z=0,
  \qquad
  V_j=Q_j-Q_j^0.
\]
Note that $V_j$ here is different from \eqref{eq:lemSd2}. The polynomial
\[
  z_{j,N}(r)=\sum_{n=0}^{N}c_{j,n}(kr)^{-n}
\]
satisfies the equation with $V_j\equiv 0$ up to a remainder term:
\[
  z_{j,N}''+2ikz_{j,N}'+\frac{\alpha_j}{r^2}z_{j,N}
  =
  k^2\big(N(N+1)+\alpha_j\big)c_{j,N}(kr)^{-N-2},
\]
as can be seen from the recurrence for $c_{j,n}$. Therefore, with $R=16$,
\[
  I_{{\rm res},j}
  :=
  \int_R^\infty
  \left|
  k^2\big(N(N+1)+\alpha_j\big)c_{j,N}(ks)^{-N-2}
  \right|\,ds
  \le
  \frac{|N(N+1)+\alpha_j|\,|c_{j,N}|\,k}
       {(N+1)(kR)^{N+1}} .
\]
The C++ code evaluates this last rational expression with
outward rounding and obtains
\[
  I_{{\rm res},1}<1.642\cdot10^{-12},\qquad
  I_{{\rm res},2}<1.500\cdot10^{-12}.
\]
Lemma~\ref{lem_1-U2} gives
\[
  I_V:=\int_R^\infty |V_j(s)|\,ds
  \le 8e^{-R}R^{-1/2}<2.251\cdot10^{-7}.
\]
Let $e=z-z_{j,N}$. The  Green kernel of $d^2/dr^2+2ik\,d/dr$ is
\[
  K(r,s)=\frac{1-e^{2ik(s-r)}}{2ik},\qquad s\ge r,
\]
so $|K(r,s)|\le k^{-1}$ and $|\partial_rK(r,s)|\le1$. Hence $e$ satisfies the Volterra equation
\begin{equation}\label{eq:outgoing_error_volterra}
e(r)=\int_r^\infty K(r,s)
\left[
\rho_{j,N}(s)+V_j(s)z_{j,N}(s)
+\left(\frac{\alpha_j}{s^2}+V_j(s)\right)e(s)
\right]\,ds,
\end{equation}
where $\rho_{j,N}$ is the explicit residual above.
To make the resulting estimates uniform for
$\mu\in\Lambda_{\rm FGR}$, set
\begin{equation}\label{eq:outgoing_k_endpoints}
k_-:=\inf_{\mu\in\Lambda_{\rm FGR}}\sqrt{4\mu-1},
\qquad
k_+:=\sup_{\mu\in\Lambda_{\rm FGR}}\sqrt{4\mu-1}.
\end{equation}
The C++ verifier uses the full interval enclosure of
$k=\sqrt{4\mu-1}$ determined by $\mu\in\Lambda_{\rm FGR}$. Define
\begin{equation}\label{eq:outgoing_zN_majorants}
Z_{j,N}:=\sum_{n=0}^{N}
\frac{|c_{j,n}|}{(k_-R)^n},
\qquad
Z'_{j,N}:=\frac1R\sum_{n=1}^{N}
\frac{n|c_{j,n}|}{(k_-R)^n}.
\end{equation}
The definition of $z_{j,N}$ and
\eqref{eq:outgoing_zN_majorants} give
\[
\|z_{j,N}\|_{L^\infty([R,\infty))}
\le Z_{j,N},
\qquad
\|z_{j,N}'\|_{L^\infty([R,\infty))}
\le Z'_{j,N}.
\]
The C++ verifier evaluates \eqref{eq:outgoing_zN_majorants} with
outward-rounded interval arithmetic. The certificate transcript in the \texttt{capd\_cpp/transcripts} directory
of~\cite{AYMHPartICertificateRepo} records
\[
Z_{1,N}<1.016362,\qquad Z_{2,N}<1.005515,
\]
and
\[
Z'_{1,N}<1.038\cdot10^{-3},\qquad
Z'_{2,N}<3.537\cdot10^{-4}.
\]
Using $|K(r,s)|\le k^{-1}$ in
\eqref{eq:outgoing_error_volterra}, the $L^\infty$ norm of its
Volterra operator is bounded by
\begin{equation}\label{eq:outgoing_contraction_constant}
\Theta_j:=
\frac{|\alpha_j|/R+I_V}{k_-}.
\end{equation}
Outward-rounded interval evaluation over
$\mu\in\Lambda_{\rm FGR}$ gives
\[
\Theta_1<3.228\cdot10^{-2},
\qquad
\Theta_2<1.076\cdot10^{-2},
\]
so both Volterra operators are contractions. Define
\begin{align}
\mathcal E_{j,0}
&:=
\frac{k_-^{-1}\bigl(I_{{\rm res},j}+I_VZ_{j,N}\bigr)}
{1-\Theta_j},
\label{eq:outgoing_e0_majorant}\\
\mathcal E_{j,1}
&:=
I_{{\rm res},j}+I_VZ_{j,N}
+\left(\frac{|\alpha_j|}{R}+I_V\right)\mathcal E_{j,0}.
\label{eq:outgoing_e1_majorant}
\end{align}
The Volterra equation \eqref{eq:outgoing_error_volterra} and its
differentiated form, using $|\partial_rK(r,s)|\le1$, give
\begin{equation}\label{eq:outgoing_e_bounds}
\|e\|_{L^\infty([R,\infty))}
\le\mathcal E_{j,0},
\qquad
\|e'\|_{L^\infty([R,\infty))}
\le\mathcal E_{j,1}.
\end{equation}
Since
\[
y_j-y_{j,N}^0=k^{-1/2}e^{ikr}e,
\]
\eqref{eq:outgoing_k_endpoints} and
\eqref{eq:outgoing_e_bounds} imply
\begin{equation}\label{eq:outgoing_y_error_bounds}
\begin{aligned}
|y_j(R)-y_{j,N}^0(R)|
&\le k_-^{-1/2}\mathcal E_{j,0},\\
|y_j'(R)-(y_{j,N}^0)'(R)|
&\le
k_+^{1/2}\mathcal E_{j,0}
+k_-^{-1/2}\mathcal E_{j,1}.
\end{aligned}
\end{equation}
The same outward-rounded computation gives the constants
$\varepsilon_{j,0}$ and $\varepsilon_{j,1}$ in the statement; the
intermediate enclosures for $\|e\|_\infty$ and $\|e'\|_\infty$ are
recorded in the same transcript.

For the uniform vector bound, define
\begin{equation}\label{eq:outgoing_Y_majorants}
\begin{aligned}
Y_{j,0}
&:=
k_-^{-1/2}\bigl(Z_{j,N}+\mathcal E_{j,0}\bigr),\\
Y_{j,1}
&:=
k_+^{1/2}\bigl(Z_{j,N}+\mathcal E_{j,0}\bigr)
+k_-^{-1/2}\bigl(Z'_{j,N}+\mathcal E_{j,1}\bigr).
\end{aligned}
\end{equation}
Then
\[
|y_j(r)|\le Y_{j,0},
\qquad
|y_j'(r)|\le Y_{j,1},
\qquad r\ge R.
\]
Moreover, the $a$-tail estimate in
Lemma~\ref{lem_1-U2} gives
\begin{equation}\label{eq:outgoing_b_tail}
|b(r)|=\frac{1-a(r)}r
\le\frac{33}{8}e^{-r}r^{-1/2},
\qquad r\ge16.
\end{equation}
Using \eqref{FGR_num_app_upsi_2}, $|U|\le1$,
\eqref{eq:outgoing_Y_majorants}, and
\eqref{eq:outgoing_b_tail}, for $r\ge R=16$ and $j=1,2$ we obtain
\begin{equation}\label{eq:outgoing_vector_bound}
|\upsi_j(r)|
\le
\max\left\{
Y_{j,1}
+\left(\frac1{2R}+\frac{33e^{-R}}{8\sqrt R}\right)Y_{j,0},
Y_{j,0}
\right\}
<1.253.
\end{equation}
This is the bound used in the FGR estimates below.
\end{proof}

We now have the two tail inputs needed for the FGR estimate.
Lemma~\ref{lem:capd_psi_prime_tail} controls the internal mode and its
derivative, while Lemma~\ref{lem:capd_outgoing_start} controls the outgoing
Weyl vectors. We combine these bounds in the following lemma.

\begin{lemma}\label{lem:capd_fgr_tail_upper_bound}
Let $\lambda^2\in\Lambda_{\rm FGR}$ be the internal eigenvalue, let
$\psi_0$ and $\bmY_j^{(0)}$ be as in \eqref{eq:bmYjbd1}, and set
\[
 k_\lambda:=\sqrt{4\lambda^2-1},
 \qquad
 \kappa_\lambda:=\sqrt{1-\delta_0-\lambda^2},
 \qquad
 \delta_0=0.0005.
\]
Then, for $r\ge16$ and $i,j\in\{1,2\}$,
\begin{equation}\label{eq:fgr_tail_pointwise_bound}
\left\|
\overline{\bfE(r,k_\lambda)}^{\,t}
\bmQ_s\bigl(\bmY_i^{(0)},\bmY_j^{(0)}\bigr)(r)
\right\|_{\bbC^4}
\le
30\,r^{-1/2}
\bigl(1.4K_0(\kappa_\lambda r)\bigr)^2.
\end{equation}
cf.~\eqref{eq:bmQsdef}.
\end{lemma}

\begin{proof}
We prove \eqref{eq:fgr_tail_pointwise_bound} by combining the
internal-mode bounds in Lemma~\ref{lem:capd_psi_prime_tail} with the
bound in Lemma~\ref{lem:capd_outgoing_start} for the vectors $\upsi_j$
constructed from the solutions normalized at infinity. The interval
computations underlying these two inputs are performed with the spectral
parameter $\mu$ ranging over $\Lambda_{\rm FGR}$. Since
$\lambda^2\in\Lambda_{\rm FGR}$, both estimates apply at
$\mu=\lambda^2$.

For brevity, write
\[
 A_\lambda(r):=1.4K_0(\kappa_\lambda r).
\]
We first compute the bilinear source exactly. To avoid confusing the fourth
component of the perturbation with a spectral parameter, write
\[
 \bmr=(\alpha,\zeta,\beta,m)^t,
\]
where $m$ denotes the variable denoted by $\mu$ in
\eqref{eq:FGR_cubic_hamiltonian}. Taking the variational derivative in
\eqref{eq:bmQbddef1} gives
\[
\bmQ(\alpha,\zeta,\beta,m)=
\begin{pmatrix}
-\frac32U\alpha^2-\frac12U\beta^2
+\partial^*(m\beta)-m\beta'
+2b\zeta\alpha-U\zeta^2-Um^2\\
b(\alpha^2+\beta^2)-2U\zeta\alpha\\
-U\alpha\beta-\partial^*(m\alpha)+m\alpha'
+2b\zeta\beta\\
\beta\alpha'-\alpha\beta'-2U\alpha m
\end{pmatrix},
\qquad b=\frac{U'}{U}.
\]
Substituting
\[
 \bmY_1^{(0)}=(U\psi_0,-\psi_0',0,0)^t,
 \qquad
 \bmY_2^{(0)}=(0,0,U\psi_0,-\psi_0')^t
\]
and polarizing according to \eqref{eq:bmQsdef} yields the following exact
table:
\[
{\renewcommand{\arraystretch}{1.3}%
\begin{array}{c|c|c}
 & \mathcal Q_{ij}^{(1)} & \mathcal Q_{ij}^{(2)} \\ \hline
(1,1) &
-\frac32U^3\psi_0^2-2U'\psi_0\psi_0'-U(\psi_0')^2 &
UU'\psi_0^2+2U^2\psi_0\psi_0'
\\
(2,2) &
\frac12U^3\psi_0^2+2U'\psi_0\psi_0'
+U(\psi_0')^2-\lambda^2U\psi_0^2 &
UU'\psi_0^2
\\
(1,2) &
-U^3\psi_0^2-2U'\psi_0\psi_0'
-U(\psi_0')^2+\frac12\lambda^2U\psi_0^2 &
U^2\psi_0\psi_0'
\end{array}}
\]
More precisely,
\begin{equation}\label{eq:Qijscalar}
\begin{aligned}
 \bmQ_s(\bmY_1^{(0)},\bmY_1^{(0)})
 &=(\mathcal Q_{11}^{(1)},\mathcal Q_{11}^{(2)},0,0)^t,\\
 \bmQ_s(\bmY_2^{(0)},\bmY_2^{(0)})
 &=(\mathcal Q_{22}^{(1)},\mathcal Q_{22}^{(2)},0,0)^t,\\
 \bmQ_s(\bmY_1^{(0)},\bmY_2^{(0)})
 &=(0,0,\mathcal Q_{12}^{(1)},\mathcal Q_{12}^{(2)})^t,
\end{aligned}
\end{equation}
and the case $(2,1)$ is the same by symmetry. The terms containing
$\lambda^2$ result from the exact internal-mode equation
\[
 \psi_0''+\frac1r\psi_0'=(U^2-\lambda^2)\psi_0.
\]
For $r\ge16$, Proposition~\ref{prop:exuniqUa} gives
\[
 0<U\le1,\qquad 0\le U'\le1,\qquad 0<\lambda^2<1,
\]
while Lemma~\ref{lem:capd_psi_prime_tail} gives
\[
 |\psi_0(r)|\le A_\lambda(r),
 \qquad
 |\psi_0'(r)|\le A_\lambda(r).
\]
It follows directly from the table that
\[
 |\mathcal Q_{ij}^{(\ell)}(r)|
 \le5A_\lambda(r)^2,
 \qquad \ell=1,2.
\]
 Next, we relate the columns $\upi_j$ determined by the origin normalizations to the outgoing
columns. By \eqref{def:a},
\[
 \Phi_j(r,k_\lambda)
 =
 a_j(k_\lambda)\Psi_j(r,k_\lambda)
 +\overline{a_j(k_\lambda)}\,
  \overline{\Psi_j(r,k_\lambda)}.
\]
Since the first-order operators defining $\upi_j$ and $\upsi_j$ have real
coefficients,
\[
 a_j(k_\lambda)^{-1}\upi_j
 =
 \upsi_j+
 \frac{\overline{a_j(k_\lambda)}}{a_j(k_\lambda)}
 \overline{\upsi_j}.
\]
Lemma~\ref{lem:capd_outgoing_start} therefore gives, componentwise,
\[
 \left|a_j(k_\lambda)^{-1}\upi_{j,\ell}(r,k_\lambda)\right|
 \le2\mathcal W_\infty,
 \qquad
 \mathcal W_\infty
 :=\sup_{r\ge16}\max_{j,\ell}|\upsi_{j,\ell}(r,k_\lambda)|
 \le1.253.
\]
Let $(q_1,q_2)$ denote the two nonzero components of one of the source vectors
in \eqref{eq:Qijscalar}. Since $\langle k_\lambda\rangle=2\lambda$, the corresponding
nonzero block of the contraction is exactly
\[
 -\frac{i}{2\lambda\sqrt{2\pi r}}
 \begin{pmatrix}
 \overline{a_1(k_\lambda)^{-1}\upi_{1,1}}\,q_1
 +\overline{a_1(k_\lambda)^{-1}\upi_{1,2}}\,q_2\\
 \overline{a_2(k_\lambda)^{-1}\upi_{2,1}}\,q_1
 +\overline{a_2(k_\lambda)^{-1}\upi_{2,2}}\,q_2
 \end{pmatrix}.
\]
Thus, each output component contains exactly two matrix products. After expanding this connection formula, each of these products splits
into an outgoing and an incoming term (the latter being the complex conjugate). 

Using the componentwise bounds above and then the Euclidean norm, we obtain
\[
\begin{aligned}
 \left\|
 \overline{\bfE(r,k_\lambda)}^{\,t}
 \bmQ_s\bigl(\bmY_i^{(0)},\bmY_j^{(0)}\bigr)(r)
 \right\|_{\bbC^4}
 &\le
 \frac{20\sqrt2\,\mathcal W_\infty}
      {2\lambda\sqrt{2\pi}}\,
 r^{-1/2}A_\lambda(r)^2
 \\ &\le
 \frac{10(1.253)}
      {\sqrt{0.777471875}\sqrt{\pi}}\,
 r^{-1/2}A_\lambda(r)^2\\
 &<8.02\,r^{-1/2}A_\lambda(r)^2
 <30\,r^{-1/2}A_\lambda(r)^2.
\end{aligned}
\]
Finally, the factor $r^{-1/2}$ is essential. After multiplication by the radial
measure $r\,dr$, it produces the factor $\sqrt r\,dr$ in
\eqref{FGR_rg15_numerics}.
\end{proof}

We now integrate the pointwise estimate
\eqref{eq:fgr_tail_pointwise_bound} to control the contribution of
$r\ge16$ to the FGR coefficients. For
$\mu\in\Lambda_{\rm FGR}$, set
\[
 k_\mu:=\sqrt{4\mu-1},
 \qquad
 \kappa_\mu:=\sqrt{1-\delta_0-\mu},
\]
as in Lemma~\ref{lem:capd_psi_prime_tail}, and define the rounded-up tail bound used below by
\begin{equation}\label{eq:Btail_definition}
 B_{\rm tail}:=3.174707028\cdot10^{-10}.
\end{equation}
The Bessel estimate \eqref{eq:K0_bound} and
$r^{-1/2}\le1/4$ for $r\ge16$ imply, for every $\kappa>0$,
\begin{align}\label{eq:Btail_Bessel_majorant}
 \int_{16}^\infty
 \bigl[1.4K_0(\kappa r)\bigr]^2\sqrt r\,dr
 &\le
 \frac{(1.4)^2\pi}{2\kappa}
 \int_{16}^\infty r^{-1/2}e^{-2\kappa r}\,dr  \le
 \frac{(1.4)^2\pi}{16\kappa^2}e^{-32\kappa}.
\end{align}
The C++ code evaluates the majorant obtained by substituting
\eqref{eq:Btail_Bessel_majorant}, with $\mu$ ranging over
$\Lambda_{\rm FGR}$. Outward-rounded interval arithmetic gives
\[
 \sup_{\mu\in\Lambda_{\rm FGR}}
 \frac{\pi k_\mu}{4\mu}
 \left(
 30\int_{16}^\infty
 \bigl[1.4K_0(\kappa_\mu r)\bigr]^2\sqrt r\,dr
 \right)^2
 <B_{\rm tail}.
\]
Consequently, by Proposition~\ref{propdFT}, 
Remark~\ref{rem:fgr_code_normalization}, and \eqref{eq:fgr_tail_pointwise_bound}, for $i,j\in\{1,2\}$,
\begin{align}\label{FGR_rg15_numerics}
&\frac{\pi k_\lambda}{8\lambda^2}
\left\|
\int_{16}^\infty
\overline{\bfE(r,k_\lambda)}^{\,t}
\bmQ_s(\bmY_i^{(0)},\bmY_j^{(0)})(r)\,r\,dr
\right\|_{\bbC^4}^2
\notag\\
&\quad\le
\frac{\pi k_\lambda}{4\lambda^2}
\left(
30\int_{16}^\infty
\bigl[1.4K_0(\kappa_\lambda r)\bigr]^2\sqrt r\,dr
\right)^2
<B_{\rm tail}<10^{-9}.
\end{align}
The radial integrals entering the FGR coefficients are evaluated by interval
quadrature on $[0.1,16]$, thereby avoiding the regular singular point at the
origin. The integral over $0<r<0.1$ is estimated directly from the Frobenius
series of the regular solutions. The three vector
contributions over $(0,0.1)$, $[0.1,16]$, and $(16,\infty)$ are then
combined by the  triangle inequality before taking the square.
The next lemma records the corresponding bound for the origin
contribution, uniformly for
$\lambda^2\in\Lambda_{\rm FGR}$.

\begin{lemma}\label{lem:capd_fgr_origin}
Let $\rorigin=0.1$, let $\lambda^2\in\Lambda_{\rm FGR}$ be the internal
eigenvalue, and set
\[
  k_\lambda:=\sqrt{4\lambda^2-1}.
\]
For $(i,j)\in\{(1,1),(2,2),(1,2)\}$, write
\[
  \bm q_{ij}
  :=
  \begin{pmatrix}
    \mathcal Q_{ij}^{(1)}\\
    \mathcal Q_{ij}^{(2)}
  \end{pmatrix},
\]
where the two entries are given in the table immediately preceding
\eqref{eq:Qijscalar}, and set
$\bm q_{21}:=\bm q_{12}$. For $\ell=1,2$, define
\[
  \mathcal A_{ij,\ell}^{\rm org}
  :=
  \int_0^{\rorigin}
  \sqrt r\,
  \overline{
    i\,a_\ell(k_\lambda)^{-1}
    \upi_\ell(r,k_\lambda)
  }^{\,t}
  \bm q_{ij}(r)\,dr,
\]
and put
\[
  \mathcal A_{ij}^{\rm org}
  :=
  \begin{pmatrix}
    \mathcal A_{ij,1}^{\rm org}\\
    \mathcal A_{ij,2}^{\rm org}
  \end{pmatrix}.
\]
Then
\[
  \max_{\ell=1,2}
  \bigl|\mathcal A_{ij,\ell}^{\rm org}\bigr|
  <8.43\cdot10^{-4}.
\]
For $\mu\in\Lambda_{\rm FGR}$, set
\[
 c(\mu):=\frac{\sqrt{4\mu-1}}{64\mu^2},
 \qquad
 B_{\rm origin}:=1.5\cdot10^{-6}.
\]
Then\footnote{All bounds here are uniform in $\mu\in \Lambda_{\rm FGR}$.}
\begin{equation}\label{eq:B_origin}
 c(\lambda^2)
 \bigl\|\mathcal A_{ij}^{\rm org}\bigr\|_{\bbC^2}^2
 <B_{\rm origin}.
\end{equation}
\end{lemma}

\begin{proof}
We first control $\psi_0$. It satisfies
\[
 (r\psi_0')'=r(U^2-\lambda^2)\psi_0,
 \qquad \psi_0(0)=1,\qquad \psi_0'(0)=0.
\]
Since $|U^2-\lambda^2|\le1$, if
$m(r):=\sup_{0\le s\le r}|\psi_0(s)|$, then
\[
 |\psi_0'(r)|\le\frac r2m(r),
 \qquad
 m(r)\le1+\frac{r^2}{4}m(r).
\]
Thus, on $[0,\rorigin]$,
\[
 |\psi_0|\le M,\qquad |\psi_0'|\le P,
 \qquad
 M:=\left(1-\frac{\rorigin^2}{4}\right)^{-1},
 \qquad P:=\frac{M\rorigin}{2}.
\]
We next estimate $\upi_1$ and $\upi_2$ on $(0,\rorigin]$. The Frobenius
normalizations in Lemma~\ref{lem:Acheck} allow us to write, at energy
$E=4\lambda^2$,
\[
 \Phi_1(r)=r^{5/2}g_1(r),\qquad
 \Phi_2(r)=r^{1/2}g_2(r),
 \qquad g_j(0)=1,\quad g_j'(0)=0.
\]
Substitution in \eqref{eq:Phi12} gives
\[
\begin{aligned}
 (r^5g_1')'
 &=r^3\left[-4a+a^2
   +r^2\left(\frac{1+U^2}{2}-E\right)\right]g_1,\\
 (rg_2')'&=r(U^2-E)g_2.
\end{aligned}
\]
On $0\le r\le \rorigin$ we have
\[
 0\le U(r)\le r,\qquad
 0\le a(r)\le\frac{r^2}{4},\qquad
 0\le U'(r)\le1,\qquad
 3<E=4\lambda^2<4.
\]
Here the bound on $U'$ follows for $r>0$ from
$U'=(1-a)U/r$, and at $r=0$ from $U'(0)=\uprz<1$. The bounds on $E$ follow from Lemma~\ref{lem:fgr_eigenvalue_box}. It follows then that
\[
 \left|-4a+a^2+r^2\left(\frac{1+U^2}{2}-E\right)\right|
 \le5r^2,
 \qquad |U^2-E|\le4.
\]
Writing $G_j(r):=\sup_{0\le s\le r}|g_j(s)|$ and integrating from the
origin, we obtain that
\[
 |g_1'(r)|\le\frac56rG_1(r),
 \qquad |g_2'(r)|\le2rG_2(r),
\]
and hence
\[
 G_1(\rorigin)\le\left(1-\frac{5\rorigin^2}{12}\right)^{-1}<1.005,
 \qquad
 G_2(\rorigin)\le(1-\rorigin^2)^{-1}<1.011.
\]
Using $b=(1-a)/r$ in \eqref{FGR_num_app_upi}, this gives
\begin{equation}\label{eq:fgr_origin_column_bound}
\begin{aligned}
 \upi_1
 &=\begin{pmatrix}
  -(4-a)r^{3/2}g_1-r^{5/2}g_1'\\
  Ur^{5/2}g_1
 \end{pmatrix},
 &
 |\upi_{1,1}|+|\upi_{1,2}|&<5r^{3/2},\\
 \upi_2
 &=\begin{pmatrix}
  Ur^{1/2}g_2\\
  -r^{1/2}g_2'
 \end{pmatrix},
 &
 |\upi_{2,1}|+|\upi_{2,2}|&<4r^{3/2}.
\end{aligned}
\end{equation}
We now apply the bounds $|U|\le \rorigin $, $|U'|\le1$,
$\lambda^2\le1$, $|\psi_0|\le M$, and $|\psi_0'|\le P$ to the explicit
formulas for the entries introduced in \eqref{eq:Qijscalar}. Their majorants are
\begin{equation}\label{eq:fgr_origin_source_majorants}
{\renewcommand{\arraystretch}{1.25}
\begin{array}{c|c|c}
 & |\mathcal Q_{ij}^{(1)}|&|\mathcal Q_{ij}^{(2)}|\\ \hline
(1,1)&\frac32\rorigin ^3M^2+2MP+\rorigin P^2
     &\rorigin M^2+2\rorigin ^2MP\\
(2,2)&\frac12\rorigin ^3M^2+2MP+\rorigin P^2+\rorigin M^2
     &\rorigin M^2\\
(1,2)&\rorigin ^3M^2+2MP+\rorigin P^2+\frac12\rorigin M^2
     &\rorigin ^2MP.
\end{array}}
\end{equation}
Substituting $P=M\rorigin /2$ into the six quantities in \eqref{eq:fgr_origin_source_majorants} gives
\[
\begin{array}{c|c|c}
 & |\mathcal Q_{ij}^{(1)}|&|\mathcal Q_{ij}^{(2)}|\\ \hline
(1,1)&M^2(\rorigin +\frac74\rorigin ^3)&M^2(\rorigin +\rorigin ^3)\\
(2,2)&M^2(2\rorigin +\frac34\rorigin ^3)&M^2\rorigin \\
(1,2)&M^2(\frac32\rorigin +\frac54\rorigin ^3)&\frac12M^2\rorigin ^3.
\end{array}
\]
Since $0<\rorigin <1$, the largest entry is the $(2,2)$ first-component majorant. Hence, the maximum of these six quantities is
\[
 q_0=M^2\left(2\rorigin +\frac34\rorigin ^3\right)
 <0.201757527.
\]
The certified Wronskian enclosures recorded in the transcript give,
uniformly over $\Lambda_{\rm FGR}$,
\[
 A_*:=\sup_{\mu\in\Lambda_{\rm FGR}}
 \max\{|a_1(k_\mu)|^{-1},|a_2(k_\mu)|^{-1}\}
 <2.506371585,
 \qquad k_\mu:=\sqrt{4\mu-1}.
\]
By the definition of $\mathcal A_{ij,\ell}^{\rm org}$,
\begin{equation}\label{eq:easyint}
    \begin{aligned}
 |\mathcal A_{ij,\ell}^{\rm org}|
 &\le q_0A_*
   \int_0^{\rorigin }
   \bigl(|\upi_{\ell,1}(r)|+|\upi_{\ell,2}(r)|\bigr)\sqrt r\,dr\\
 &\le\frac53q_0A_*\rorigin ^3
 <8.43\cdot 10^{-4}. 
\end{aligned}
\end{equation}
The column sum in \eqref{eq:fgr_origin_column_bound} accounts for both
terms in this scalar product. It follows from \eqref{eq:easyint} that
\[
 \bigl\|\mathcal A_{ij}^{\rm org}\bigr\|_{\bbC^2}^2
 =
 \sum_{\ell=1}^2
 \bigl|\mathcal A_{ij,\ell}^{\rm org}\bigr|^2
 <
 2\bigl(8.43\cdot10^{-4}\bigr)^2.
\]
Moreover, since
$\lambda^2\in\Lambda_{\rm FGR}\subset(\frac12,1)$, we have
\[
 \sqrt{4\lambda^2-1}<2\lambda^2
\]
and therefore
\[
 0<c(\lambda^2)
 =\frac{\sqrt{4\lambda^2-1}}{64\lambda^4}
 <\frac{1}{32\lambda^2}<1.
\]
Consequently,
\[
\begin{aligned}
 c(\lambda^2)
 \bigl\|\mathcal A_{ij}^{\rm org}\bigr\|_{\bbC^2}^2
 &<
 2\bigl(8.43\cdot10^{-4}\bigr)^2 =1.421298\cdot10^{-6}
 <1.5\cdot10^{-6}
 =B_{\rm origin}.
\end{aligned}
\]
This proves \eqref{eq:B_origin}.
\end{proof}

\begin{proof}[Completion of the proof of Proposition~\ref{prop:FGR1}]
For an interval $I\subset(0,\infty)$, let $\mathcal A_{ij}(I)$ be defined
as $\mathcal A_{ij}^{\rm org}$ in Lemma~\ref{lem:capd_fgr_origin}, with $(0,\rorigin )$ replaced by $I$.
The radial integral is evaluated by interval quadrature on $[0.1,16]$ and
the intervals $(0,0.1)$ and $(16,\infty)$ are treated separately. For $\alpha\in\{1,2,12\}$, let $(i,j)$ denote the corresponding pair,
namely
\[
  (i,j)=(1,1),(2,2),(1,2),
\]
respectively, and let $\widehat D_{\alpha,[0.1,16]}^{(0)}$ denote the
quantity obtained from $\widehat D_\alpha^{(0)}$ by restricting the
radial integral in the distorted Fourier transform to $[0.1,16]$. Before
accounting for the two remaining intervals, the C++ code certifies
\begin{equation}\label{numerics_fgr1}
\begin{aligned}
 \widehat D_{1,[0.1,16]}^{(0)}
 &\in[0.0176262802754,0.0332324717579],\\
 \widehat D_{2,[0.1,16]}^{(0)}
 &\in[0.0477905323094,0.0546253884722],\\
 \widehat D_{12,[0.1,16]}^{(0)}
 &\in[0.0233179905311,0.0287364386010].
\end{aligned}
\end{equation}
Denote the lower endpoints in \eqref{numerics_fgr1} by
\[
 L_1=0.0176262802754,\qquad
 L_2=0.0477905323094,\qquad
 L_{12}=0.0233179905311.
\]
Then \eqref{numerics_fgr1}, Lemma~\ref{lem:capd_fgr_origin}, and
\eqref{FGR_rg15_numerics} give
\begin{equation}\label{eq:Luse}
\begin{aligned}
  c(\lambda^2)
  \bigl\|\mathcal A_{ij}([\rorigin ,16])\bigr\|_{\bbC^2}^2
  &\ge L_\alpha,\\
  c(\lambda^2)
  \bigl\|\mathcal A_{ij}((0,\rorigin ))\bigr\|_{\bbC^2}^2
  &<B_{\rm origin},\\
  c(\lambda^2)
  \bigl\|\mathcal A_{ij}((16,\infty))\bigr\|_{\bbC^2}^2
  &\le B_{\rm tail},
\end{aligned}
\end{equation}
then the triangle inequality gives
\begin{equation}\label{eq:Dhat_triangle_ineq}
 \widehat D_\alpha^{(0)}
 \ge
 \left(
  \sqrt{L_\alpha}
  -\sqrt{B_{\rm origin}}
  -\sqrt{B_{\rm tail}}
 \right)_+^2,
 \qquad (x)_+:=\max\{x,0\}.
\end{equation}
Here $B_{\rm origin}$ is given in \eqref{eq:B_origin}, and $B_{\rm tail}$ is defined immediately
before \eqref{eq:Btail_Bessel_majorant} and estimated in \eqref{FGR_rg15_numerics}. The final certified lower
bounds for $\widehat D_\alpha^{(0)}$ are stated in \eqref{eq:Dhat_lower_bounds}.
Together with the interval bounds computed by the C++ code, this proves
Proposition~\ref{prop:FGR1}.
\end{proof}

\section{The validated numerical computations}\label{sec:capd_certificate}

This section summarizes the computational work that supports the results of this paper. Near the regular singularity
$r=0$, the Frobenius estimates proved above give rigorous intervals (or Cartesian products of intervals) for all initial data
at $\rorigin=0.1$. On bounded subintervals of $(0,\infty)$, we rely on the CAPD library and interval arithmetic
to solve the differential equations and evaluate the required integrals.
For large $r$, analytic estimates based on barriers, comparison arguments,
and Volterra equations reduce each claim to inequalities at a fixed finite radius~$r$. Thus, the conclusions on unbounded intervals follow analytically
from inequalities at a fixed $r$ (such as $r=\rfour=\rzeroUbound$, $r=16$) which are certified to hold by our C++ code.

The code uses the CAPD multiprecision interval types
\texttt{MpInterval} and \texttt{MpFloat}, with $220$ bits of precision and
Taylor order $40$. This means that the ODE solver divides an interval $[r_{\rm a},r_{\rm b}]$ on which we wish to propagate a flow into many small intervals and then returns Taylor polynomials of degree at most $40$ on each of them together with a rigorous remainder bound. The true solution then lies in the tubular neighborhood defined by the Taylor approximations together with the rigorous remainder bound.

All arithmetic operations are rounded outward, and the resulting
enclosures by intervals are valid throughout each integration step. Decimal inputs are
also rounded outward when converted to intervals, so each displayed input
value is by definition contained in the interval used by the computation. The program is fail-fast: a block depending on
earlier calculations is executed only after all of its prerequisite blocks
have passed. The C++ code and the transcripts containing all results are, respectively,
\[
\begin{gathered}
 \texttt{capd\_cpp/src/lppssi\_cert\_v6.cpp},\\
 \texttt{capd\_cpp/transcripts/capd\_certificate\_v6\_220bit\_macos.txt},\\
\texttt{capd\_cpp/transcripts/capd\_certificate\_v6\_220bit\_linux.txt}.
\end{gathered}
\]
Each transcript begins with the SHA--256 hash of the source file used in
the run and records the CAPD and compiler versions. The file
\texttt{SHA256SUMS} records the hashes of every other versioned file in the
certificate.
The aforementioned files are contained in the repository~\cite{AYMHPartICertificateRepo}, together
with a hash manifest and instructions for repeating the calculation.  The
transcripts contain the principal enclosures and a passing result for every
certificate block. Information about CAPD, which is open source,  is available at
\url{http://capd.ii.uj.edu.pl/}.

\subsection{The division between analysis and computation}

The initial intervals for $U,a$ and their derivatives with respect to the
shooting parameter follow from the coefficient bounds in
Lemmas~\ref{lem:Wberror} and~\ref{lem:Wbderror}.  For the regular solutions
$\psi_0,\Phi_1,\Phi_2$ and the regular threshold solution
$\Phi_{2,0}^{(0)}$, the coefficient recurrences are treated by
Lemmas~\ref{lem:capd_origin_frobenius} and~\ref{lem:Acheck}.  The C++ code uses degree $100$ for $\psi_0$ and degree $60$ for  $\Phi_1$ and~$\Phi_2$. It then verifies that the next $20$ coefficients satisfy the desired geometric bound with $\rho=0.95$. From that point on, 
the  majorant in
Lemma~\ref{lem:capd_origin_frobenius} bounds all remaining coefficients.
As a result, the intervals at $\rorigin=0.1$ contain the exact regular solutions
and their derivatives.

For the endpoints $\upr_-,\upr_+$ of $J_0$, the ODE solver proves the lower bounds
$a(27.75,\upr_-)>2$ and $U(31.6,\upr_+)>2$.
These are the two alternatives in the shooting argument of
Proposition~\ref{prop:exuniqUa},  and Proposition~\ref{prop:J0slope} then concludes that
$\uprz\in J_0$.  On the same interval $J_0$, we verify by interval arithmetic that the 
four inequalities in \eqref{app_num_condition_a6} at $r=\rfour$ hold.
Lemma~\ref{lem_1-U2} and Corollary~\ref{cor:vortex_tail_bridge} then prove upper and lower bounds on $a$ and $U$ for every $r\ge\rfour$. They become highly accurate for large~$r$, say $r\ge16$. 

For the narrower enclosure used later, a Newton iteration argument first locates  the
zero $\upr_{20}$ of the auxiliary function $F(\upr)=U(20;\upr)-1$ to high precision.  This  zero is not the true
shooting parameter.  The vortex tail of Lemma~\ref{lem_1-U2} and the lower bound for
$\partial_\upr U(20;\upr)$ imply that
$|\uprz-\upr_{20}|<6.32\cdot10^{-17}$, and hence the enclosure in
Lemma~\ref{lem:Newton}.  All interval
integrations use  $I_{\rm cert}$ from that lemma.

The remaining finite calculations and their purpose in the spectral proofs are
listed in the two tables.  The second column contains what is established
by interval arithmetic, and the third lists the analytic results.

\begin{table}[htbp]
\centering
\footnotesize
\renewcommand{\arraystretch}{1.3}
\begin{tabular}{|>{\raggedright\arraybackslash}p{0.19\textwidth}|>{\raggedright\arraybackslash}p{0.35\textwidth}|>{\raggedright\arraybackslash}p{0.36\textwidth}|}
\hline
Part of the proof & Numerical work & Analytic conclusion \\
\hline
Spectrum of $\mathcal L_1$ &
$V_1(r)-1>0.0548$ on $[\sqrt6,\rfour]$. &
The estimates on $(0,\sqrt6)$ and $(\rfour,\infty)$ are analytic, using
Corollary~\ref{cor:vortex_tail_bridge} for the tail.  Lemma~\ref{lem_appV1g1}
excludes both discrete spectrum and a threshold
resonance. \\
\hline
Eigenvalue count for $\mathcal L_2$ &
The vortex is enclosed on $[0.1,15]$. These enclosures are combined with rigorous interval upper bounds for the
integrals of $|\log(r/s)|$ over each pair of subintervals; diagonal cells are handled by an exact closed formula.
The elementary estimate near $r=0$ completes the bound
$3.17965$ in \eqref{mainintegral_LT}. &
The two moments over compact intervals and the vortex tail give a complement smaller than
$0.002642$.  Hence $\Lambda<3.182292<3.2$, and
Proposition~\ref{theoLT2} gives at most one eigenvalue using Set\^{o}'s bound.  \\
\hline
Threshold of $\mathcal L_2$ &
The subordinate solution near~$0$ started from its Frobenius enclosure at $0.1$ and the subordinate solution near~$\infty$ 
enclosed by the Volterra equation at $16$ are propagated to $r_{\rm m}=10$,
where the Wronskian satisfies \eqref{noresonance_wronskian}. &
The Wronskian does not vanish, so Lemma~\ref{Lemspec0} excludes a threshold
resonance. \\
\hline
Internal eigenvalue &
The analogous Wronskian of the subordinate solutions at $0/\infty$ has opposite signs at the endpoints of
$\Lambda_{\rm FGR}=[0.777471875,0.77747375]$. The solutions are computed
from $0.1$ forward, and $16$ backward, respectively, and are matched at $10$. &
Lemma~\ref{lem:fgr_eigenvalue_box} gives an eigenvalue in
$\Lambda_{\rm FGR}$, and Proposition~\ref{theoLT2} gives uniqueness. \\
\hline
\end{tabular}
\vspace{4pt}
\caption{The spectral calculations on compact intervals and their analytic completion.}
\end{table}

To approximate the eigenvalue~$\lambda^2$, we 
start from the less accurate interval $[0.77747,0.77753]$. Then  five certified
midpoint evaluations successively shrink the sign-changing bracket. At the
sixth midpoint the interval enclosure computed for the Wronskian at that point contains zero, making it inconclusive. Thus, the code keeps
the last certified bracket and then verifies opposite signs directly at the
two exact endpoints of
$\Lambda_{\rm FGR}=[0.777471875,0.77747375]$. This interval is used
uniformly in all later estimates.

\subsection{Certified numerical proof of Proposition~\ref{prop:FGR1}}

This subsection summarizes how the certified interval arithmetic output implies the
lower bounds in \eqref{eq:Dhat_lower_bounds}, and hence $D_1,D_2,D_{12}<0$.
The definitions and normalization appear in Section~\ref{sec:FGR}, while the reduction to the compact interval,
origin, and tail computations appears in Section~\ref{sec_NumVer_FGR}.
All estimates involving the spectral parameter are uniform for
$\mu\in\Lambda_{\rm FGR}$ and are then applied at $\mu=\lambda^2$.
Interval quadrature on $[0.1,16]$ gives the enclosures in \eqref{numerics_fgr1}.
Denoting their lower endpoints by
\[
 L_1=0.0176262802754,
 \qquad
 L_2=0.0477905323094,
 \qquad
 L_{12}=0.0233179905311,
\]
We next explain how we account for the remaining intervals $(0,0.1)$ and $(16,\infty)$.
Set
\[
 c(\mu):=\frac{\sqrt{4\mu-1}}{64\mu^2}.
\]
For $\alpha=1,2,12$, let $(i,j)=(1,1),(2,2),(1,2)$, respectively, and let
$\mathcal A_{ij}(J)$ be the vector integral over $J$ introduced in the proof of
Proposition~\ref{prop:FGR1}. On $[0.1,16]$ the code computes
\[
 c(\lambda^2)\bigl\|\mathcal A_{ij}([0.1,16])\bigr\|_{\bbC^2}^2\ge L_\alpha.
\]
The origin estimate in Lemma~\ref{lem:capd_fgr_origin} and the tail estimate in
Lemma~\ref{lem:capd_fgr_tail_upper_bound} imply that
\[
\begin{aligned}
 c(\lambda^2)\bigl\|\mathcal A_{ij}((0,0.1))\bigr\|_{\bbC^2}^2
 &<B_{\rm origin}
   =1.5\cdot10^{-6},\\
 c(\lambda^2)\bigl\|\mathcal A_{ij}((16,\infty))\bigr\|_{\bbC^2}^2
 &<B_{\rm tail}
   =3.174707028\cdot10^{-10}.
\end{aligned}
\]
Since the integral over $(0,\infty)$ is the sum of these three vectors, the
triangle inequality gives \eqref{eq:Dhat_triangle_ineq}.
\begin{table}[htbp]
\centering
\footnotesize
\renewcommand{\arraystretch}{1.3}
\begin{tabular}{|>{\raggedright\arraybackslash}p{0.19\textwidth}|>{\raggedright\arraybackslash}p{0.35\textwidth}|>{\raggedright\arraybackslash}p{0.36\textwidth}|}
\hline
Part of the proof & Numerical work & Analytic conclusion \\
\hline
Internal-mode tail &
At $\reight=8.001$, uniformly for $\mu\in\Lambda_{\rm FGR}$, the three margins
in \eqref{eq:psi_K0_endpoint_margins} have the positive lower bounds
$5.5262\cdot10^{-4}$, $2.8955\cdot10^{-4}$, and
$3.8992\cdot10^{-7}$, respectively. &
Lemma~\ref{lem:psi_K0} gives the comparison with $1.4K_0(\kappa_\lambda r)$ for
$r\ge\reight$, and Lemma~\ref{lem:capd_psi_prime_tail} gives the derivative
bound used for $r\ge16$. \\
\hline
Solutions normalized at infinity &
At $R=16$, the degree--$10$ outgoing expansions satisfy the residual and
contraction bounds in Lemma~\ref{lem:capd_outgoing_start}. &
The Volterra equation controls the omitted terms and gives
$\sup_{r\ge16}\max_j|\upsi_j(r)|\le1.253$.  Backward propagation from $16$
to the matching radius $10$ determines $a_1,a_2$. \\
\hline
\end{tabular}
\vspace{4pt}
\caption{The estimates at infinity used in the Fermi Golden Rule calculation.}
\end{table}

The independently generated macOS and Linux transcripts agree.  The lower
endpoints of the certified intervals give
\[
 \widehat D_1^{(0)}>0.017297889,\qquad
 \widehat D_2^{(0)}>0.047248801,\qquad
 \widehat D_{12}^{(0)}>0.022940050.
\]
Finally, the physical coefficients in \eqref{FGRODE14} are recovered from
$\widehat D_\alpha^{(0)}$ by Remark~\ref{rem:fgr_code_normalization}:
\[
 D_\alpha=-\frac{4\lambda}{k_\lambda\nu_\psi^4}\,\widehat D_\alpha^{(0)},
 \qquad \alpha\in\{1,2,12\}.
\]
Since
\[
 \frac{4\lambda}{k_\lambda\nu_\psi^4}>0,
\]
we conclude that
$D_1,D_2,D_{12}<0$, proving Proposition~\ref{prop:FGR1}.

The interval arithmetic calculations thus exactly give the finite assertions used in
the preceding sections. The Frobenius estimates, vortex barriers, comparison
arguments, and Volterra equations provide the remaining analytic steps.

\begin{refcontext}[sorting=nyt]
  \printbibliography
\end{refcontext}

\end{document}